\documentclass[11pt]{article}

\usepackage[tbtags]{amsmath}
\usepackage{amssymb}
\usepackage{amsthm}
\usepackage[misc]{ifsym}
\usepackage{cases}
\usepackage{mathrsfs}
\usepackage{bm}
\usepackage{bbm}
\usepackage{color}
\usepackage{amsfonts}
\usepackage{stmaryrd}
\usepackage[active]{srcltx}
\usepackage{enumerate}
\usepackage{mathtools}
\usepackage{empheq}
\usepackage{wasysym}
\usepackage{verbatim}
\usepackage{graphicx}
\usepackage[
bookmarks=true,         
bookmarksnumbered=true, 
colorlinks=true, pdfstartview=FitV, linkcolor=blue, citecolor=blue,
urlcolor=blue]{hyperref}
\usepackage{microtype}

\newcommand{\esssup}{\mathop{\mathrm{esssup}}}

\numberwithin{equation}{section}
\normalsize

\title{\bf Global Maximum Principle for Partially Observed Risk-Sensitive Progressive Optimal Control of FBSDE with Poisson Jumps \thanks{This work is supported by National Key R\&D Program of China (2022YFA1006104), National Natural Science Foundations of China (12471419, 12271304), and Shandong Provincial Natural Science Foundations (ZR2024ZD35, ZR2022JQ01).}}
\author{\normalsize Jingtao Lin\thanks{\it School of Mathematics, Shandong University, Jinan 250100, P.R. China, E-mail: linjingtao@mail.sdu.edu.cn},\quad Jingtao Shi\thanks{\it Corresponding author. School of Mathematics, Shandong University, Jinan 250100, P.R. China, E-mail: shijingtao@sdu.edu.cn}}

\newtheorem{proposition}{Proposition}[section]
\newtheorem{theorem}{Theorem}[section]
\newtheorem{definition}{Definition}[section]
\newtheorem{lemma}{Lemma}[section]
\newtheorem{remark}{Remark}[section]
\newtheorem{assumption}{Assumption}[]

\begin{document}

\maketitle

\noindent{\bf Abstract:}\quad This paper is concerned with one kind of partially observed progressive optimal control problems of coupled forward-backward stochastic systems driven by both Brownian motion and Poisson random measure with risk-sensitive criteria. The control domain is not necessarily convex, and the control variable can enters into all the coefficients. The observation equation also has correlated noises with the state equation. Under the Poisson jump setting, the original problem is equivalent to a complete information stochastic recursive optimal control problem of a forward-backward system with quadratic-exponential generator. In order to establish the first- and second-order variations, some new techniques are introduced to overcome difficulties caused by the quadratic-exponential feature. A new global stochastic maximum principle is deduced. As an application, a risk-sensitive optimal investment problem with factor model is studied. Moreover, the risk-sensitive stochastic filtering problem is also studied, which involves both Brownian and Poissonian correlated noises. A modified Zakai equation is obtained.

\vspace{2mm}


\vspace{2mm}


\section{Introduction}

The stochastic optimal control problem is one kind of important problems in modern control theory. It is well known that the Pontryagin's maximum principle, namely, the necessary condition for optimality, is an important tool in solving stochastic optimal control problems. When the control variable enters into the diffusion term and the control domain is not convex, the classical first-order expansion is not enough to formulate the variational equation for the state process, since the It\^o's integral $\int_{t}^{t+\epsilon}\sigma(s,x_s,u_s)dW_s$ is only of order $O(\epsilon^{\frac{1}{2}})$. In 1900, Peng \cite{Peng90} proved the global maximum principle for the forward stochastic control system by using a second-order variation equation to overcome this difficulty. Moreover, first- and second-order adjoint equations were introduced as vector- and matrix-valued {\it backward stochastic differential equations} (BSDEs), respectively. General non-linear BSDE theory was established by Pardoux and Peng \cite{PP90}, where the existence and uniqueness of solutions are obtained. Independently, Duffie and Epstein \cite{DE92} introduce the notion of recursive utility, which is a special type of BSDEs. In fact, recursive utility is an extension of the standard additive utility, where the instantaneous utility not only depends on the instantaneous consumption rate but also on the future utility (El Karoui et al. \cite{EPQ97}). Peng \cite{Peng93} first established a local maximum principle for the stochastic recursive optimal control problem, where the state process satisfies a controlled {\it forward-backward stochastic differential equation} (FBSDE). See Wu \cite{Wu98}, Yong \cite{Yong10}, Wu \cite{Wu13}, Hu \cite{Hu17}, Hu et al. \cite{HJX18}, Lin and Shi \cite{LS23} for more developments about the maximum principle of FBSDEs.

Due to their wide applications in finance and economics, many scholars have done research for the optimal control problem of stochastic systems with Poisson jumps. Situ \cite{Situ91} first obtained the maximum principle for a system driven by a controlled {\it stochastic differential equation with Poisson jumps} (SDEP), where the jump term is control independent. Tang and Li \cite{TL94} established the general maximum principle of controlled SDEP and proved existence and uniqueness results for the general {\it BSDE with Poisson jumps} (BSDEP). By introducing a new spike variation technique, Song et al. \cite{STW20} obtained a new global maximum principle for controlled SDEPs, which made up some shortcomings in \cite{TL94}. As extensions, \O ksendal and Sulem \cite{OS09}, Shi and Wu \cite{SW10}, Hao and Meng \cite{HM20}, Wang et al. \cite{WSS2024}, etc., considered systems driven by the controlled {\it FBSDE with Poisson jumps} (FBSDEP).

Originating from the concept of utility, risk sensitivity is widely related to mathematical finance. When taking into account the controller's risk preference, it is natural to consider risk-sensitive optimal control problems, where usually some risk-sensitive parameters/indices are introduced in a cost functional of exponential-of-integral type. Risk-sensitive optimal control problems have close connections with {\it linear-exponential-quadratic Gaussian} (LEQG) problems (Whittle \cite{Whittle81}, Duncan \cite{Duncan13}), and robust $H_\infty$ control problems (Lim and Zhou \cite{LZ01}).

Maximum principle is a useful tool to solve risk-sensitive optimal control problems. Pioneering work for stochastic systems can be seen in Whittle \cite{Whittle90, Whittle91}. Assuming the smoothness of the value function and using the relation between the global stochastic maximum principle of \cite{Peng90} and the dynamic programming principle of Yong and Zhou \cite{YZ99}, Lim and Zhou \cite{LZ05} established a global maximum principle for the risk-sensitive stochastic optimal control problem. By introducing a non-linear transformation from \cite{EH03}, Djehiche et al. \cite{DTT15} eliminated the smoothness assumption of the value function, and obtained a global maximum principle for risk-sensitive optimal control of {\it stochastic differential equations} (SDEs) of mean-field type. These results are extended to various stochastic systems, see Shi and Wu \cite{SW11,SW12}, Ma and Liu \cite{ML16}, Moon \cite{Moon20}, Moon et al. \cite{MDB19}, Lin and Shi \cite{LS24} and references therein.

One way to characterize risk-sensitivity is to use the nonlinear expectation rather than the linear one. That is, for a random variable $\xi$ and a constant $\theta$, considering $\mathcal{E}_\theta[\xi]\coloneqq\frac{1}{\theta}\ln\mathbb{E}[e^{\theta\xi}]$. Performing the second-order Taylor expansion on the nonlinear expectation $\mathcal{E}_\theta[\xi]$ with respect to $\theta$ around $\theta=0$ gives
\begin{equation*}
	\mathcal{E}_{\theta}[\xi]=\mathbb{E}[\xi]+\frac{\theta}{2}\mbox{Var}[\xi]+o(\theta),
\end{equation*}
where Var$[\xi]$ is the variance of $\xi$. Since Var$[\xi]$, the volatility, can be regarded as risk, to minimize $\mathcal{E}_{\theta}[\xi]$, a decision maker processes the risk-averse (resp. risk-seeking) if $\theta>0$ (resp. $\theta<0$).

As a special case of exponential utility, the risk-sensitive stochastic optimal control problem is naturally related to a {\it quadratic BSDE} (QBSDE), that is, a BSDE with quadratic generator. See El Karoui and Hamadene \cite{EH03}, Hu et al. \cite{HIM05}, Morlais \cite{Morlais09}, Hu and Tang \cite{HT16}, Ji and Xu \cite{JX24} for more details. Barrieu and El Karoui \cite{BE13} proved the existence of a solution to QBSDE, without the uniqueness, by introducing a so-called quadratic structure condition. Their result is extended to the exponential utility optimization in a market with counter party default risks by generalizing quadratic structure condition to a {\it quadratic-exponential} (Q$_{exp}$) structure condition in Ngoupeyou \cite{Ngoupeyou10}. El Karoui et al. \cite{EMN16} studied a class of {\it QBSDEs under Q$_{exp}$ structure condition} (Q$_{exp}$BSDEs) with Poisson jumps, which were exactly the ones appearing in utility maximization or indifference pricing problems in a jump setting. More details can be seen in \cite{Morlais10,KPZ15,AM16,FT18,MEK22}.

In reality, state processes usually cannot be observed directly. In many practical situations, it often happens that the state can be only partially observed via other variables, and there could be noise existing in the observation system. Based on pioneering work by Bensoussan \cite{Ben92}, Li and Tang \cite{LT95} and Tang \cite{Tang98} about maximum principles of partially observed controlled stochastic systems, there has been considerable literature about maximum principles of controlled forward-backward stochastic systems. Wang and Wu \cite{WW09-1} obtained a maximum principle for partially observed stochastic recursive optimal control problem. A partially observed optimal control problem with risk-sensitive objective is considered in Wang and Wu \cite{WW09-2}. Wang et al. \cite{WWX13} studied a partially observed stochastic control problem of FBSDEs with correlated noises between the system and the observation. Xiao \cite{Xiao11} studied a maximum principle for partially observed optimal control of forward-backward stochastic systems with Poisson jumps. Zheng and Shi \cite{ZS23} obtained a global maximum principle for partially observed forward-backward stochastic systems with Poisson jumps in progressive structure. Other work can be seen in Wu \cite{Wu10}, Shi and Wu \cite{SW10-1}, Jiang \cite{Jiang23} and a nice monograph by Wang et al. \cite{WWX18}.

Motivated by the above work, in this paper, we consider a partially observed risk-sensitive progressive optimal control problem of coupled forward-backward stochastic system with Poisson jumps. As mentioned above, the risk-sensitive criteria are connected to a {\it Q$_{exp}$BSDE with Poisson jumps} (Q$_{exp}$BSDEP). Hence, the original problem can be transformed to a complete information stochastic recursive optimal control problem with respect to the system driven by a coupled FBSDEP and a Q$_{exp}$BSDEP. Under diffusion setting, Hu et al. \cite{HJX22} studied the global maximum principle for non-coupled FBSDEs with quadratic generator and Buckdahn et al. \cite{BLLW24} extended it to the mean-field type.

Another important approach to solve the partially observed stochastic optimal control problem is to treat it as a fully observed one driven by the so-called Zakai equation. It is a linear {\it stochastic partial differential equation} (SPDE) satisfied by the unnormalized conditional density of the state (for more details, see Bain and Crisan \cite{BC09}). In the jump-diffusion framework, Germ and Gy{\"o}ngy \cite{GG22} recently derived the filtering equations for jump-diffusion systems with correlated noises. More related work can be found in Ceci and Colaneri \cite{CC12} and the references therein. To solve risk-sensitive stochastic optimal control problems, the modified Zakai equation, introduced by Nagai \cite{Nagai01}, is a useful tool when systems are partially observable. In this paper, we derive the corresponding Zakai equations, extending the work of \cite{Nagai01} and \cite{GG22}.

Our work distinguishes itself from the existing literatures in the following aspects:

(1) Compared with \cite{HJX22} and \cite{BLLW24}, a quadratic BSDEP with special exponential structure is involved in our system since Poisson jumps are considered. Under diffusion framework, to explore the global stochastic maximum principle, the quadratic growth leads to the fact that the first- and second-order variational equations are linear BSDEs with unbounded stochastic Lipschitz coefficients involving BMO-martingales, which has been studied by Briand and Confortola \cite{BC08} (see \cite{HJX22} and \cite{BLLW24} for some extensions). When Poisson jumps are involved, linear BSDEPs with stochastic Lipschitz coefficients are studied in Fujii and Takahashi \cite{FT18}. However, the jump coefficient in \cite{FT18} is just bounded Lipschitz.

(2) Compared with \cite{ZS23}, the state equation in this paper is coupled and a Q$_{exp}$BSDEP is needed, which leads to some new techniques to obtain estimates to deduce the variational inequality. More specifically, Lemma 3.2 of \cite{ZS23} shows that $\mathbb{E}\big(\int_0^T|z^{i,\epsilon}_t-\bar{z}^i_t-z^{i,1}_t|^2dt\big)^2=o(\epsilon^2)$ rather than $O(\epsilon^4)$. When it comes to the estimate of $z^{i,\epsilon}_t-\bar{z}^i_t-z^{i,1}_t-z^{i,2}_t$, we need to estimate $\mathbb{E}\big(\int_0^T|\bar{z}^i_t||z^{i,\epsilon}_t-\bar{z}^i_t-z^{i,1}_t|\mathbbm{1}_{[\bar{t},\bar{t}+\epsilon]}(t)dt\big)^p$, which is caused by the quadratic-exponential structure. However, $\mathbb{E}\big(\int_0^T|z^{i,\epsilon}_t-\bar{z}^i_t-z^{i,1}_t|^2dt\big)^2=o(\epsilon^2)$ is not enough to get our desired order (see Lemma \ref{lemma estimate of zeta1 hat zeta2} and Lemma \ref{lemma estimate of zeta2,hat zeta 3} for more details). Extending the method in \cite{BLLW24} to the Poisson jump setting, we introduce a deterministic set $\Gamma_M$, $M\geq1$, and define $E_\epsilon\coloneqq[\bar{t},\bar{t}+\epsilon]\cap\Gamma_M$. We first deduce that, noting that we use the notion $\kappa$ rather than $z$, $\mathbb{E}\big(\int_0^T|\kappa^{i,\epsilon}_t-\bar{\kappa}^i_t-\kappa^{i,1}_t|^2dt\big)^2=O(\epsilon^4)$ in Lemma \ref{lemma estimate of zeta1 hat zeta2} and $\mathbb{E}\big(\int_0^T|\kappa^{i,\epsilon}_t-\bar{\kappa}^i_t-\kappa^{i,1}_t-\kappa^{i,2}_t|^2dt\big)^2=o(\epsilon^4)$ in Lemma \ref{lemma estimate of zeta2,hat zeta 3}. These lead to a global maximum principle which first holds true for $t\in E_\epsilon$. Then, the general result follows from $M\to\infty$.

(3) We consider a risk-sensitive optimal investment model in which the goal is to maximize the exponential utility of wealth. In this model, the mean return of the stock is explicitly affected by the underlying economic factor. The result extends those in Fleming and Sheu \cite{FS00}, Nagai \cite{Nagai01}, Davis and Lleo \cite{DL13}.

(4) The modified Zakai equation, which is satisfied by the unnormalized conditional density of the jump-diffusion system is presented. Both Brownian and Poissonian correlated noises are involved in our setting, which is generalization of Nagai \cite{Nagai01}, Germ and Gy{\"o}ngy \cite{GG22}.

The rest of this paper is organized as follows. In Section 2, we formulate a partially observed risk-sensitive progressive optimal control problem of forward-backward stochastic system with Poisson jumps, which is transformed to a stochastic recursive optimal control problem of controlled FBSDEP and quadratic-exponential BSDEP. $L^p$-estimate of fully coupled FBSDEPs and well-posedness of quadratic-exponential BSDEPs are studied. In Section 3, we derive the global maximum principle of the equivalent stochastic recursive optimal control problem. As an application, in Section 4, we consider a risk-sensitive optimal investment model in which the goal is to maximize the exponential utility of wealth. The risk-sensitive filtering Zakai equation under the jump-diffusion framework is presented in Section 5. Section 6 gives some concluding remarks.

\section{Problem formulation and preliminaries}\label{section Problem formulation and preliminaries}

Let $T>0$ be fixed. Consider a complete filtered probability space $(\Omega,\mathcal{F},(\mathcal{F}_t)_{0\leqslant t \leqslant T},\bar{\mathbb{P}})$ and two one-dimensional independent standard Brownian motions $W^1$ and $\tilde{W}^2$ defined in $\mathbb{R}^2$ with $W^1_0=\tilde{W}^2_0=0$. Let $(E,\mathcal{B}(E))$ be a Polish space with the $\sigma$-finite measure $\nu_1$ on $E_1$, $\nu_2$ on $E_2$ and $E_1\subset E,E_2\subset E$. Suppose that $N_1(de,dt)$ is a Poisson random measure on $(\mathbb{R}^{+}\times E_1,\mathcal{B}(\mathbb{R}^{+})\times\mathcal{B}(E_1))$ under $\bar{\mathbb{P}}$ and for any $A_1\in\mathcal{B}(E_1)$, $\nu_1(A_1)<\infty$, the compensated Poisson random measure is given by $\tilde{N}_1(de,dt)=N_1(de,dt)-\nu_1(de)dt$. Let $N_2(de,dt)$ be an integer-valued random measure and its predictable compensator is given by $\lambda(t,x_{t-},e)\nu_2(de)dt$, where the function $\lambda(t,x,e)\in[l,1),0< l<1$, and for any $A_2\in\mathcal{B}(E_2)$, $\nu_2(A_2)<\infty$. The compensated random measure is given by $\tilde{N}^{\prime}_2(de,dt)=N_2(de,dt)-\lambda(t,x_{t-},e)\nu_2(de)dt$. Moreover, $W^1,\tilde{W}^2,N_1,N_2$ are mutually independent under $\bar{\mathbb{P}}$, and let $\mathcal{F}_t^{W^1},\mathcal{F}_t^{\tilde{W}^2},\mathcal{F}_t^{N_1},\mathcal{F}_t^{N_2}$ be the $\bar{\mathbb{P}}$-completed natural filtrations generated by $W^1,\tilde{W}^2,N_1,N_2$, respectively. Set $\mathcal{F}_t\coloneqq\mathcal{F}_t^{W^1}\vee\mathcal{F}_t^{\tilde{W}^2}\vee\mathcal{F}_t^{N_1}\vee\mathcal{F}_t^{N_2}\vee\mathcal{N}$ and $\mathbb{F}:=\{\mathcal{F}_t\}_{0\leqslant t\leqslant T}$, where $\mathcal{N}$ denotes all $\bar{\mathbb{P}}$-null sets. $\bar{\mathbb{E}}$ denotes the expectation under $\bar{\mathbb{P}}$.

Different from \cite{TL94,TH02} but similar as \cite{STW20}, the integrand of the stochastic integral in our paper is $E$-progressive measurable instead of $E$-predictable. We first introduce some preliminaries.

\begin{definition}
	Suppose that $\mathcal{H}$ is an Euclidean space, and $\mathcal{B}(\mathcal{H})$ is the Borel $\sigma$-field on $\mathcal{H}$. Given $T>0$, a process $x:[0,T]\times\Omega\rightarrow\mathcal{H}$ is called \emph{progressive measurable (predictable)} if $x$ is $\mathcal{G}/\mathcal{B}(\mathcal{H})(\mathcal{P}/\mathcal{B}(\mathcal{H}))$ measurable, where $\mathcal{G}(\mathcal{P})$ is the corresponding progressive measurable (predictable) $\sigma$-field on $[0,T]\times\Omega$, and a process $x:[0,T]\times\Omega\times E\rightarrow\mathcal{H}$ is called \emph{$E$-progressive measurable ($E$-predictable)} if $x$ is $\mathcal{G}\times\mathcal{B}(E)/\mathcal{B}(\mathcal{H})(\mathcal{P}\times\mathcal{B}(E)/\mathcal{B}(\mathcal{H}))$ measurable.
\end{definition}

Now, given a process $x$ which has RCLL paths, $x_{0-}\coloneqq0$ and $\Delta x_t\coloneqq x_t-x_{t-},t\geqslant0$, for $i=1,2$, let $m_i$ denote the measure on $\mathcal{F}\otimes\mathcal{B}([0,T])\otimes\mathcal{B}(E_i)$ generated by $N_i$ that $m_i(A)=\bar{\mathbb{E}}\int_0^T\int_{E_i}\mathbbm{1}_{A}N_i(de,dt)$. For any $\mathcal{F}\otimes\mathcal{B}([0,T])\otimes\mathcal{B}(E_i)/\mathcal{B}(\mathbb{R})$ measurable integrable process $x$, we set $\mathbb{E}_i[x]\coloneqq\int xdm_i$ and denote by $\mathbf{E}_i[x|\mathcal{P}\otimes\mathcal{B}(E_i)]$ the Radon-Nikodym derivatives w.r.t. $\mathcal{P}\otimes\mathcal{B}(E_i)$. In fact, $\mathbf{E}_i$ is not an expectation (for $m_i$ is not a probability measure), but it owns similar properties to expectation. A more general definition of stochastic integral of random measure has been introduced by \cite{STW20} where the theory of dual predictable projection is utilized. Therefore, we omit the details here and give the following lemma directly.

\begin{lemma}
	If $g$ is a positive $E_i$-progressive measurable process that $\bar{\mathbb{E}}\int_0^T\int_{E_i}gN_i(de,dt)<\infty,i=1,2$, then we have the following results:
	\begin{equation*}\begin{aligned}
			&(i)\qquad \Big(\int_0^{\cdot}\int_{E_i}gN_i(de,dt)\Big)_t^p=\int_0^t\int_{E_i}\mathbf{E}_i\big[g|\mathcal{P}\otimes\mathcal{B}(E_i)\big]\nu_i(de)dt,\\
			&(ii)\qquad \int_0^T\int_{E_i}g\tilde{N}_i(de,dt)=\int_0^T\int_{E_i}gN_i(de,dt)-\Big(\int_0^{\cdot}\int_{E_i}gN_i(de,dt)\Big)_T^p,\\
			&(iii)\qquad \int_0^T\int_{E_i}g\tilde{N}_i(de,dt)=\int_0^T\int_{E_i}gN_i(de,dt)-\int_0^T\int_{E_i}\mathbf{E}_i\big[g|\mathcal{P}\otimes\mathcal{B}(E_i)\big]\nu_i(de)dt,\\
			&(iv)\qquad \bar{\mathbb{E}}\int_0^T\int_{E_i}gN_i(de,dt)=\bar{\mathbb{E}}\int_0^T\int_{E_i}\mathbf{E}_i\big[g|\mathcal{P}\otimes\mathcal{B}(E_i)\big]\nu_i(de)dt,\\
			&(v)\qquad \Delta(g_{\cdot}\tilde{N}_i)_t=\int_{E_i}gN_i(de,\{t\}),\\
			&(vi)\qquad \big[g_{\cdot}\tilde{N}_{i,t},g_{\cdot}\tilde{N}_{i,t}\big]=\int_0^t\int_{E_i}g^2N_i(de,dt),
	\end{aligned}\end{equation*}
	where $x^p$ is the dual predictable projection of $x$.
\end{lemma}

Consider the following state equation which is a controlled coupled FBSDEP:
\begin{equation}\label{state equation x y}
	\left\{
	\begin{aligned}
		dx_t&=b_1\Big(t,x_t,y_t,z^1_t,z^2_t,\int_{E_1}\tilde{z}^1_{(t,e)}\nu_1(de),\int_{E_2}\tilde{z}^2_{(t,e)}\nu_2(de),u_t\Big)dt\\
		&\quad+\sigma_1(t,x_t,y_t,u_t)dW^1_t+\sigma_2(t,x_t,y_t,u_t)d\tilde{W}^2_t\\
		&\quad+\int_{E_1}f_1(t,x_{t-},y_{t-},u_t,e)\tilde{N}_1(de,dt)+\int_{E_2}f_2(t,x_{t-},y_{t-},u_t,e)\tilde{N}'_2(de,dt),\\
		-dy_t&=g\Big(t,x_t,y_t,z^1_t,z^2_t,\int_{E_1}\tilde{z}^1_{(t,e)}\nu_1(de),\int_{E_2}\tilde{z}^2_{(t,e)}\nu_2(de),u_t\Big)dt\\
		&\quad-z^1_tdW^1_t-z^2_tdW^2_t-\int_{E_1}\tilde{z}^1_{(t,e)}\tilde{N}_1(de,dt)-\int_{E_2}\tilde{z}^2_{(t,e)}\tilde{N}_2(de,dt),\quad t\in[0,T],\\
		x_0&=x_0,\quad y_T=\phi(x_T).
	\end{aligned}
	\right.
\end{equation}
Moreover, we have explicit representations of $W^2,\tilde{N}_2$ by $\tilde{W}^2,\tilde{N}^\prime_2$ as follows:
\begin{equation}\label{relation W N}
	\begin{aligned}
		dW^2_t&=d\tilde{W}^2_t+\sigma_3^{-1}(t)b_2(t,\Theta(t),u_t)dt,\\
		\tilde{N}_2(de,dt)&=\tilde{N}^\prime_2(de,dt)+(\lambda(t,x_{t-},e)-1)\nu_2(de)dt.
	\end{aligned}
\end{equation}
where $\sigma_3(t)$ and $b_2(t,\Theta(t),u_t)$ are defined in the following.

Suppose that the state process $(x,y,z^1,z^2,\tilde{z}^1,\tilde{z}^2)$ can only be observed through a related process $Y$ which is governed by the following SDEP:
\begin{equation}\label{observation}
	\left\{
	\begin{aligned}
		dY_t&=b_2\Big(t,x_t,y_t,z^1_t,z^2_t,\int_{E_1}\tilde{z}^1_{(t,e)}\nu_1(de),\int_{E_2}\tilde{z}^2_{(t,e)}\nu_2(de),u_t\Big)dt\\
		&\quad +\sigma_3(t)d\tilde{W}^2_t+\int_{E_2}f_3(t,e)\tilde{N}^{\prime}_2(de,dt),\quad t\in[0,T],\\
		Y_0&=0.
	\end{aligned}
	\right.
\end{equation}
In this paper, we consider the risk-averse cost functional. That is, for $\theta>0$,
\begin{equation}\label{cf1} J'(u)=\bar{\mathbb{E}}\bigg[\exp\biggl\{\theta\int_0^T\tilde{l}\Big(t,x_t,y_t,z^1_t,z^2_t,\int_{E_1}\tilde{z}^1_{(t,e)}\nu_1(de),\int_{E_2}\tilde{z}^2_{(t,e)}\nu_2(de),u_t\Big)dt+\theta\varphi(x_T,y_0)\biggr\}\bigg],
\end{equation}
where we suppose that $(\sigma_1,\sigma_2)(t,x,y,u):[0,T]\times \mathbb{R}\times \mathbb{R}\times \mathbb{R}\rightarrow \mathbb{R}$, $(f_1,f_2)(t,x,y,u,e):[0,T]\times \mathbb{R}\times \mathbb{R}\times \mathbb{R}\times E_1(E_2)\rightarrow \mathbb{R}$, and $(b_1,g,b_2,\tilde{l})(t,x,y,z^1,z^2,\tilde{z}^1,\tilde{z}^2,u):[0,T]\times \mathbb{R}\times \mathbb{R}\times \mathbb{R}\times \mathbb{R}\times \mathbb{R}\times \mathbb{R}\times \mathbb{R}\rightarrow \mathbb{R}$, $\phi(x): \mathbb{R}\rightarrow \mathbb{R}$, $\sigma_3(t):[0,T]\rightarrow \mathbb{R}$, $f_3(t,e):[0,T]\times E_2\rightarrow \mathbb{R}$, $\varphi(x,y):\mathbb{R}\times\mathbb{R}\rightarrow \mathbb{R}$ are suitable maps.

\begin{remark}
	\cite{ZS23} presented the $L^p(p>2)$ theory of fully coupled FBSDEPs. When it comes to the stochastic optimal control problem, the controlled forward-backward system they consider was not coupled. In our work, we study a coupled controlled forward-backward system. However, it is worth mentioning that diffusion terms $\sigma_i(\cdot)$ and jump terms $f_i(\cdot)$ are still independent of $z^i$ and $\tilde{z}^i$. In fact, the general case remains a challenging problem.
\end{remark}

The following assumptions are exerted.
\begin{assumption}\label{assumption A1 in Zheng-Shi}
	\begin{enumerate}[\bfseries (1)]
		\item $\sigma_1,\sigma_2$ are $\mathcal{G}\otimes\mathcal{B}(\mathbb{R})\otimes\mathcal{B}(\mathbb{R})\otimes\mathcal{B}(\mathbb{R})/\mathcal{B}(\mathbb{R})$ measurable, $f_1,f_2$
              are $\mathcal{G}\otimes\mathcal{B}(\mathbb{R})\otimes\mathcal{B}(\mathbb{R})\otimes\mathcal{B}(\mathbb{R})\otimes\mathcal{B}(E_1)(\mathcal{B}(E_2))/\mathcal{B}(\mathbb{R})$ measurable, $\phi$ is $\mathcal{B}(\mathbb{R})/\mathcal{B}(\mathbb{R})$ measurable.
		\item $\sigma_1,\sigma_2,f_1$ and $f_2$ are twice continuously differentiable in $x$ with bounded first and second order derivatives, $\sigma_2$ is also bounded and there is a constant $C$ such that
		      \begin{equation*}
				\big|(\sigma_1,\sigma_2,f_1,f_2)(t,x,u)\big|\leqslant C(1+|x|+|y|+|u|).
		      \end{equation*}
	    \item $\phi$ are twice continuous differentiable in $x$ with bounded second order derivatives and there is a constant $C$ such that
	          \begin{equation*}
		        |\phi(x)|\leqslant C(1+|x|).
	          \end{equation*}
	    \item For $\beta\geqslant2$, for $i=1,2$, the following hold
	          \begin{equation*}
		      \begin{aligned}
			    &\Big(\int_0^T|b_i(t,0,0,0,0,0,0,0)|dt\Big)^\beta<\infty,\quad \Big(\int_0^T|\sigma_i(t,0,0,0)|^2dt\Big)^\frac{\beta}{2}<\infty,\\
			    &\Big(\int_0^T\int_{E_i}|f_i(t,0,0,0,e)|^2\nu_i(de)dt\Big)^\frac{\beta}{2}<\infty,\quad\Big(\int_0^T|g(t,0,0,0,0,0,0,0)|dt\Big)^\beta<\infty.
		      \end{aligned}
	          \end{equation*}
	    \item $b_i,i=1,2$, $g$ are $\mathcal{G}\otimes\mathcal{B}(\mathbb{R})\otimes\mathcal{B}(\mathbb{R})\otimes\mathcal{B}(\mathbb{R})\otimes\mathcal{B}(\mathbb{R})\otimes\mathcal{B}
	          (\mathbb{R})\otimes\mathcal{B}(\mathbb{R})\otimes\mathcal{B}(\mathbb{R})/
	          \mathcal{B}(\mathbb{R})$ measurable, and are twice continuously differentiable with respect to $(x,y,z^1,z^2,\tilde{z}^1,\tilde{z}^2)$; $b_i,g,Db_i,Dg$, $D^2b_i,D^2g$ are continuous
               in $(x,y,z^1,z^2,\tilde{z}^1,\tilde{z}^2,u)$; $b_i,Db_i,Dg,D^2b_i,D^2g$ are bounded, and $|\sigma_3^{-1}(t)|$ is bounded by some constant $C$, and
	          \begin{equation*}
		      \begin{aligned}
			    &\sum_{i=1}^2\big|b_i(t,x,,y,z^1,z^2,\tilde{z}^1,\tilde{z}^2,u)\big|+\big|g(t,x,,y,z^1,z^2,\tilde{z}^1,\tilde{z}^2,u)\big|\\
			    &\ \leqslant C(1+|x|+|y|+|z^1|+|z^2|+\Vert\tilde{z}^1\Vert+\Vert\tilde{z}^2\Vert+|u|).
		      \end{aligned}
	          \end{equation*}
	    \item For any $(t,x,e)\in[0,T]\times \mathbb{R}\times E_2$, there exists a constant $C$ such that $|\lambda_x(t,x,e)|+|\lambda_{xx}(t,x,e)|\leqslant C$.	
	\end{enumerate}
\end{assumption}

\begin{assumption}\label{assumption A2 in Zheng-Shi}
	$\tilde{l},\varphi$ are twice continuous differentiable in $(x,y,z^1,z^2,\tilde{z}^1,\tilde{z}^2)$ with bounded second order derivatives and there is a constant $C$ such that
		\begin{equation*}
			\begin{aligned}
				&\big|\tilde{l}_{(x,y,z^1,z^2,\tilde{z}^1,\tilde{z}^2)}(t,x,y,z^1,z^2,\tilde{z}^1,\tilde{z}^2,u)\big|\leqslant C\big(1+|x|+|y|+|z^1|+|z^2|+\Vert\tilde{z}^1\Vert+\Vert\tilde{z}^2\Vert+|u|\big),\\
				&\big|\tilde{l}(t,x,y,z^1,z^2,\tilde{z}^1,\tilde{z}^2,u)\big|\leqslant C\big(1+|x|^2+|y|^2+|z^1|^2+|z^2|^2+\Vert\tilde{z}^1\Vert^2+\Vert\tilde{z}^2\Vert^2+|u|^2\big),\\
				&|\varphi(x,y)|\leqslant C(1+|x|^2+|y|^2),\quad |D\varphi(x,y)|\leqslant C(1+|x|+|y|).
			\end{aligned}
		\end{equation*}
\end{assumption}

Throughout this paper, for notational simplicity, we define $\Theta(t)\coloneqq(x_t,y_t,z^1_t,z^2_t,\tilde{z}^1_{(t,e)},\tilde{z}^2_{(t,e)})$ and
\begin{equation*}
		\tilde{g}(t,\Theta(t),u_t)\coloneqq\tilde{g}\Big(t,x_t,y_t,z^1_t,z^2_t,\int_{E_1}\tilde{z}^1_{(t,e)}\nu_1(de),\int_{E_2}\tilde{z}^2_{(t,e)}\nu_2(de),u_t\Big),
\end{equation*}
for $\tilde{g}=g,b_i,l$ and $(t,x,y,z^1,z^2,\tilde{z}^1,\tilde{z}^2)\in[0,T]\times \mathbb{R}\times \mathbb{R}\times \mathbb{R}\times \mathbb{R}\times L^2(E_1,\mathcal{B}(E_1),\nu_1;\mathbb{R})\times L^2(E_2,\mathcal{B}(E_2),\nu_2;\mathbb{R})$. Here, the space $L^2(E_i,\mathcal{B}(E_i),\nu_i;\mathbb{R})$, $i=1,2$ are defined in the Assumption \ref{assumption A4 in Zheng-Shi}, in the following.

For $t\in[0,T]$, define ${\mathcal{F}}_t^Y\coloneqq\sigma\{Y_s;0\leqslant s\leqslant t\}$. For $U\subseteq\mathbb{R}$, caused by the partially observed feature, the admissible control set is defined as follows:
\begin{equation}\label{ad_control_set}
	\begin{aligned}
		\mathcal{U}_{ad}[0,T]&\coloneqq\bigg\{u\Big|u_t\mbox{ is }\mathcal{F}_t^Y\mbox{-progressive }U\mbox{-valued process, such that }\sup_{0\leqslant t\leqslant T}\bar{\mathbb{E}}|u_t|^p<\infty,\\
		&\qquad \mbox{ for any $p>1$ }\mbox{and }\bar{\mathbb{E}}\int_0^T|u_t|^2N_i(E_i,dt)<\infty,\mbox{\ for\ } i=1,2\bigg\}.
	\end{aligned}
\end{equation}
Our main goal in this paper is to select an optimal control $\bar{u}\in{\mathcal{U}_{ad}[0,T]}$ such that
\begin{equation*}
	J'(\bar{u})=\inf_{u\in\,\mathcal{U}_{ad}[0,T]}J'(u).
\end{equation*}

Similar as \cite{ZS23}, in order to solve the problem, we first set
\begin{equation}\label{RN}
	\begin{aligned}
		\Gamma_T\coloneqq&\exp\bigg\{-\int_0^T\sigma_3^{-1}(t)b_2\big(t,\Theta(t),u_t\big)d\tilde{W}^2_t-\frac{1}{2}\int_0^T|\sigma_3^{-1}(t)b_2\big(t,\Theta(t),u_t\big)|^2dt\\
		&\qquad -\int_0^T\int_{E_2}\ln{\lambda(t,x_{t-},e)}N_2(de,dt)-\int_0^T\int_{E_2}(1-\lambda(t,x_{t-},e))\nu_2(de)dt\bigg\}.
	\end{aligned}
\end{equation}
The following assumption is necessary to guarantee the Girsanov measure transformation.
\begin{assumption}\label{assumption A3 in Zheng-Shi}
	$\quad\bar{\mathbb{E}}\Big[\exp\big\{\int_0^T\int_{E_2}\frac{(1-\lambda(t,x_{t-},e))^2}{\lambda(t,x_{t-},e)}\nu_2(de)dt\big\}\Big]<\infty$.
\end{assumption}

Under Assumption \ref{assumption A3 in Zheng-Shi}, define a locally square-integrable martingale $M$ through
\begin{equation*}
		M(t)\coloneqq-\int_0^t\sigma_3^{-1}(s)b_2\big(s,\Theta(s),u_s\big)d\tilde{W}^2_s+\int_0^t\int_{E_2}\frac{1-\lambda(s,x_{s-},e)}{\lambda(s,x_{s-},e)}\tilde{N}^{\prime}_2(de,ds).
\end{equation*}
Moreover, $M_t-M_{t-}>-1,\ \bar{\mathbb{P}}\mbox{-}a.s.$, and
\begin{equation*}
	\begin{aligned}
		&\bar{\mathbb{E}}\bigg[\exp\Big\{\frac{1}{2}\langle M^c,M^c\rangle_T+\langle M^d,M^d\rangle_T\Big\}\bigg]\\
		&=\bar{\mathbb{E}}\bigg[\exp\Big\{\frac{1}{2}\int_0^T\big|\sigma_3^{-1}(t)b_2\big(t,\Theta(t),u_t\big)\big|^2dt+\int_0^T\int_{E_2}\frac{(1-\lambda(t,x_{t},e))^2}{\lambda(t,x_{t},e)}\nu_2(de)dt\Big\}\bigg]<\infty,
	\end{aligned}
\end{equation*}
where $M^c$ and $M^d$ are continuous and purely discontinuous martingale parts of $M$, respectively. Therefore, it follows that $\Gamma_t$, the Dol\'{e}ans-Dade exponential of $M$, is a martingale (Protter and Shimbo \cite{PS08}). Then we can define a probability measure $\mathbb{P}$ via
\begin{equation*}
	\frac{d\bar{\mathbb{P}}}{d\mathbb{P}}\coloneqq\tilde{\Gamma}_T\equiv(\Gamma_T)^{-1},
\end{equation*}
where $\tilde{\Gamma}$ satisfies the following equation:
\begin{equation}\label{RN2}
	\left\{
	\begin{aligned}
		d\tilde{\Gamma}_t&=\tilde{\Gamma}_t\sigma_3^{-1}(t)b_2\big(t,\Theta(t),u_t\big)dW^2_t+\int_{E_2}\tilde{\Gamma}_{t-}(\lambda(t,x_{t-},e)-1)\tilde{N}_2(de,dt),\quad t\in[0,T],\\
		\tilde{\Gamma}_0&=1.
	\end{aligned}
	\right.
\end{equation}
Moreover, under the new probability measure $\mathbb{P}$ and denoting the expectation with respect to $\mathbb{P}$ by $\mathbb{E}$, $W^1,W^2$ are mutually independent Brownian motions and $\tilde{N}_1, \tilde{N}_2$ are mutually independent Poisson martingale measures, where
\begin{equation}\label{relation1}
		dW^2_t=d\tilde{W}^2_t+\sigma_3^{-1}(t)b_2(t,\Theta(t),u_t)dt,\quad \tilde{N}_2(de,dt)=N_2(de,dt)-\nu_2(de)dt.
\end{equation}
Observing $\tilde{N}_2$ and $\tilde{N}_2^{\prime}$, then their relationship can be built as \eqref{relation W N}.

\subsection{Equivalent stochastic recursive optimal control problem}

Under the new probabilily $\mathbb{P}$, the cost functional can be rewritten as
\begin{equation}
	J'(u)=\mathbb{E}\Big[\rho_Te^{\theta\varphi(x_T,y_0)}\Big],
\end{equation}
where
\begin{equation*}
	\begin{aligned}
		\rho_t&\equiv\exp\biggl\{\theta\bigg[\int_0^T\Big\{\tilde{l}(s)-\frac{1}{2\theta}|\sigma^{-1}_3(s)b_2(s)|^2+\frac{1}{\theta}\int_{E_2}\big(\ln\lambda(s)+1-\lambda(s)\big)\nu_2(de)\Big\}ds\\
		&\qquad\quad +\frac{1}{\theta}\int_0^T\sigma^{-1}_3(s)b_2(s)dW^2_s+\frac{1}{\theta}\int_0^T\int_{E_2}\ln\lambda(s)\tilde{N}_2(de,ds)\bigg]\biggr\}.
	\end{aligned}
\end{equation*}
Inspired by \cite{EH03}, we introduce a process $\zeta$, which satisfies the following BSDEP:
\begin{equation}
	\left\{
	\begin{aligned} -d\zeta_t&=\biggl\{\tilde{l}(t,\Theta(t),u_t)+\frac{\theta}{2}\sum_{i=1}^2(\kappa^i_t)^2+\frac{1}{\theta}\sum_{i=1}^2\int_{E_i}\Big[e^{\theta\tilde{\kappa}^i_{(t,e)}}-\theta\tilde{\kappa}^i_{(t,e)}-1\Big]\nu_i(de)\\
         &\qquad+\sigma^{-1}_3(t)b_2(t)\kappa^2_t+\int_{E_2}(\lambda(t)-1)\big(e^{\theta\tilde{\kappa}^2_{(t,e)}}-1\big)\nu_2(de)\biggr\}dt\\
		 &\quad-\sum_{i=1}^2\kappa^i_tdW^i_t-\sum_{i=1}^2\int_{E_i}\tilde{\kappa}^i_{(t,e)}\tilde{N}_i(de,dt),\quad t\in[0,T],\\
		\zeta_T&=\varphi(x_T,y_0).
	\end{aligned}
	\right.
\end{equation}
It follows from It\^o's formula to $\rho_te^{\theta\zeta_t}$ that
\begin{equation*}
	J'(u)=\mathbb{E}\Big[\rho_Te^{\theta\varphi(x_T,y_0)}\Big]=e^{\theta\zeta_0}.
\end{equation*}

Noting that $\theta>0$, our optimal control problem is equivalent to minimize the cost functional $J(u)\coloneqq\zeta_0$ over $\mathcal{U}_{ad}[0,T]$, that is
\begin{equation}\label{cost functional inf}
	\inf_{u\in\,\mathcal{U}_{ad}[0,T]}J(u),
\end{equation}
corresponding to the following state equation, where we have substituted \eqref{relation W N} into \eqref{state equation x y}:
\begin{equation}\label{state equation}
	\left\{
	\begin{aligned}
		dx_t&=b(t,\Theta(t),u_t)dt+\sigma_{i}(t,x_t,y_t,u_t)dW^{i}_t+\int_{E_i}f_i(t,x_t,y_t,u_t,e)\tilde{N}_{i}(de,dt),\\
		-dy_t&=g(t,\Theta(t),u_t)dt-z^i_tdW^i_t-\int_{E_i}\tilde{z}^i_{(t,e)}\tilde{N}_i(de,dt),\\
		-d\zeta_t&=l\Big(t,\Theta(t),\kappa^1_t,\kappa^2_t,\int_{E_1}\tilde{\kappa}^1_{(t,e)}\nu_1(de),\int_{E_2}\tilde{\kappa}^2_{(t,e)}\nu_2(de),u_t\Big)dt\\
        &\quad-\kappa^i_tdW^i_t-\int_{E_i}\tilde{\kappa}^i_{(t,e)}\tilde{N}_i(de,dt),\quad t\in[0,T],\\
		x_0&=x_0,\quad y_T=\phi(x_T),\quad \zeta_T=\varphi(x_T,y_0),
	\end{aligned}
	\right.
\end{equation}
where
\begin{equation}\label{defination of b l in state equation}
	\begin{aligned}
		&b(t,\Theta(t),u_t)\coloneqq b_1(t)-\sigma_2(t)\sigma^{-1}_3(t)b_2(t)-\int_{E_2}(\lambda(t)-1)f_2(t,e)\nu_2(de),\\
		&l\Big(t,\Theta(t),\kappa^1_t,\kappa^2_t,\int_{E_1}\tilde{\kappa}^1_{(t,e)}\nu_1(de),\int_{E_2}\tilde{\kappa}^2_{(t,e)}\nu_2(de),u_t\Big)\\
        &\coloneqq\tilde{l}(t,\Theta(t),u_t)+\frac{\theta}{2}\sum_{i=1}^2(\kappa^i_t)^2+\frac{1}{\theta}\sum_{i=1}^2\int_{E_i}\Big[e^{\theta\tilde{\kappa}^i_{(t,e)}}-\theta\tilde{\kappa}^i_{(t,e)}-1\Big]\nu_i(de)\\
        &\quad+\sigma^{-1}_3(t)b_2(t)\kappa^2_t+\int_{E_2}(\lambda(t)-1)\big(e^{\theta\tilde{\kappa}^2_{(t,e)}}-1\big)\nu_2(de).
	\end{aligned}
\end{equation}

In this paper, we adopt Einstein's notation on summation. That is, we use repeated scripts to stand for the summation over these scripts.

\subsection{$L^p$ solution of FBSDEPs}

Noting the couple feature of \eqref{state equation}, we firstly, consider the following fully coupled FBSDEP of $(x,y,z^i,\tilde{z}^i)$:
\begin{equation}\label{equation x y}
	\left\{
	\begin{aligned}
		dx_t&=b(t,\Theta(t))dt+\sigma_{i}(t,x_t,y_t)dW^{i}_t+\int_{E_i}f_i(t,x_t,y_t,e)\tilde{N}_{i}(de,dt),\\
		-dy_t&=g(t,\Theta(t))dt-z^i_tdW^i_t-\int_{E_i}\tilde{z}^i_{(t,e)}\tilde{N}_i(de,dt),\quad t\in[0,T],\\
		x_0&=x_0,\quad y_T=\phi(x_T),
	\end{aligned}
	\right.
\end{equation}
where $\Theta(t)\coloneqq\big(x_t,y_t,z^1_t,z^2_t,\tilde{z}^1_{(t,e)},\tilde{z}^2_{(t,e)}\big)$ and coefficients $b$, $\sigma_1$, $\sigma_2$, $f_1$, $f_2$, $g$ and $\phi$ could be random.

We introduce the following norms and spaces for any $p\geqslant1$.

$\mathcal{S}^p$ is the space of $\mathbb{R}$-valued c\`adl\`ag and $\mathbb{F}$-adapted processes $X$ such that
$$
\Vert X\Vert^p_{\mathcal{S}^p}\coloneqq\mathbb{E}\bigg[\sup_{0\leqslant t\leqslant T}|X_t|^p\bigg]<\infty.
$$

\noindent$\mathcal{S}^\infty$ is the space of $\mathbb{R}$-valued c\`adl\`ag and $\mathbb{F}$-progressively measurable processes $Y$ such that
$$
\Vert Y\Vert_{\mathcal{S}^\infty}\coloneqq\sup_{0\leqslant t\leqslant T}\Vert Y_t\Vert_\infty<\infty.
$$

\noindent$\mathbb{H}^p$ is the space of $\mathbb{R}$-valued and $\mathbb{F}$-predictable process $Z$ such that
$$
\Vert Z\Vert^p_{\mathbb{H}^p}\coloneqq\mathbb{E}\bigg(\int_0^T|Z_t|^2dt\bigg)^{\frac{p}{2}}<\infty.
$$

\noindent$\mathbb{J}^p_i$ is the space of $\mathbb{R}$-valued and $\mathbb{F}$-predictable process $U$ such that
$$
\Vert U\Vert^p_{\mathbb{J}^p_i}\coloneqq\mathbb{E}\bigg(\int_0^T\int_{E_i}|U_{(s,e)}|^2\nu_i(de)dt\bigg)^{\frac{p}{2}}<\infty,\quad i=1,2.
$$

\noindent For simplicity, we set $\mathcal{N}^p\coloneqq\mathcal{S}^p\times\mathbb{H}^p\times\mathbb{H}^p\times\mathbb{J}^p_1\times\mathbb{J}^p_2$ and $\mathcal{M}^p\coloneqq\mathcal{S}^p\times\mathcal{N}^p$.

\begin{assumption}\label{assumption A4 in Zheng-Shi}
	\begin{enumerate}[\bfseries (1)]
		\item Let $(\Omega,\mathcal{F},\{\mathcal{F}_t\}_{t\geqslant0},\mathbb{P})$ be a complete probability space, on which standard Brownian motions $\{W^1_t,W^2_t\}_{t\geqslant0}\in\mathbb{R}^2$ and Poisson random
              measures $N_i$ with the compensator $\mathbb{E}N_i(de,dt)=\nu_i(de)dt$, for $i=1,2,$ are mutually independent. Here $\nu_i$ is assumed to be a $\sigma$-finite L\'{e}vy measure on $(E_i,\mathcal{B}(E_i))$ with $\int_{E_i}(1\land|e|^2)\nu_i(de)<\infty$, for $i=1,2$.
		\item Define $L^2(E_i,\mathcal{B}(E_i),\nu_i;\mathbb{R})\coloneqq\big\{\tilde{z}^i_{(t,e)}\in \mathbb{R}: \big[\int_{E_i}|\tilde{z}^i_{(t,e)}|^2\nu_i(de)\big]^{\frac{1}{2}}<\infty\big\}$, for $i=1,2$.
		\item $b$, $\sigma_1$, $\sigma_2$, $g$, $f_1$, $f_2$ and $\phi$ are measurable.
		\item $b$, $g$ are uniformly Lipschitz with respect to $(x,y,z^1,z^2,\tilde{z}^1,\tilde{z}^2)$, and $\phi(x)$ is uniformly Lipschitz with respect to $x\in \mathbb{R}$. $\sigma_1$, $\sigma_2$, $f_1$, $f_2$ are
              uniformly Lipschitz with respect to $(x,y)$. Especially, the Lipschitz coefficient of $f_i$ is a measurable function $\rho:E_i\rightarrow\mathbb{R}^{+}$ with $\int_{E_i}\rho^p(e)\nu_i(de)<\infty$ $(p\geqslant2)$, such that for all $t\in[0,T]$, $(x,y)\in\mathbb{R}\times\mathbb{R}$, $(x',y')\in\mathbb{R}\times\mathbb{R}$, and $e\in E_i$, $\mathbb{P}\mbox{-}a.s.$,
		      \begin{equation*}
			     \big|f_i(t,x,y,e)-f_i(t,x',y',e)\big|\leqslant\rho(e)(|x-x'|+|y-y'|).
		      \end{equation*}
		\item For $p\geqslant2$, we have
		      \begin{equation*}
			  \begin{aligned}
				 &\mathbb{E}|\phi(0)|^p+\mathbb{E}\Big(\int_0^T|b(t,\omega,0,0,0,0,0,0)|dt\Big)^p+\mathbb{E}\Big(\int_0^T|\sigma_i(t,\omega,0,0)|^2dt\Big)^{\frac{p}{2}}\\
				 &+\mathbb{E}\Big(\int_0^T\int_{E_i}|f_i(t,\omega,0,0,e)|^2N_i(de,dt)\Big)^{\frac{p}{2}}+\mathbb{E}\Big(\int_0^T|g(t,\omega,0,0,0,0,0,0)|dt\Big)^p<\infty.
			  \end{aligned}
		      \end{equation*}
		\item For any $t\in[0,T]$, $\Theta\in \mathbb{R}\times \mathbb{R}\times \mathbb{R}\times \mathbb{R}\times L^2(E_1,\mathcal{B}(E_1),\nu_1;\mathbb{R})\times L^2(E_2,\mathcal{B}(E_2),\nu_2;\mathbb{R})$,
              $\mathbb{P}\mbox{-}a.s.$,
		      \begin{equation*}
			    |b(t,\Theta)|+|\sigma_i(t,x,y)|+|g(t,\Theta)|+|\phi(x)|\leqslant L(1+|x|+|y|+|z^1|+|z^2|+\Vert\tilde{z}^1\Vert+\Vert\tilde{z}^2\Vert),
		      \end{equation*}
		      and
		      \begin{equation*}
			    |f_i(t,x,y,e)|\leqslant\rho(e)(1+|x|+|y|).
		      \end{equation*}
	\end{enumerate}
\end{assumption}

The following $L^p$ result of FBSDEPs is of \cite{ZS23} (Theorem 2.2 and Theorem 2.6).

\begin{proposition}\label{theorem solvability xyz}
	Suppose that Assumption \ref{assumption A4 in Zheng-Shi} holds, for any $p\geqslant2$, there exists a constant $\tilde{T}>0$ depending on Lipschitz coefficients $L$ and $\rho$ such that, for every $0\leqslant T\leqslant \tilde{T}$, the fully coupled FBSDEP \eqref{equation x y} has a unique solution $\big(x,y,z^1,z^2,\tilde{z}^1,\tilde{z}^2\big)$ in $\mathcal{M}^p$. Moreover, the $L^p$($p\geqslant2$)-estimate hold:
	\begin{equation}\label{Lp estimate xy}
		\begin{aligned}
			&\mathbb{E}\bigg[\sup_{0\leqslant t\leqslant T}|x_t|^p+\sup_{0\leqslant t\leqslant T}|y_t|^p+\Big(\int_0^T|z^i_t|^2dt\Big)^{\frac{p}{2}}
			+\Big(\int_0^T\int_{E_i}|\tilde{z}^i|^2\nu_i(de)dt\Big)^{\frac{p}{2}}\bigg]\\
			&\leqslant C_{p,L,\rho}\mathbb{E}\bigg[|x_0|^p+|\phi(0)|^p+\Big(\int_0^T|g(t,0,0,0,0,0,0)|dt\Big)^p+\Big(\int_0^T|b(t,0,0,0,0,0,0)|dt\Big)^p\\
			&\qquad\qquad +\Big(\int_0^T|\sigma_i(t,0,0)|^2dt\Big)^{\frac{p}{2}}+\Big(\int_0^T\int_{E_i}|f_i(t,0,0,e)|^2N_i(de,dt)\Big)^{\frac{p}{2}}\bigg].
		\end{aligned}
	\end{equation}
\end{proposition}

\subsection{Well-posedness of quadratic-exponential BSDEPs}

In this section, we study the well-posedness of the third equation in \eqref{state equation} of $(\zeta,\kappa^i,\tilde{\kappa}^i)$, which, actually, is the so-called Q$_{exp}$BSDEP. We consider the following general form:
\begin{equation}\label{equation QEXP BSDEP}
	\left\{
	\begin{aligned}
		-dY_t&=\gamma\big(t,Y_t,Z_t,\tilde{Z}_{(t,e)}\big)dt-Z_tdW_t-\int_E \tilde{Z}_{(t,e)}\tilde{N}(de,dt),\quad t\in[0,T],\\
		Y_T&=\xi,
	\end{aligned}
	\right.
\end{equation}
where $\xi:\Omega\rightarrow\mathbb{R}$, $\gamma:\Omega\times[0,T]\times\mathbb{R}\times\mathbb{R}\times\mathbb{L}^2(E,\mathcal{B}(E),\nu;\mathbb{R})\rightarrow\mathbb{R}$.

BMO-martingales play a crucial role in research of quadratic BSDEs (\cite{BE13,Zhang17}). The recent literature is very rich on the theory of continuous BMO-martingales. However, it is clearly not as well documented when it comes to c\`adl\`ag BMO-martingales. So, for the readers' convenience, some properties and notation for c\`adl\`ag BMO-martingales are listed in this subsection. We refer readers to \cite{Kazamaki94,KPZ15,FT18} and the references therein for more details.

$BMO$ is the space of square integrable c\`adl\`ag martingales $M$ such that for all stopping time $\tau\in[0,T]$,
$$
\Vert M\Vert_{BMO}\coloneqq\esssup_\tau\Big\Vert\mathbb{E}_\tau(M_T-M_{\tau-})^2\Big\Vert_\infty<\infty.
$$

\noindent$\mathbb{H}^2_{BMO}$ is the space of $\mathbb{F}$-progressively measurable processes $Z$ such that $\int_0^\cdot Z_sdW_s\in BMO$, i.e.,
$$
\Vert Z\Vert^2_{\mathbb{H}^2_{BMO}}\coloneqq\bigg\Vert\int_0^\cdot Z_sdW_s\bigg\Vert_{BMO}=\sup_\tau\bigg\Vert\mathbb{E}\Big[\int_\tau^T|Z_s|^2ds\Big|\mathcal{F}_\tau\Big]\bigg\Vert_\infty<\infty.
$$

\noindent$\mathbb{J}^2_{BMO}$ and $\mathbb{J}^2_B$ are the spaces of predictable processes $\tilde{Z}$ such that $\int_0^\cdot\int_{E}\tilde{Z}_{(s,e)}\tilde{N}(de,ds)\in BMO$, i.e.,
\begin{equation*}
	\begin{aligned}
		\Vert \tilde{Z}\Vert^2_{\mathbb{J}^2_{BMO}}&\coloneqq\bigg\Vert\int_0^\cdot \int_E \tilde{Z}_{(s,e)}\tilde{N}(de,ds)\bigg\Vert_{BMO}\\
		&=\sup_\tau\bigg\Vert\mathbb{E}\int_0^T\int_E|\tilde{Z}_{(s,e)}|^2N(de,ds)+(\Delta M_\tau)^2\bigg\Vert_\infty<\infty,
	\end{aligned}
\end{equation*}
where $\Delta M_\tau$ is a jump of $M\coloneqq\int_\tau^\cdot \int_E \tilde{Z}_{(s,e)}\tilde{N}(de,ds)$ at $\tau$, and
$$
\Vert \tilde{Z}\Vert^2_{\mathbb{J}^2_B}\coloneqq\sup_\tau\bigg\Vert\mathbb{E}\int_0^T\int_E|\tilde{Z}_{(t,e)}|^2\nu(de)ds+(\Delta M_\tau)^2\bigg\Vert_\infty<\infty,
$$
respectively.

\begin{lemma}
	Suppose $M$ is a square integrable martingale with initial value $M_0=0$. If $M$ is a $BMO$-martingale, then its jump component is essentially bounded $\Delta M\in\mathcal{S}^\infty$. On the other hand, if $\Delta M\in\mathcal{S}^\infty$ and $\sup_\tau\big\Vert\mathbb{E}[\langle M\rangle_T-\langle M\rangle_\tau|\mathcal{F}_\tau]\big\Vert_\infty<\infty$, then $M$ is a $BMO$-martingale.
\end{lemma}

The following result is the so-called energy inequality.
\begin{lemma}\label{energy inequality}
	Let $Z\in\mathbb{H}^2_{BMO}$ and $\tilde{Z}\in\mathbb{J}^2_{BMO}$. Then, for any $n\in\mathbb{N}$,
	\begin{equation*}
		\begin{aligned}
			&\mathbb{E}\bigg(\int_0^T|Z_s|^2ds\bigg)^n\leqslant n!\Big(\Vert Z\Vert^2_{\mathbb{H}^2_{BMO}}\Big)^n,\\
			&\mathbb{E}\bigg(\int_0^T\int_E|\tilde{Z}_{(s,e)}|^2N(de,ds)\bigg)^n\leqslant n!\Big(\Vert \tilde{Z}\Vert^2_{\mathbb{J}^2_{BMO}}\Big)^n,\\
			&\mathbb{E}\bigg(\int_0^T\int_E|\tilde{Z}_{(s,e)}|^2\nu(de)ds\bigg)^n\leqslant n!\Big(\Vert \tilde{Z}\Vert^2_{\mathbb{J}^2_{B}}\Big)^n\leqslant n!\Big(\Vert \tilde{Z}\Vert^2_{\mathbb{J}^2_{BMO}}\Big)^n.
		\end{aligned}
	\end{equation*}
\end{lemma}

Let $\mathcal{E}(M)$ be the Dol\'{e}an-Dade exponential. The following reverse H\"{o}lder's inequality holds.
\begin{lemma}
	Let $\delta>0$ be a positive constant and $M$ be a $BMO$-martingale satisfying $\Delta M_t\geqslant-1+\delta$, $\mathbb{P}$-a.s., for all $t\in[0,T]$. Then $(\mathcal{E}_t(M), t\in[0,T])$ is a uniformly integrable martingale, and for every stopping time $\tau$, there exists some $p>1$ and some positive constant $C_{p,\Vert M\Vert_{BMO}}$, such that
$$
\mathbb{E}\big[\mathcal{E}_T(M)^p|\mathcal{F}_\tau\big]\leqslant C_{p,\Vert M\Vert_{BMO}}\mathcal{E}_\tau(M)^p.
$$
\end{lemma}

Results above allow us to obtain immediately a Girsanov's theorem, which will be useful throughout this paper.

\begin{lemma}
	Consider a c\`adl\`ag martingale $M$ given by
	\begin{equation*}
		M_t\coloneqq\int_0^Ta_sdW_s+\int_0^T\int_{E}b_{(s,e)}\tilde{N}(de,ds),\quad\mathbb{P}\mbox{-}a.s.,
	\end{equation*}
where $b$ is bounded and $(a,b)\in\mathbb{H}^2_{BMO}\times\mathbb{J}^2_{BMO}$ and where there exists $\delta>0$ with $b_t\geqslant-1+\delta$, $\mathbb{P}\times\nu$-a.e., for all $t\in[0,T]$. Then, the probability measure $\mathbb{Q}$ defined by $\frac{d\mathbb{Q}}{d\mathbb{P}}=\mathcal{E}(M_\cdot)$ is well-defined and starting from any $\mathbb{P}$-martingale, by, as usual, changing adequately the drift and the jump
intensity, we can obtain a $\mathbb{Q}$-martingale.
\end{lemma}

We need now to specify in more details the assumptions we make on the generator $\gamma$. The most important one in our setting
will be the quadratic-exponential structure of Assumption \ref{assumption QEXP} below. It is the natural generalization to the jump case of the usual quadratic growth assumption in $z$. Before proceeding further, let us define the following function
\begin{equation}
	[\tilde{Z}]_\theta\coloneqq\int_{E}\frac{1}{\theta}\Big[e^{\theta \tilde{Z}_{(t,e)}}-1-\theta \tilde{Z}_{(t,e)}\Big]\nu(de).
\end{equation}

Now let us back to \eqref{equation QEXP BSDEP}. Let us introduce the quadratic-exponential structure condition by Barrieu and El Karoui \cite{BE13} and extended to the jump-diffusion case by Ngoupeyou \cite{Ngoupeyou10} or El Karoui et al. \cite{EMN16}.
\begin{assumption}\label{assumption QEXP}
	\begin{enumerate}[\bfseries (1)]
		\item (Q$_{exp}$ structure condition) The map $(\omega,t)\rightarrow\gamma(\omega,t,\cdot,\cdot,\cdot)$ is $\mathbb{F}$-progressively measurable.
               For all $(y,z,u)\in\mathbb{R}\times\mathbb{R}\times\mathbb{L}^2(E,\nu;\mathbb{R})$, there exist two constants $\beta\geqslant0$ and $\theta\geqslant0$ and a positive $\mathbb{F}$-progressively measurable process $\{\alpha_t\}_{t\geqslant0}$ such that
		       \begin{equation}
			   \begin{aligned}
				  -\alpha_t-\beta|y|-\frac{\theta}{2}|z|^2-[-\tilde{z}]_\theta \leqslant\gamma(t,y,z,u)\leqslant\alpha_t+\beta|y|+\frac{\theta}{2}|z|^2+[\tilde{z}]_\theta,\quad a.e.\,t\in[0,T],\ \mathbb{P}\mbox{-}a.s..
			   \end{aligned}
		       \end{equation}
	    \item $\xi$ and $\{\alpha_t\}_{t\geqslant0}$ are essentially bounded, i.e., $\Vert\xi\Vert_{\infty},\Vert\alpha\Vert_{\mathcal{S}^\infty}<\infty$.
        \item For each $M>0$, and for all $(y,z,\tilde{z})$, $(y',z',\tilde{z}')\in\mathbb{R}\times\mathbb{R}\times\mathbb{L}^2(E,\nu;\mathbb{R})$ satisfying $|y|$, $|y'|$,
              $\Vert \tilde{z}\Vert_{\mathbb{L}^\infty(\nu)}$, $\Vert \tilde{z}'\Vert_{\mathbb{L}^\infty(\nu)}\leqslant M$, there exists some positive constant $K_M$ possibly depending on $M$ such that
              \begin{equation*}
	          \begin{aligned}
		        &|\gamma(t,y,z,\tilde{z})-\gamma(t,y',z',\tilde{z}')|\leqslant K_M\big(|y-y'|+\Vert \tilde{z}-\tilde{z}'\Vert_{\mathbb{L}^2(\nu)}\big)\\
		        &\quad+K_M\big(1+|z|+|z'|+\Vert \tilde{z}\Vert_{\mathbb{L}^2(\nu)}+\Vert \tilde{z}\Vert_{\mathbb{L}^2(\nu)}\big)|z-z'|,\quad a.e.\,t\in[0,T],\ \mathbb{P}\mbox{-}a.s..
	          \end{aligned}
              \end{equation*}
        \item ($A_{\Gamma}$ condition). For all $t\in[0,T]$, $M>0$, $y$, $z\in\mathbb{R}$, $\tilde{z}$, $\tilde{z}'\in\mathbb{L}^2(E,\nu;\mathbb{R})$ with $|y|$, $\Vert \tilde{z}\Vert_{\mathbb{L}^\infty(\nu)}$,
              $\Vert \tilde{z}'\Vert_{\mathbb{L}^\infty(\nu)}\leqslant M$, there exists a $\mathcal{P}\otimes\mathcal{B}(E)$-measurable process $\Gamma^{y,z,\tilde{z},\tilde{z}'}$ satisfying
              \begin{equation*}
	            \gamma(t,y,z,\tilde{z})-\gamma(t,y,z,\tilde{z}')\leqslant\int_{E}\Gamma^{y,z,\tilde{z},\tilde{z}'}_{(t,e)}[\tilde{z}_e-\tilde{z}'_e]\nu(de),\quad a.e.\,t\in[0,T],\ \mathbb{P}\mbox{-}a.s.,
              \end{equation*}
              and $C_M^1(1\wedge|e|)\leqslant\Gamma^{y,z,\tilde{z},\tilde{z}'}_{(t,e)}\leqslant C_M^2(1\wedge|e|)$ with two constants $C_M^1$, $C_M^2$. Here, $C_M^1\geqslant-1+\delta$ for some $\delta>0$ and $C_M^2>0$ depend on $M$.
     \end{enumerate}
\end{assumption}

The following result is from \cite{FT18}'s Lemma 3.2, Lemma 3.3 and Theorem 4.1.

\begin{proposition}\label{theorem solvability zeta}
	Under Assumption \ref{assumption QEXP}, there exists a unique bounded solution $(Y,Z,\tilde{Z})\in\mathcal{S}^\infty\times\mathbb{H}^2_{BMO}\times\mathbb{J}^2_{BMO}$ of the BSDEP \eqref{equation QEXP BSDEP}. Moreover, it satisfies
\begin{equation*}\label{norm YZU}
\begin{aligned}
	\Vert Y\Vert_{\mathcal{S}^\infty}&\leqslant e^{\beta T}\big(\Vert\xi\Vert_\infty+T\Vert\alpha\Vert_{\mathcal{S}^\infty}\big),\\
	\Vert Z\Vert_{\mathbb{H}^2_{BMO}}+\Vert \tilde{Z}\Vert_{\mathbb{J}^2_{BMO}}&\leqslant C\bigg(\Vert Y\Vert_{\mathcal{S}^\infty}+\Vert\xi\Vert_\infty
      +\sup_\tau\bigg\Vert\mathbb{E}\Big[\int_\tau^T|\gamma(s,0,0,0)|ds\Big|\mathcal{F}_\tau\Big]\bigg\Vert_\infty\bigg),\\
	\Vert Y\Vert^p_{\mathcal{S}^p}+\Vert Z\Vert^p_{\mathbb{H}^p}+\Vert \tilde{Z}\Vert^p_{\mathbb{J}^p}&
    \leqslant C'\bigg(\mathbb{E}\bigg[|\xi|^{p\bar{q}^2}+\Big(\int_0^T|\gamma(s,0,0,0)|ds\Big)^{p\bar{q}^2}\bigg]\bigg)^{\frac{1}{\bar{q}^2}},\quad\forall p\geqslant2,\ \forall\bar{q}\geqslant q_*,
\end{aligned}
\end{equation*}
where $q_*$ is a positive constant depending on $(K_M, \theta, \beta, T, \Vert\xi\Vert_{\infty}, \Vert\alpha\Vert_{\mathcal{S}^{\infty}})$.
\end{proposition}

\begin{remark}
	Suppose $M$ is a continuous BMO-martingale. There exist a function $\varPhi(x)=\bigl\{1+\frac{1}{x^2}\ln\frac{2x-1}{2(x-1)}\bigr\}^{\frac{1}{2}}-1$, such that, for all $1<p<p_*$, $M$ satisfies p-reverse H\"older's inequality. Here $p_*$ is determined by $\varPhi(p_*)=\Vert M\Vert_{BMO}$, and $q_*$ in the above proposition is the conjugate of $p_*$. When $M$ is a c\`adl\`ag BMO-martingale, there still exists a $p_*$. However, we can not determine $p_*$ explicitly as in the continuous setting. More details can be found in \cite{Kazamaki94} and \cite{FT18}.
\end{remark}

\section{Main results}

In this section, we study the global maximum principle for the partially observed risk-sensitive progressive optimal control problem of coupled FBSDEPs. The main result is an extension of the work in \cite{ZS23} and \cite{HJX22}.

\subsection{Spike variation}

As discussed above, we first transfer the original problem into a complete information stochastic recursive optimal control one, of the controlled coupled Q$_{exp}$FBSDEP (see also \eqref{state equation}):
\begin{equation}\label{equation state x y zeta in section 3}
	\left\{
	\begin{aligned}
		dx_t&=b(t,\Theta(t),u_t)dt+\sigma_{i}(t,x_t,y_t,u_t)dW^{i}_t+\int_{E_i}f_i(t,x_t,y_t,u_t,e)\tilde{N}_{i}(de,dt),\\
		-dy_t&=g(t,\Theta(t),u_t)dt-z^i_tdW^i_t-\int_{E_i}\tilde{z}^i_{(t,e)}\tilde{N}_i(de,dt),\\
		-d\zeta_t&=l\Big(t,\Theta(t),\kappa^1_t,\kappa^2_t,\int_{E_1}\tilde{\kappa}^1_{(t,e)}\nu_1(de),\int_{E_2}\tilde{\kappa}^2_{(t,e)}\nu_2(de),u_t\Big)dt\\
        &\quad-\kappa^i_tdW^i_t-\int_{E_i}\tilde{\kappa}^i_{(t,e)}\tilde{N}_i(de,dt),\quad t\in[0,T],\\
		x_0&=x_0,\quad y_T=\phi(x_T),\quad \zeta_T=\varphi(x_T,y_0),
	\end{aligned}
	\right.
\end{equation}
which is uniquely solvable due to Proposition \ref{theorem solvability xyz} and Proposition \ref{theorem solvability zeta}. Moreover, the cost functional we consider is the following recursive case (see \eqref{cost functional inf}):
\begin{equation}\label{cost functional zeta0}
	J(u)=\zeta_0.
\end{equation}

Since control domain $U$ is not necessarily convex, we apply the new spike variation technique established in \cite{STW20}, in our progressive framework. Suppose that $\bar{u}\in\mathcal{U}_{ad}$ is an optimal control. For any $\bar{t}\in[0,T]$, define $u^{\epsilon}$ as
\begin{equation*}
	\begin{aligned}
		u^\epsilon_s=
		\begin{cases}
			u,&\text{if $(s,\omega)\in\mathcal{O}\coloneqq\rrbracket\bar{t},\bar{t}+\epsilon\rrbracket\backslash\bigcup_{n=1}^\infty\llbracket T_n\rrbracket$},\\
			\bar{u}_s,&\text{otherwise},
		\end{cases}
	\end{aligned}
\end{equation*}
where $\llbracket T_n\rrbracket\coloneqq\big\{(\omega,t)\in\Omega\times[0,T]|T_n(\omega)=t\big\}$, which is the graph of stopping time $T_n$, is a progressive set, and $u$ is a bounded $\mathcal{F}^Y_{\bar{t}}$-measurable function taking values in $U$. The big difference is that the value of $u^\epsilon_s$ at $T_n(\omega)$ is equal to $\bar{u}_s$ rather than $u$ when some jump appears in $(\bar{t},\bar{t}+\epsilon]$, that is, some $T_n(\omega)$ is in $(\bar{t},\bar{t}+\epsilon]$. Moreover, due to technical difficulties caused by partial observability, we also assume that there are no jumps of the unobservable process appearing in $(\bar{t},\bar{t}+\epsilon]$. It is easy to show that $u^\epsilon\in\mathcal{U}_{ad}[0,T]$.

Besides Assumption \ref{assumption A2 in Zheng-Shi}, to get a more general result, we introduce the following assumption.
\begin{assumption}\label{assumption of l(t)}
	The first-order derivatives of $l$ with respect to $(x,y,z^1,z^2,\tilde{z}^1,\tilde{z}^2)$ are Lipschitz continuous and bounded, and also the second order derivatives of $l$ with respect to $(x,y,z^1,z^2,\tilde{z}^1,\tilde{z}^2,\kappa^1,\kappa^2,\tilde{\kappa}^1,\tilde{\kappa}^2)$ are continuous and bounded. There exist positive constants $L_1$, $L_2$, $L_3$, $L_4$, such that
	\begin{equation*}
		\begin{aligned}
			&|l(t,x,0,0,0,0,0,0,0,0,0,u)|\leqslant L_1,\\
	    \end{aligned}
	\end{equation*}
    \begin{equation*}
		\begin{aligned}		
            &|l(t,x,y,z^1,z^2,\tilde{z}^1,\tilde{z}^2,\kappa^1,\kappa^2,\tilde{\kappa}^1,\tilde{\kappa}^2,u_1)-l(t,x,y,z^1,z^2,\tilde{z}^1,\tilde{z}^2,\kappa^1,\kappa^2,\tilde{\kappa}^1,\tilde{\kappa}^2,u_2)|\\
			&\quad\leqslant L_2\big(1+|x|+|y|+|z^1|+|z^2|+\Vert\tilde{z}^1\Vert+\Vert\tilde{z}^2\Vert+|\kappa^1|+|\kappa^2|+\Vert\tilde{\kappa}^1\Vert+\Vert\tilde{\kappa}^2\Vert\big),\\
	\end{aligned}
	\end{equation*}
    \begin{equation*}
		\begin{aligned}				
            &|l_{\kappa^i}(t,x,y,z^1,z^2,\tilde{z}^1,\tilde{z}^2,\kappa^1,\kappa^2,\tilde{\kappa}^1,\tilde{\kappa}^2,u)|\leqslant L_3(1+|\kappa^i|),\quad i=1,2,\\
			&|l_{\tilde{\kappa}^i}(t,x,y,z^1,z^2,\tilde{z}^1,\tilde{z}^2,\kappa^1,\kappa^2,\tilde{\kappa}^1,\tilde{\kappa}^2,u)|\leqslant L_4(1+\Vert e^{\tilde{\kappa}^i}\Vert),\quad i=1,2.
		\end{aligned}
	\end{equation*}
\end{assumption}

Let $\bar{\Xi}\coloneqq(\bar{x},\bar{y},\bar{z}^1,\bar{z}^2,\bar{\tilde{z}}^1,\bar{\tilde{z}}^2,\bar{\kappa}^1,\bar{\kappa}^2,\bar{\tilde{\kappa}}^1,\bar{\tilde{\kappa}}^2)$ be the optimal trajectory corresponding to $\bar{u}$ and $\Xi^\epsilon\coloneqq(x^\epsilon,y^\epsilon,z^{1,\epsilon},z^{2,\epsilon},\tilde{z}^{1,\epsilon},\tilde{z}^{2,\epsilon},\kappa^{1,\epsilon},\kappa^{2,\epsilon},\tilde{\kappa}^{1,\epsilon},\tilde{\kappa}^{2,\epsilon})$ corresponding to $u^\epsilon$. Denote $\bar{\Theta}\coloneqq(\bar{x},\bar{y},\bar{z}^1,\bar{z}^2,\\\bar{\tilde{z}}^1,\bar{\tilde{z}}^2)$
and $\Theta^\epsilon\coloneqq(x^{\epsilon},y^{\epsilon},z^{1,\epsilon},z^{2,\epsilon},\tilde{z}^{1,\epsilon},\tilde{z}^{2,\epsilon})$.

For simplicity, for $\psi=\sigma_i$, $f_i$, $i=1,2$, we denote
\begin{equation*}
	\begin{aligned}
		\psi(t)&\coloneqq\psi(t,\bar{x}_t,\bar{y}_t,\bar{u}_t),\quad\psi_{x}(t)\coloneqq\psi_{x}(t,\bar{x}_t,\bar{y}_t,\bar{u}_t),\\
		\delta\psi(t)&\coloneqq\psi(t,\bar{x}_t,\bar{y}_t,u_t)-\psi(t),\quad\delta\psi_{x}(t)\coloneqq\psi_{x}(t,\bar{x}_t,\bar{y}_t,u_t)-\psi_{x}(t);
	\end{aligned}
\end{equation*}
for $\psi=b$, $g$, we denote
\begin{equation*}
	\begin{aligned}
		\psi(t)&\coloneqq\psi(t,\bar{\Theta}(t),\bar{u}_t),\quad\psi_x(t)\coloneqq\psi_x(t,\bar{\Theta}(t),\bar{u}_t),\\
		\delta\psi(t)&\coloneqq\psi(t,\bar{\Theta}(t),u_t)-\psi(t),\quad\delta\psi_{x}(t)\coloneqq\psi_x(t,\bar{\Theta}(t),u_t)-\psi_x(t),\\
		\delta\psi(t,\Delta^1,\Delta^2)&\coloneqq\psi(t,\bar{x}_t,\bar{y}_t,\bar{z}^1_t+\Delta^1_t,\bar{z}^2_t+\Delta^2_t,\bar{\tilde{z}}^1_{(t,e)},\bar{\tilde{z}}^2_{(t,e)},u_t)-\psi(t),\\
		\delta\psi_x(t,\Delta^1,\Delta^2)&\coloneqq\psi_x(t,\bar{x}_t,\bar{y}_t,\bar{z}^1_t+\Delta^1_t,\bar{z}^2_t+\Delta^2_t,\bar{\tilde{z}}^1_{(t,e)},\bar{\tilde{z}}^2_{(t,e)},u_t)-\psi_x(t);
	\end{aligned}
\end{equation*}
and for $l$, we denote
\begin{equation*}
	\begin{aligned}
		l(t)&\coloneqq l(t,\bar{\Xi}(t),\bar{u}_t),\quad l_{x}(t)\coloneqq\psi_x(t,\bar{\Xi}(t),\bar{u}_t),\\
		\delta l(t)&\coloneqq l(t,\bar{\Xi}(t),u_t)-l(t),\quad\delta l_{x}(t)\coloneqq l_x(t,\bar{\Xi}(t),u_t)-l_x(t),\\
		\delta l\big(t,\Delta^1,\Delta^2,\pi^1,\pi^2)&\coloneqq l\big(t,\bar{x}_t,\bar{y}_t,\bar{z}^1_t+\Delta^1_t,\bar{z}^2_t+\Delta^2_t,\bar{\tilde{z}}^1_{(t,e)},\bar{\tilde{z}}^2_{(t,e)},\\
        &\qquad \bar{\kappa}^1_t+\pi^1_t,\bar{\kappa}^2_t+\pi^2_t,\bar{\tilde{\kappa}}^1_{(t,e)},\bar{\tilde{\kappa}}^2_{(t,e)},u_t\big)-l(t),\\
		\delta l_x(t,\Delta^1,\Delta^2,\pi^1,\pi^2)&\coloneqq l_{x}\big(t,\bar{x}_t,\bar{y}_t,\bar{z}^1_t+\Delta^1_t,\bar{z}^2_t+\Delta^2_t,\bar{\tilde{z}}^1_{(t,e)},\bar{\tilde{z}}^2_{(t,e)},\\
        &\qquad \bar{\kappa}^1_t+\pi^1_t,\bar{\kappa}^2_t+\pi^2_t,\bar{\tilde{\kappa}}^1_{(t,e)},\bar{\tilde{\kappa}}^2_{(t,e)},u_t\big)-l_x(t);
	\end{aligned}
\end{equation*}
and similar notations are used for derivatives of $y$, $z^i$, $\tilde{z}^i$, $\kappa^i$, $\tilde{\kappa}^i$. In the above, $\Delta^i$ and $\pi^i$ are $\mathbb{F}$-adapted processes to be determined later. Set
\begin{equation*}
\begin{aligned}
	&\hat{x}^{1,\epsilon}\coloneqq x^\epsilon-\bar{x},\quad\hat{y}^{1,\epsilon}\coloneqq y^\epsilon-\bar{y},\quad\hat{z}^{i,1,\epsilon}\coloneqq z^{i,\epsilon}-\bar{z}^i,\\
    &\hat{\tilde{z}}^{i,1,\epsilon}\coloneqq \tilde{z}^{i,\epsilon}-\bar{\tilde{z}}^i,\quad\hat{\kappa}^{i,1,\epsilon}\coloneqq \kappa^{i,\epsilon}-\bar{\kappa}^i,\quad
    \hat{\tilde{\kappa}}^{i,1,\epsilon}\coloneqq \tilde{\kappa}^{i,\epsilon}-\bar{\tilde{\kappa}},
\end{aligned}
\end{equation*}
which satisfy the following equation:
\begin{equation}\label{equation of hat x y zeta 1}
	\left\{
	\begin{aligned}
		d\hat{x}^{1,\epsilon}_t&=\Bigl\{\tilde{b}_x(t)\hat{x}^{1,\epsilon}_t+\tilde{b}_y(t)\hat{y}^{1,\epsilon}_t+\tilde{b}_{z^i}(t)\hat{z}^{i,1,\epsilon}_t
        +\tilde{b}_{\tilde{z}^i}(t)\int_{E_i}\hat{\tilde{z}}^{i,1,\epsilon}_{(t,e)}\nu_{i}(de)+\delta b(t)\mathbbm{1}_{[\bar{t},\bar{t}+\epsilon]}(t)\Bigr\}dt\\
        &\quad+\Big\{\tilde{\sigma}_{ix}(t)\hat{x}^{1,\epsilon}_t+\tilde{\sigma}_{iy}(t)\hat{y}^{1,\epsilon}_t+\delta \sigma_{i}(t)\mathbbm{1}_{[\bar{t},\bar{t}+\epsilon]}(t)\Big\}dW^i_t\\
		&\quad+\int_{E_i}\Big\{\tilde{f}_{ix}(t)\hat{x}^{1,\epsilon}_t+\tilde{f}_{iy}(t)\hat{y}^{1,\epsilon}_t+\delta f_{i}(t)\mathbbm{1}_{\mathcal{O}}(t)\Big\}\tilde{N}_i(de,dt),\\
        -d\hat{y}^{1,\epsilon}_t&=\Bigl\{\tilde{g}_x(t)\hat{x}^{1,\epsilon}_t+\tilde{g}_y(t)\hat{y}^{1,\epsilon}_t+\tilde{g}_{z^i}(t)\hat{z}^{i,1,\epsilon}_t
        +\tilde{g}_{\tilde{z}^i}(t)\int_{E_i}\hat{\tilde{z}}^{i,1,\epsilon}_{(t,e)}\nu_{i}(de)+\delta g(t)\mathbbm{1}_{[\bar{t},\bar{t}+\epsilon]}(t)\Bigr\}dt\\
        &\quad-\hat{z}^{i,1,\epsilon}_tdW^i_t-\int_{E_i}\hat{\tilde{z}}^{i,1,\epsilon}_{(t,e)}\tilde{N}_i(de,dt),\\
		-d\hat{\zeta}^{1,\epsilon}_t&=\Bigl\{\tilde{l}_x(t)\hat{x}^{1,\epsilon}_t+\tilde{l}_y(t)\hat{y}^{1,\epsilon}_t+\tilde{l}_{z^i}(t)\hat{z}^{i,1,\epsilon}_t
        +\tilde{l}_{\tilde{z}^i}(t)\int_{E_i}\hat{\tilde{z}}^{i,1,\epsilon}_{(t,e)}\nu_{i}(de)+\tilde{l}_{\kappa^i}(t)\hat{\kappa}^{i,1,\epsilon}_t\\
		&\qquad+\tilde{l}_{\tilde{\kappa}^i}(t)\int_{E_i}\hat{\tilde{\kappa}}^{i,1,\epsilon}_{(t,e)}\nu_i(de)+\delta l(t)\mathbbm{1}_{[\bar{t},\bar{t}+\epsilon]}(t)\Bigr\}dt\\
        &\quad-\hat{\kappa}^{i,1,\epsilon}_tdW^i_t-\int_{E_i}\hat{\tilde{\kappa}}^{i,1,\epsilon}_{(t,e)}\tilde{N}_i(de,dt),\quad t\in[0,T],\\
		\hat{x}^{1,\epsilon}_0&=0,\quad\hat{y}^{1,\epsilon}_T=\tilde{\phi}_x(T)\hat{x}^{1,\epsilon}_T,\quad\hat{\zeta}^{1,\epsilon}_T=\tilde{\varphi}_x(T,0)\hat{x}^{1,\epsilon}_T+\tilde{\varphi}_y(T,0)\hat{y}^{1,\epsilon}_0,
	\end{aligned}
	\right.
\end{equation}
where
\begin{equation*}
	\begin{aligned}
		\tilde{l}_x(t)&\coloneqq\int_0^1l_x\Big(t,\bar{x}_t+\theta\hat{x}^{1,\epsilon}_t,\bar{y}_t+\theta\hat{y}^{1,\epsilon}_t,\bar{z}_t+\theta\hat{z}^{1,\epsilon}_t,\int_{E}(\bar{\tilde{z}}_{(t,e)}
        +\theta\hat{\tilde{z}}^{1,\epsilon}_{(t,e)})\nu(de),\\
		&\qquad\qquad\bar{\kappa}_t+\theta\hat{\kappa}^{1,\epsilon}_t,\int_{E}(\bar{\tilde{\kappa}}_{(t,e)}+\theta\hat{\tilde{\kappa}}^{1,\epsilon}_{(t,e)})\nu(de),u^{\epsilon}_t\Big)d\theta,\\
		\tilde{\varphi}_x(T,0)&\coloneqq\int_0^1\varphi_x(\bar{x}_T+\theta\hat{x}^{1,\epsilon}_T,\bar{y}_0+\theta\hat{y}^{1,\epsilon}_0)d\theta,\quad
		\tilde{\varphi}_y(T,0)\coloneqq\int_0^1\varphi_y(\bar{x}_T+\theta\hat{x}^{1,\epsilon}_T,\bar{y}_0+\theta\hat{y}^{1,\epsilon}_0)d\theta,
	\end{aligned}
\end{equation*}
and similarly for $b$, $\sigma_i$, $f_i$, $g$, $\phi$ and their derivatives with respect to $y$, $z^i$, $\tilde{z}^i$, $\kappa^i$, $\tilde{\kappa}^i$.

The following result is from Lemma 3.1 of Wang et al. \cite{WSS2024}.

\begin{lemma}\label{lemma of hat x 1}
	Let Assumption \ref{assumption A1 in Zheng-Shi} hold. Then for $p\geqslant2$, we obtain that
\begin{equation}
		\mathbb{E}\bigg[\sup_{0\leqslant t\leqslant T}\Big(|\hat{x}^{1,\epsilon}_t|^p+|\hat{y}^{1,\epsilon}_t|^p\Big)+\Big(\int_0^T\int_{E_i}|\hat{z}^{i,1,\epsilon}_t|^2dt\Big)^{\frac{p}{2}}
        +\Big(\int_0^T\int_{E_i}|\hat{\tilde{z}}^{i,1,\epsilon}_{(t,e)}|^2N_i(de,dt)\Big)^{\frac{p}{2}}\bigg]=O(\epsilon^{\frac{p}{2}}).
\end{equation}
\end{lemma}

\subsection{Stochastic Lipschitz BSDEPs}

Due to the quadratic-exponential feature, some BSDEPs with stochastic Lipschitz coefficients appear in variational equations of $(\zeta,\kappa^1,\kappa^2,\tilde{\kappa}^1,\tilde{\kappa}^2)$. We need the result of the existence and the uniqueness of solutions to this kind equations and their estimates. Consider the following BSDEP:
\begin{equation}\label{equation stochastic Lip BSDEP}
	\left\{
	\begin{aligned}
		-dY_t&=\gamma\big(t,Y_t,Z_t,\tilde{Z}_{(t,e)}\big)dt-Z_tdW_t-\int_{E}\tilde{Z}_{(t,e)}\tilde{N}(de,dt),\quad t\in[0,T],\\
		Y_T&=\xi,
	\end{aligned}
	\right.
\end{equation}
where $\xi:\Omega\rightarrow\mathbb{R}$, $\gamma:\Omega\times[0,T]\times\mathbb{R}\times\mathbb{R}\times\mathbb{L}^2(E,\nu;\mathbb{R})\rightarrow\mathbb{R}$.
\begin{assumption}\label{assumption stochastic Lip BSDEP}
	The map $(\omega,t)\to\gamma(\omega,t,\cdot,\cdot,\cdot)$ is $\mathbb{F}$-progressively measurable.
\begin{enumerate}[{\bfseries (1)}]
	\item There exist a positive constant $K$ and a positive $\mathbb{F}$-progressively measurable process $H\in\mathbb{H}^2_{BMO}$ such that, for
          every $(y,z,\tilde{z}),(y',z',\tilde{z}')\in\mathbb{R}\times\mathbb{R}^d\times\mathbb{L}^2(E,\nu;\mathbb{R}^k)$,
		  \begin{equation*}
          \begin{aligned}
			 &|\gamma(\omega,t,y,z,\tilde{z})-\gamma(\omega,t,y',z',\tilde{z}')|\\
             &\leqslant K\big(|y-y'|+\Vert \tilde{z}-\tilde{z}'\Vert_{\mathbb{L}^2(\nu)}\big)+H_t(\omega)|z-z'|,\quad a.e.\,t\in[0,T],\ \mathbb{P}\mbox{-}a.s..
		  \end{aligned}
          \end{equation*}
	\item $\xi$ is $\mathcal{F}_T$-measurable and for $\forall p\geqslant2$, $\mathbb{E}\Big[|\xi|^p+\Big(\int_0^T|\gamma(t,0,0,0)|dt\Big)^p\Big]<\infty$.
\end{enumerate}
\end{assumption}

The following result is of \cite{FT18}, Theorem A.1.
\begin{lemma}\label{lemma stochastic Lip BSDEP}
	Under Assumption \ref{assumption stochastic Lip BSDEP}, there exists a unique solution $(Y,Z,\tilde{Z})$ to the BSDEP \eqref{equation stochastic Lip BSDEP}. Moreover, for all $p\geqslant2$,
\begin{equation*}
	\begin{aligned}
	&\mathbb{E}\bigg[\sup_{0\leqslant t\leqslant T}|Y_t|^p+\Big(\int_0^T|Z_t|^2dt\Big)^{\frac{p}{2}}+\Big(\int_0^T\int_{E}|\tilde{Z}_{(t,e)}|^2\nu(de)dt\Big)^{\frac{p}{2}}\bigg]\\
    &\leqslant C\bigg(\mathbb{E}\bigg[|\xi|^{p\bar{q}^2}+\Big(\int_0^T\gamma(0)dt\Big)^{p\bar{q}^2}\bigg]\bigg)^{\frac{1}{\bar{q}^2}},
	\end{aligned}
\end{equation*}
with a positive constant $\bar{q}$ satisfying $q_*\leqslant\bar{q}<\infty$ whose lower bound $q_*>1$ is controlled only by $\Vert H\Vert_{\mathbb{H}^2_{BMO}}$, and some positive constant $C$ depending only on $\big(p,\bar{q},T,K,\Vert H\Vert_{\mathbb{H}^2_{BMO}}\big)$.
\end{lemma}

\subsection{First- and second-order variational equations}

With the help of Lemma \ref{lemma stochastic Lip BSDEP}, we can deduce the following estimate whose proof is in the Appendix.
\begin{lemma}\label{lemma of hat zeta 1}
	Under some Assumption \ref{assumption A1 in Zheng-Shi} and \ref{assumption of l(t)}, for each $p>1$,
	\begin{equation}
		\mathbb{E}\bigg[\sup_{0\leqslant t\leqslant T}|\hat{\zeta}^{1,\epsilon}_t|^p+\Big(\int^T_0|\hat{\kappa}^{i,1,\epsilon}_t|^2dt\Big)^{\frac{p}{2}}
        +\Big(\int^T_0\int_{E_i}|\hat{\tilde{\kappa}}^{i,1,\epsilon}_{(t,e)}|^2\nu_i(de)dt\Big)^{\frac{p}{2}}\bigg]=O(\epsilon^{\frac{p}{2}}).
	\end{equation}
\end{lemma}

According to Lemma \ref{lemma of hat x 1} and Lemma \ref{lemma of hat zeta 1}, we could set
\begin{equation}\label{expansion guss}
\begin{aligned}
	&x^\epsilon_t-\bar{x}_t=x^1_t+x^2_t+o(\epsilon),\quad y^\epsilon_t-\bar{y}_t=y^1_t+y^2_t+o(\epsilon),\quad \zeta^{\epsilon}_t-\bar{\zeta}_t=\zeta^1_t+\zeta^2_t+o(\epsilon),\\
	&z^{i,\epsilon}_t-\bar{z}^i_t=z^{i,1}_t+z^{i,2}_t+o(\epsilon),\quad \tilde{z}^{i,\epsilon}_{(t,e)}-\bar{\tilde{z}}^i_{(t,e)}=\tilde{z}^{i,1}_{(t,e)}+\tilde{z}^{i,2}_{(t,e)}+o(\epsilon),\\
	&\kappa^{i,\epsilon}_t-\bar{\kappa}^i_t=\kappa^{i,1}_t+\kappa^{i,2}_t+o(\epsilon),\quad \tilde{\kappa}^{i,\epsilon}_{(t,e)}-\bar{\tilde{\kappa}}^i_{(t,e)}
     =\tilde{\kappa}^{i,1}_{(t,e)}+\tilde{\kappa}^{i,2}_{(t,e)}+o(\epsilon),\quad i=1,2,
\end{aligned}
\end{equation}
where $x^1,y^1,z^{i,1},\tilde{z}^{i,1},\zeta^1,\kappa^{i,1},\tilde{\kappa}^{i,1}\sim O(\sqrt{\epsilon})$ and $x^2,y^2,z^{i,2},\tilde{z}^{i,2},\zeta^2,\kappa^2,\tilde{\kappa}^{i,2}\sim O(\epsilon)$.

Further, inspired by \cite{HJX18,ZS23}, we introduce that $z^{i,1}_t$, $\kappa^{i,1}_t$ have the following forms, respectively:
\begin{equation*}
	z_t^{i,1}=\Delta^i(t)\mathbbm{1}_{[\bar{t},\bar{t}+\epsilon]}+z_t^{i,1'},\quad \kappa_t^{i,1}=\pi^i(t)\mathbbm{1}_{[\bar{t},\bar{t}+\epsilon]}+\kappa_t^{i,1'}, \quad i=1,2,
\end{equation*}
where $\Delta^i(\cdot),\pi^i(\cdot)$ are $\mathbb{F}$-adapted processes, to be determined later, and importantly, $z_t^{i,1'},\kappa_t^{i,1'},i=1,2$ have good estimates similarly as $x_t^1$. Then, for example, for $l$, we have the expansion:
\begin{equation*}
	\begin{aligned}
		&l\big(t,x^\epsilon_t,y^\epsilon_t,z^{1,\epsilon}_t,z^{2,\epsilon}_t,\tilde{z}^{1,\epsilon}_{(t,e)},\tilde{z}^{2,\epsilon}_{(t,e)},\kappa^{1,\epsilon}_t,
        \kappa^{2,\epsilon}_t,\tilde{\kappa}^{1,\epsilon}_{(t,e)},\tilde{\kappa}^{2,\epsilon}_{(t,e)},u^\epsilon_t\big)-l(t)\\
        &=l\big(t,\bar{x}_t+x_t^1+x_t^2,\bar{y}_t+y_t^1+y_t^2,\bar{z}^1_t+\Delta^1(t)\mathbbm{1}_{[\bar{t},\bar{t}+\epsilon]}+z_t^{1,1'}+z_t^{1,2},\\
        &\qquad \bar{z}^2_t+\Delta^2(t)\mathbbm{1}_{[\bar{t},\bar{t}+\epsilon]}+z_t^{2,1'}+z_t^{2,2},\bar{\tilde{z}}^1_{(t,e)}+\tilde{z}_{(t,e)}^{1,1}+\tilde{z}_{(t,e)}^{1,2},
        \bar{\tilde{z}}^2_{(t,e)}+\tilde{z}_{(t,e)}^{2,1}+\tilde{z}_{(t,e)}^{2,2},\\
        &\qquad\bar{\kappa}^1_t+\pi^1(t)\mathbbm{1}_{[\bar{t},\bar{t}+\epsilon]}+\kappa_t^{1,1'}+\kappa_t^{1,2},\bar{\kappa}^2_t+\pi^2(t)\mathbbm{1}_{[\bar{t},\bar{t}+\epsilon]}+\kappa_t^{2,1'}+\kappa_t^{2,2},\\
		&\qquad\bar{\tilde{\kappa}}^1_{(t,e)}+\tilde{\kappa}_{(t,e)}^{1,1}+\tilde{\kappa}_{(t,e)}^{1,2},\bar{\tilde{\kappa}}^2_{(t,e)}+\tilde{\kappa}_{(t,e)}^{2,1}+\tilde{\kappa}_{(t,e)}^{2,2},u^\epsilon_t\big)
        -l(t)+o(\epsilon)\\
		&=l_x(t)(x_t^1+x_t^2)+l_y(t)(y_t^1+y_t^2)+l_{z^i}(t)(z_t^{i,1'}+z_t^{i,2})+l_{\tilde{z}^i}(t)\int_{E_i}(\tilde{z}_{(t,e)}^{i,1}+\tilde{z}_{(t,e)}^{i,2})\nu_i(de)\\
		&\quad+l_{\kappa^i}(t)(\kappa_t^{i,1'}+\kappa_t^{i,2})+l_{\tilde{\kappa}^i}(t)\int_{E_i}(\tilde{\kappa}_{(t,e)}^{i,1}+\tilde{\kappa}_{(t,e)}^{i,2})\nu_i(de)+\frac{1}{2}\tilde{\Xi}(t)D^2l(t)\tilde{\Xi}(t)^\top\\
        &\quad +\delta l(t,\Delta^1,\Delta^2,\pi^1,\pi^2)\mathbbm{1}_{[\bar{t},\bar{t}+\epsilon]}+o(\epsilon),
	\end{aligned}
\end{equation*}
where
\begin{equation*}
\begin{aligned}
    \tilde{\Xi}&\coloneqq\bigg[x^1,y^1,z^{1,1'},z^{2,1'},\int_{E_1}\tilde{z}^{1,1}_{(t,e)}\nu_1(de),\int_{E_2}\tilde{z}^{2,1}_{(t,e)}\nu_2(de),\\
    &\qquad \kappa^{1,1'},\kappa^{2,1'},\int_{E_1}\tilde{\kappa}^{1,1}_{(t,e)}\nu_1(de),\int_{E_2}\tilde{\kappa}^{2,1}_{(t,e)}\nu_2(de)\bigg].
\end{aligned}
\end{equation*}

The first-order variation process 6-tuple $(x^1,y^1,z^{1,1},z^{2,1},\tilde{z}^{1,1},\tilde{z}^{2,1})$ satisfies:
\begin{equation}
	\left\{
	\begin{aligned}
		dx^1_t&=\Bigl\{b_x(t)x^1_t+b_y(t)y^1_t+b_{z^i}(t)(z^{i,1}_t-\Delta^i_t\mathbbm{1}_{[\bar{t},\bar{t}+\epsilon]}(t))+b_{\tilde{z}^1}(t)\int_{E_i}\tilde{z}^{i,1}_{(t,e)}\nu_i(de)\Bigr\}dt\\
		&\quad+\Bigl\{\sigma_{ix}(t)x^1_t+\sigma_{iy}(t)y^1_t+\delta \sigma_i(t)\mathbbm{1}_{[\bar{t},\bar{t}+\epsilon]}(t)\Bigr\}dW^i_t\\
		&\quad+\int_{E_i}\Bigl\{f_{ix}(t,e)x^1_{t-}+f_{iy}(t,e)y^1_{t-}\Bigr\}\tilde{N}_i(de,dt),\\
		-dy^1_t&=\Bigl\{g_x(t)x^1_t+g_y(t)y^1_t+g_{z^i}(t)(z^{i,1}_t-\Delta^i_t\mathbbm{1}_{[\bar{t},\bar{t}+\epsilon]}(t))+g_{\tilde{z}^i}(t)\int_{E_i}\tilde{z}^{i,1}_{(t,e)}\nu_i(de)\\
		&\qquad-n^i_t\delta\sigma_i(t)\mathbbm{1}_{[\bar{t},\bar{t}+\epsilon]}(t)\Bigr\}dt-z^{i,1}_tdW^i_t-\int_{E_i}\tilde{z}^{i,1}_{(t,e)}\tilde{N}_i(de,dt),\quad t\in[0,T],\\
		x^1_0&=0,\quad y^1_T=\phi_x(\bar{x}_T)x^1_T,
	\end{aligned}
	\right.
\end{equation}
where $(m,n^1,n^2,\tilde{n}^1,\tilde{n}^2)$ is the first-order decouple process quintuple given by
\begin{equation}\label{equation first order decouple m,n}
	\left\{
	\begin{aligned}
		-dm_t&=\biggl\{m_t\Big[b_x(t)+b_y(t)m_t+b_{z^i}(t)K_i(t)+b_{\tilde{z}^i}(t)\int_{E_i}\tilde{K}_i(t,e)\nu_i(de)\Big]\\
		&\qquad +n^i_t\big[\sigma_{ix}(t)+\sigma_{iy}(t)m_t\big]+\Big[g_x(t)+g_y(t)m_t+g_{z^i}(t)K_i(t)\\
        &\qquad +g_{\tilde{z}^i}(t)\int_{E_i}\tilde{K}_i(t,e)\nu_i(de)\Big]\\
		&\qquad +\int_{E_i}\tilde{n}^i_{(t,e)}\mathbf{E}_i\big[f_{ix}(t,e)+f_{iy}(t,e)m_{t-}|\mathcal{P}\otimes\mathcal{B}(E_i)\big]\nu_i(de)\biggr\}dt\\
        &\quad -n^i_tdW^i_t-\int_{E_i}\tilde{n}^i_{(t,e)}\tilde{N}_i(de,dt),\quad t\in[0,T],\\
		m_T&=\phi_x(\bar{x}_T).
	\end{aligned}
	\right.
\end{equation}

The following lemma is also from Wang et al. \cite{WSS2024}, Lemma 3.2.

\begin{lemma}\label{estimate first-order xyz}
	For first-order variation process 6-tuple $(x^1,y^1,z^{1,1},z^{2,1},\tilde{z}^{1,1},\tilde{z}^{2,1})$, we have the following relationships:
	\begin{equation}
		y^1_t=m_tx^1_t,\quad z^{i,1}_t=K_i(t)x^1_t+\Delta^i(t)\mathbbm{1}_{[\bar{t},\bar{t}+\epsilon]}(t),\quad \tilde{z}^{i,1}_{(t,e)}=\tilde{K}_i(t,e)x^1_{t-},
	\end{equation}
where $K_i(t)$, $\Delta^i(t)$ and $\tilde{K}_i(t,e)$ are given by
\begin{equation*}
	\left\{
	\begin{aligned}
		K_i(t)&\coloneqq m_t(\sigma_{ix}(t)+\sigma_{iy}(t)m_t)+n^i_t,\quad \Delta^i(t)\coloneqq m_t\delta\sigma_i(t),\\
		\tilde{K}_i(t,e)&\coloneqq m_{t-}(f_{ix}(t,e)+f_{iy}(t,e)m_{t-})+\tilde{n}^i_{(t,e)}+\tilde{n}^i_{(t,e)}(f_{ix}(t,e)+f_{iy}(t,e)m_{t-}).
	\end{aligned}
	\right.
\end{equation*}
Moreover, supposing that $(m,n^1,n^2,\tilde{n}^1,\tilde{n}^2)$ are bounded, then for $p\geqslant 2$, we have
	\begin{equation}
		\begin{aligned}
			&\mathbb{E}\bigg[\sup_{0\leqslant t\leqslant T}\big[|x^1_t|^p+|y^1_t|^p\big]+\Big(\int_0^T|z^{i,1}_t|^2dt\Big)^{\frac{p}{2}}
            +\Big(\int_0^T\int_{E_i}|\tilde{z}^{i,1}_{(t,e)}|^2N(de,dt)\Big)^{\frac{p}{2}}\bigg]=O(\epsilon^{\frac{p}{2}}),
		\end{aligned}
	\end{equation}	
	\begin{equation}
		\begin{aligned}
			&\mathbb{E}\bigg[\sup_{0\leqslant t\leqslant T}\Big(|x^{\epsilon}_t-\bar{x}_t-x^1_t|^2+|y^{\epsilon}_t-\bar{y}_t-y^1_t|^2\Big)+\int_0^T|z^{i,\epsilon}_t-\bar{z}^i_t-z^{i,1}_t|^2dt\\
			&\quad+\int_0^T\int_{E_i}|\tilde{z}^{i,\epsilon}_{(t,e)}-\bar{\tilde{z}}^i_{(t,e)}-\tilde{z}^{i,1}_{(t,e)}|^2N(de,dt)\bigg]=O(\epsilon^2),
		\end{aligned}
	\end{equation}	
	\begin{equation}
		\begin{aligned}
			&\mathbb{E}\bigg[\sup_{0\leqslant t\leqslant T}\Big(|x^{\epsilon}_t-\bar{x}_t-x^1_t|^4+|y^{\epsilon}_t-\bar{y}_t-y^1_t|^4\Big)+\Big(\int_0^T|z^{i,\epsilon}_t-\bar{z}^i_t-z^{i,1}_t|^2dt\Big)^2\\
			&\quad+\Big(\int_0^T\int_{E_i}|\tilde{z}^{i,\epsilon}_{(t,e)}-\bar{\tilde{z}}^i_{(t,e)}-\tilde{z}^{i,1}_{(t,e)}|^2N(de,dt)\Big)^2\bigg]=o(\epsilon^2).
		\end{aligned}
	\end{equation}
\end{lemma}

Next, we give the estimates about $x^2,y^2,z^{i,2},\tilde{z}^{i,2}$, whose proof is in the Appendix.
\begin{lemma}\label{lemma of x2 y2 z2 etc}
	Suppose Assumption \ref{assumption A1 in Zheng-Shi} holds. For $p\geqslant2$, we have the following estimates:
	\begin{equation*}
		\begin{aligned}
			&\mathbb{E}\bigg[\sup_{0\leqslant t\leqslant T}\Big(|x^2_t|^2+|y^2_t|^2\Big)+\int_0^T|z^{i,2}_t|^2dt+\int_0^T\int_{E_i}|\tilde{z}^{i,2}_{(t,e)}|^2\nu_i(de)dt\bigg]=o(\epsilon),\\
			&\mathbb{E}\bigg[\sup_{0\leqslant t\leqslant T}\Big(|x^2_t|^{p}+|y^2_t|^{p}\Big)+\Big(\int_0^T|z^{i,2}_t|^2dt\Big)^{\frac{p}{2}}
             +\Big(\int_0^T\int_{E_i}|\tilde{z}^{i,2}_{(t,e)}|^2\nu_i(de)dt\Big)^{\frac{p}{2}}\bigg]=o(\epsilon^{\frac{p}{2}}),\\
			&\mathbb{E}\bigg[\sup_{0\leqslant t\leqslant T}\Big(|x^{\epsilon}_t-\bar{x}_t-x^1_t-x^2_t|^2+|y^{\epsilon}_t-\bar{y}_t-y^1_t-y^2_t|^2\Big)+\int_0^T|z^{i,\epsilon}_t-\bar{z}^i_t-z^{i,1}_t-z^{i,2}_t|^2dt\\
			&\quad+\int_0^T\int_{E_i}|\tilde{z}^{i,\epsilon}_{(t,e)}-\bar{\tilde{z}}^i_{(t,e)}-\tilde{z}^{i,1}_{(t,e)}-\tilde{z}^{i,2}_{(t,e)}|^2\nu_i(de)dt\bigg]=o(\epsilon^2).
		\end{aligned}
	\end{equation*}
\end{lemma}

Now we consider the 5-tuple of first-order variation process  $(\zeta,\kappa^1,\kappa^2,\tilde{\kappa}^1,\tilde{\kappa}^2)$. We will give some properties which are similar to the above. However, due to the quadratic-exponential nature of state equation, some techniques we needed are new.

In order to obtain the first-order variational equation of $(\zeta,\kappa^1,\kappa^2,\tilde{\kappa}^1,\tilde{\kappa}^2)$, we introduce another first-order decouple equation:
\begin{equation}\label{equation first order decouple alpha beta}
	\left\{
	\begin{aligned}
		-d\alpha_t&=\biggl\{\alpha_t\Big[b_x(t)+b_y(t)m_t+b_{z^i}(t)K_i(t)+b_{\tilde{z}^i}(t)\int_{E_i}\tilde{K}_i(t,e)\nu_i(de)\Big]+\beta^i_t\big[\sigma_{ix}(t)+\sigma_{iy}(t)m_t\big]\\
		&\qquad+\int_{E_i}\tilde{\beta}^i_{(t,e)}\mathbf{E}_i\big[f_{ix}(t,e)+f_{iy}(t,e)m_t|\mathcal{P}\otimes\mathcal{B}(E_i)\big]\nu_i(de)+l_x(t)+l_y(t)m_t\\
		&\qquad+l_{z^i}(t)K_i(t)+l_{\tilde{z}^i}(t)\int_{E_i}\tilde{K}_i(t,e)\nu_i(de)+l_{\kappa^i}(t)K_i'(t)+l_{\tilde{\kappa}^i}(t)\int_{E_i}\tilde{K}'_i(t,e)\nu_i(de)\biggr\}dt\\
		&\quad-\beta^i_tdW^i_t-\int_{E_i}\tilde{\beta}^i_{(t,e)}\tilde{N}_i(de,dt),\quad t\in[0,T],\\
		\alpha_T&=\varphi_x(\bar{x}_T,\bar{y}_0),
	\end{aligned}
	\right.
\end{equation}
which, from Lemma \ref{lemma stochastic Lip BSDEP}, admits a unique solution $(\alpha,\beta^1,\beta^2,\tilde{\beta}^1,\tilde{\beta}^2)\in\bigcap_{p>1}\mathcal{N}^p$.

The first-order variational equation for the third equation in \eqref{equation state x y zeta in section 3} is given by
\begin{equation}\label{equation zeta1}
	\left\{
	\begin{aligned}
		-d\zeta^1_t&=\biggl\{l_x(t)x^1_t+l_y(t)y^1_t+l_{z^i}(t)(z^{i,1}_t-\Delta^i_t\mathbbm{1}_{[\bar{t},\bar{t}+\epsilon]}(t))+l_{\tilde{z}^i}(t)\int_{E_i}\tilde{z}^{i,1}_{(t,e)}\nu_i(de)\\
		&\qquad+l_{\kappa^i}(t)(\kappa^{i,1}_t-\pi^i_t\mathbbm{1}_{[\bar{t},\bar{t}+\epsilon]}(t))+l_{\tilde{\kappa}^i}(t)\int_{E_i}\tilde{\kappa}^{i,1}_{(t,e)}\nu_i(de)\\
		&\qquad-\beta^i_t\delta\sigma_i(t)\mathbbm{1}_{[\bar{t},\bar{t}+\epsilon]}(t)\biggr\}dt-\kappa^{i,1}_tdW^i_t-\int_{E_i}\tilde{\kappa}^{i,1}_{(t,e)}\tilde{N}_i(de,dt),\quad t\in[0,T],\\
		\zeta^1_T&=\varphi_x(\bar{x}_T,\bar{y}_0)x^1_T+\varphi_y(\bar{x}_T,\bar{y}_0)y^1_0.
	\end{aligned}
	\right.
\end{equation}
Applying It\^o's formula, we can obtain $\zeta^1=\alpha x^1$, $\kappa^{i,1'}=K'_i(t)x^1_t$, $\tilde{\kappa}^{i,1}=\tilde{K}'_i(t,e)x^1_{t-}$, and
\begin{equation}
	\begin{aligned}
		&K'_i(t)=\alpha_t(\sigma_{ix}(t)+\sigma_{iy}(t)m_t)+\beta^i_t,\quad\pi^i_t=\alpha_t\delta\sigma_i(t)\\
		&\tilde{K}'_i(t,e)=\alpha_{t-}[f_{ix}(t,e)+f_{iy}(t,e)m_{t-}]+\tilde{\beta}^i_{(t,e)}+\tilde{\beta}^i_{(t,e)}[f_{ix}(t,e)+f_{iy}(t,e)m_{t-}],
	\end{aligned}
\end{equation}

We then have the following result whose proof is in the Appendix.

\begin{lemma}\label{lemma estimate of zeta1 hat zeta2}
	Suppose Assumption \ref{assumption A1 in Zheng-Shi} and Assumption \ref{assumption of l(t)} hold, equation \eqref{equation zeta1} admits a unique solution $(\zeta^1,\kappa^{1,1},\kappa^{2,1},\tilde{\kappa}^{1,1},\tilde{\kappa}^{2,1})\in\bigcap_{p>1}\mathcal{N}^p$. Moreover, for all $p\geqslant2$,
	\begin{equation}\label{estimate of zeta1}
		\begin{aligned}
			&\mathbb{E}\bigg[\sup_{0\leqslant t\leqslant T}|\zeta^1_t|^p+\Big(\int_0^T|\kappa^{i,1}_t|^2dt\Big)^{\frac{p}{2}}
            +\Big(\int_0^T\int_{E_i}|\tilde{\kappa}^{i,1}_{(t,e)}|^2N_i(de,dt)\Big)^{\frac{p}{2}}\bigg]=O(\epsilon^{\frac{p}{2}}),
		\end{aligned}
	\end{equation}
	\begin{equation}
		\begin{aligned}
			&\mathbb{E}\bigg[\sup_{0\leqslant t\leqslant T}|\zeta^{\epsilon}_t-\bar{\zeta}_t-\zeta^1_t|^p+\Big(\int_0^T|\kappa^{i,\epsilon}_t-\bar{\kappa}^{i}_t-\kappa^{i,1}_t|^2dt\Big)^{\frac{p}{2}}\\
			&\quad+\Big(\int_0^T\int_{E_i}|\tilde{\kappa}^{i,\epsilon}_{(t,e)}-\bar{\tilde{\kappa}}^{i}_{(t,e)}-\tilde{\kappa}^{i,1}_{(t,e)}|^2N(de,dt)\Big)^{\frac{p}{2}}\bigg]=O(\epsilon^p).
		\end{aligned}
	\end{equation}
\end{lemma}

Noting \eqref{expansion guss} and inspired by \cite{ZS23}, we introduce the following variational equation of $(\zeta,\kappa^1,\kappa^2,\\\tilde{\kappa}^1,\tilde{\kappa}^2)$:
\begin{equation}\label{equation zeta1+zeta2}
	\left\{
	\begin{aligned}
		-d(\zeta^1_t+\zeta^2_t)&=\biggl\{l_x(t)(x^1_t+x^2_t)+l_y(t)(y^1_t+y^2_y)+l_{z^i}(t)(z^{i,1}_t-\Delta^i_t\mathbbm{1}_{E_\epsilon}(t)+z^{i,2}_t)\\
		&\qquad+l_{\tilde{z}^i}(t)\int_{E_i}\big(\tilde{z}^{i,1}_{(t,e)}+\tilde{z}^{i,2}_{(t,e)}\big)\nu_i(de)+l_{\kappa^i}(t)(\kappa^{i,1'}_t+\kappa^{i,2}_t)\\
        &\qquad+l_{\tilde{\kappa}^i}(t)\int_{E_i}\big(\tilde{\kappa}^{i,1}_{(t,e)}+\tilde{\kappa}^{i,2}_{(t,e)}\big)\nu_i(de)+\frac{1}{2}\tilde{\Xi}(t)D^2l(t)\tilde{\Xi}(t)^\top\\
		&\qquad+\delta l(t,\Delta^1,\Delta^2,\pi^1,\pi^2)\mathbbm{1}_{E_\epsilon}(t)\biggr\}dt\\
		&\quad-(\kappa^{i,1}_t+\kappa^{i,2}_t)dW^i_t-\int_{E_i}\big(\tilde{\kappa}^{i,1}_{(t,e)}+\tilde{\kappa}^{i,2}_{(t,e)}\big)\tilde{N}_i(de,dt),\quad t\in[0,T],\\
		\zeta^1_T+\zeta^2_T&=\varphi_x(\bar{x}_T,\bar{y}_0)(x^1_T+x^2_T)+\varphi_y(\bar{x}_T,\bar{y}_0)(y^1_0+y^2_0)+\frac{1}{2}D^2\varphi(\bar{x}_T,\bar{y}_0)[x^1_T,y^1_0]^2,
	\end{aligned}
	\right.
\end{equation}
where $E_{\epsilon}$ is defined in the proof of Lemma \ref{lemma estimate of zeta1 hat zeta2}. Then, the second order variational equation is given by
\begin{equation}\label{equation zeta2}
	\left\{
	\begin{aligned}
		-d\zeta^2_t&=\biggl\{l_x(t)x^2_t+l_y(t)y^2_t+l_{z^i}(t)z^{i,2}_t+l_{\tilde{z}^i}(t)\int_{E_i}\tilde{z}^{i,2}_{(t,e)}\nu_i(de)+l_{\kappa^i}(t)\kappa^{i,2}_t\\
		&\qquad+l_{\tilde{\kappa}^i}(t)\int_{E_i}\tilde{\kappa}^{i,2}_{(t,e)}\nu_i(de)+\frac{1}{2}\tilde{\Xi}(t)D^2l(t)\tilde{\Xi}(t)^\top\\
        &\qquad+[\delta l(t,\Delta^1,\Delta^2,\pi^1,\pi^2)+\beta^i_t\delta\sigma_i(t)]\mathbbm{1}_{E_\epsilon}(t)\biggr\}dt\\
		&\quad-\kappa^{i,2}_tdW^i_t-\int_{E_i}\tilde{\kappa}^{i,2}_{(t,e)}\tilde{N}_i(de,dt),\quad t\in[0,T],\\
		\zeta^2_T&=\varphi_x(\bar{x}_T,\bar{y}_0)x^2_T+\varphi_y(\bar{x}_T,\bar{y}_0)y^2_0+\frac{1}{2}D^2\varphi(\bar{x}_T,\bar{y}_0)[x^1_T,y^1_0]^2.
	\end{aligned}
	\right.
\end{equation}

\begin{lemma}\label{lemma estimate of zeta2,hat zeta 3}
	Suppose Assumption \ref{assumption A1 in Zheng-Shi} and Assumption \ref{assumption of l(t)} hold. Equation \eqref{equation zeta2} admits a unique solution
$(\zeta^2,\kappa^{1,2},\kappa^{2,2},\tilde{\kappa}^{1,2},\tilde{\kappa}^{2,2})\in\bigcap_{p>1}\mathcal{N}^p$. Moreover, for all $p\geqslant2$, we have the following estimates
\begin{equation}
	\mathbb{E}\bigg[\sup_{0\leqslant t\leqslant T}|\zeta^2_t|^2+\int_0^T|\kappa^{i,2}_t|^2dt+\int_0^T\int_{E_i}|\tilde{\kappa}^{i,2}_{(t,e)}|^2\nu_i(de)dt\bigg]=o(\epsilon),
\end{equation}
\begin{equation}
	\mathbb{E}\bigg[\sup_{0\leqslant t\leqslant T}|\zeta^2_t|^p+\Big(\int_0^T|\kappa^{i,2}_t|^2dt\Big)^{\frac{p}{2}}
    +\Big(\int_0^T\int_{E_i}|\tilde{\kappa}^{i,2}_{(t,e)}|^2\nu_i(de)dt\Big)^{\frac{p}{2}}\bigg]=o(\epsilon^{\frac{p}{2}}),
\end{equation}
\begin{equation}
	\begin{aligned}
		\mathbb{E}\bigg[&\sup_{0\leqslant t\leqslant T}|\zeta^{\epsilon}_t-\bar{\zeta}_t-\zeta^1_t-\zeta^2_t|^2+\int_0^T|\kappa^{i,\epsilon}_t-\bar{\kappa}^i_t-\kappa^{i,1}_t-\kappa^{i,2}_t|^2dt\\
		&+\int_0^T\int_{E_i}|\tilde{\kappa}^{i,\epsilon}_{(t,e)}-\bar{\tilde{\kappa}}^{i}_{(t,e)}-\tilde{\kappa}^{i,1}_{(t,e)}-\tilde{\kappa}^{1,2}_{(t,e)}|^2\nu_i(de)dt\bigg]=o(\epsilon^2).
	\end{aligned}
\end{equation}
\end{lemma}
The proof is left in the Appendix.

\subsection{Adjoint equations and optimal condition}

It follows from Lemma \ref{lemma estimate of zeta2,hat zeta 3} that the cost functional \eqref{cost functional zeta0} has the following expansion
\begin{equation}\label{expansion of cost functional zeta0}
	J(u^{\epsilon})-J(\bar{u})=\zeta^{\epsilon}_0-\bar{\zeta}_0=\zeta^1_0+\zeta^2_0+o(\epsilon).
\end{equation}
Then we introduce the adjoint equations for $x^1+x^2$, $y^1+y^2$, $\zeta^1+\zeta^2$ as:
\begin{equation}\label{equation adjoint r}
	\left\{
	\begin{aligned}
		dr_t&=r_tl_{\kappa^i}(t)dW^i_t+\int_{E_i}r_{t-}l_{\tilde{\kappa}^i}(t)\tilde{N}_i(de,dt),\quad t\in[0,T],\\
		r_0&=1,
	\end{aligned}
	\right.
\end{equation}
and
\begin{equation}\label{equation adjoint s p q}
	\left\{
	\begin{aligned}
		ds_t&=\Bigl\{r_tl_y(t)+s_tg_y(t)+p_tb_y(t)+q^i_t\sigma_{iy}(t)+\int_{E_i}\tilde{q}^i_{(t,e)}\mathbf{E}_i\big[f_{iy}(t,e)|\mathcal{P}\otimes\mathcal{B}(E_i)\big]\nu_i(de)\Bigr\}dt\\
		&\quad+\Bigl\{r_tl_{z^i}(t)+s_tg_{z^i}(t)+p_tb_{z^i}(t)\Bigr\}dW^i_t\\
		&\quad+\int_{E_i}\Bigl\{r_{t-}l_{\tilde{z}^i}(t)+s_{t-}g_{\tilde{z}^i}(t)+p_{t-}b_{\tilde{z}^i}(t)\Bigr\}\tilde{N}_i(de,dt),\\
		-dp_t&=\Bigl\{r_tl_x(t)+s_tg_x(t)+p_tb_x(t)+q^i_t\sigma_{ix}(t)+\int_{E_i}\tilde{q}^i_{(t,e)}\mathbf{E}_i\big[f_{ix}(t,e)|\mathcal{P}\otimes\mathcal{B}(E_i)\big]\nu_i(de)\Bigr\}dt\\
		&\quad-q^i_tdW^i_t-\int_{E_i}\tilde{q}^i_{(t,e)}\tilde{N}_i(de,dt),\quad t\in[0,T],\\
		s_0&=r_T\varphi_y(\bar{x}_T,\bar{y_0}),\quad p_T=r_T\varphi_x(\bar{x}_T,\bar{y}_0)+s_T\phi_x(\bar{x}_T).
	\end{aligned}
	\right.
\end{equation}
Applying It\^o's formula to $r_t(\zeta^1_t+\zeta^2_t)+s_t(y^1_t+y^2_t)-p_t(x^1_t+x^2_t)$, \eqref{expansion of cost functional zeta0} can be written as
\begin{equation}\label{cost functional del first order}
	\begin{aligned}		
        &J(u^\epsilon)-J(\bar{u})=\mathbb{E}\bigg[\frac{1}{2}(r_T\varphi_{xx}(\bar{x}_T,\bar{y}_0)+s_T\phi_{xx}(\bar{x}_T))(x^1_T)^2\bigg]
        +\mathbb{E}\bigg[\int_0^T\frac{1}{2}\Big(p_t\tilde{\Xi}_2(t)D^2b(t)\tilde{\Xi}_2(t)^\top\\
		&\ +q^i_t\tilde{\Xi}_3(t)D^2\sigma_i(t)\tilde{\Xi}_3(t)^\top+\int_{E_i}\tilde{q}^i_{(t,e)}\mathbf{E}_i\big[\tilde{\Xi}_3(t)D^2f_i(t,e)\tilde{\Xi}_3(t)^\top|\mathcal{P}\otimes\mathcal{B}(E_i)\big]\nu_i(de)\\
		&\ +s_t\tilde{\Xi}_2(t)D^2g(t)\tilde{\Xi}_2(t)^\top+r_t\tilde{\Xi}(t)D^2l(t)\tilde{\Xi}(t)^\top\Big)+q^i_t[\delta\sigma_{ix}(t)x^1_t+\delta\sigma_{iy}(t)y^1_t]\mathbbm{1}_{E_\epsilon}(t)\\
		&\ +\Bigl\{p_t[-b_{z^i}(t)m_t\delta\sigma_i(t)+\delta b(t,\Delta^1,\Delta^2)]+q^i_t\delta\sigma_i(t)+s_t[-g_{z^i}(t)m_t\delta\sigma_i(t)+\delta g(t,\Delta^1,\Delta^2)]\\
		&\ +r_t[-l_{z^i}(t)m_t\delta\sigma_i(t)-l_{\kappa^i}(t)\alpha_t\delta\sigma_i(t)+\delta l(t,\Delta^1,\Delta^2,\pi^1,\pi^2)]\Bigr\}\mathbbm{1}_{E_\epsilon}(t)dt\bigg]+o(\epsilon),
	\end{aligned}
\end{equation}
where
\begin{equation*}
	\begin{aligned}
		\tilde{\Xi}&\equiv\Big[x^1,y^1,z^{1,1'},z^{2,1'},\int_{E_1}\tilde{z}^{1,1}_{(t,e)}\nu_1(de),\int_{E_2}\tilde{z}^{2,1}_{(t,e)}\nu_2(de),\\
        &\qquad \kappa^{1,1'},\kappa^{2,1'},\int_{E_1}\tilde{\kappa}^{1,1}_{(t,e)}\nu_1(de),\int_{E_2}\tilde{\kappa}^{2,1}_{(t,e)}\nu_2(de)\Big],\\
		\tilde{\Xi}_2&\equiv\Big[x^1,y^1,z^{1,1'},z^{2,1'},\int_{E_1}\tilde{z}^{1,1}_{(t,e)}\nu_1(de),\int_{E_2}\tilde{z}^{2,1}_{(t,e)}\nu_2(de)\Big],\quad\tilde{\Xi}_3\equiv[x^1,y^1].
	\end{aligned}
\end{equation*}

We can check that
\begin{equation*}
	\mathbb{E}\bigg[\int_0^Tq^i_t\big[\delta\sigma_{ix}(t)x^1_t+\delta\sigma_{iy}(t)y^1_t\big]\mathbbm{1}_{E_\epsilon}(t)dt\bigg]=o(\epsilon).
\end{equation*}

We continue to give the adjoint equation of $(x^1_t)^2$ as
\begin{equation}\label{equation adjoint P Q}
	\left\{
	\begin{aligned}
		-dP_t&=\Bigl\{r_t\Theta_1(t)D^2l(t)\Theta_1(t)^\top+s_t\Theta_2(t)D^2g(t)\Theta_2(t)^\top+p_t\Theta_2(t)D^2b(t)\Theta_2(t)^\top\\
		&\qquad+q^i_t\Theta_3(t)D^2\sigma_i(t)\Theta_3(t)^\top+\int_{E_i}\tilde{q}^i_{(t,e)}\mathbf{E}_i\big[\Theta_3(t)D^2f_i(t,e)\Theta_3(t)^\top|\mathcal{P}\otimes\mathcal{B}(E_i)\big]\nu_i(de)\\
		&\qquad+2P_t[b_x(t)+b_y(t)m_t+b_{z^i}(t)K_i(t)+\int_{E_i}b_{\tilde{z}^i}(t)\tilde{K}_i(t,e)\nu_i(de)]+P_t[\sigma_x(t)\\
		&\qquad+\sigma_y(t)m_t]^2+2Q^i_t[\sigma_{ix}(t)+\sigma_{iy}(t)m_t]+\int_{E_i}P_t\mathbf{E}_i\big[[f_{ix}(t,e)\\
		&\qquad+f_{iy}(t,e)m_t]^2|\mathcal{P}\otimes\mathcal{B}(E_i)\big]\nu_i(de)+\int_{E_i}\tilde{Q}^i_{(t,e)}\mathbf{E}_i\big[2[f_{ix}(t,e)+f_{iy}(t,e)m_t]\\
        &\qquad+[f_{ix}(t,e)+f_{iy}(t,e)m_t]^2|\mathcal{P}\otimes\mathcal{B}(E_i)\big]\nu_i(de)\Bigr\}dt\\
		&\quad-Q^i_tdW^i_t-\int_{E_i}\tilde{Q}^i_{(t,e)}\tilde{N}_i(de,dt),\quad t\in[0,T],\\
		P_T&=r_T\varphi_{xx}(\bar{x}_T,\bar{y}_0)+s_T\phi_{xx}(\bar{x}_T),
	\end{aligned}
	\right.
\end{equation}
where
\begin{equation*}
	\begin{aligned}
		\Theta_1(t)&\coloneqq\Big[1,m,K_1,K_2,\int_{E_1}\tilde{K}_1(t,e)\nu_1(de),\int_{E_2}\tilde{K}_2(t,e)\nu_2(de),\\
        &\qquad K'_1,K'_2,\int_{E_1}\tilde{K}'_1(t,e)\nu_1(de),\int_{E_2}\tilde{K}'_2(t,e)\nu_2(de)\Big],\\
		\Theta_2(t)&\coloneqq\Big[1,m,K_1,K_2,\int_{E_1}\tilde{K}_1(t,e)\nu_1(de),\int_{E_2}\tilde{K}_2(t,e)\nu_2(de)\Big],\quad\Theta_3(t)\coloneqq[1,m].
	\end{aligned}
\end{equation*}

Applying It\^o's formula to $P_t(x^1_t)^2$ and combining \eqref{cost functional del first order}, we educe the variational inequality:
\begin{equation}\label{equation variation inequality}
	\begin{aligned}
		0&\leqslant J(u^{\epsilon})-J(\bar{u})\\
		&=\mathbb{E}\int_0^T\Bigl\{\frac{1}{2}P_t(\delta\sigma_i(t))^2+p_t\big[-b_{z^i}(t)m_t\delta\sigma_i(t)+\delta b(t,\Delta^1,\Delta^2)\big]+q^i_t\delta\sigma_i(t)\\
		&\qquad\qquad+s_t\big[-g_{z^i}(t)m_t\delta\sigma_i(t)+\delta g(t,\Delta^1,\Delta^2)\big]+r_t\big[-l_{z^i}(t)m_t\delta\sigma_i(t)\\
		&\qquad\qquad-l_{\kappa^i}(t)\alpha_t\delta\sigma_i(t)+\delta l(t,\Delta^1,\Delta^2,\pi^1,\pi^2)\big]\Bigr\}\mathbbm{1}_{E_\epsilon}(t)dt+o(\epsilon).
	\end{aligned}
\end{equation}

Define the Hamiltonian function
\begin{equation}\label{Hamiltonian function}
	\begin{aligned}
		\mathcal{H}&\big(t,x,y,z^1,z^2,\tilde{z}^1,\tilde{z}^2,\zeta,\kappa^1,\kappa^2,\tilde{\kappa}^1,\tilde{\kappa}^2,m,\alpha,p,q^1,q^2,s,r,P,u\big)\\
		&=\frac{1}{2}P_t\sum_{i=1}^2\big[\sigma_i(t,x,y,u)-\sigma_i(t,\bar{x}_t,\bar{y}_t,\bar{u}_t)\big]^2+p_tb(t,x,y,z^1+\Delta^1,z^2+\Delta^2,\tilde{z}^1,\tilde{z}^2,u)\\
		&\quad+\sum_{i=1}^2\big[q^i_t-(p_tb_{z^i}(t)+s_tg_{z^i}(t)-r_tl_{z^i}(t))m_t-r_tl_{\kappa^i}(t)\alpha_t\big]\sigma_i(t,x,y,u)\\
		&\quad+s_tg(t,x,y,z^1+\Delta^1,z^2+\Delta^2,\tilde{z}^1,\tilde{z}^2,u)\\
		&\quad+r_tl(t,x,y,z^1+\Delta^1,z^2+\Delta^2,\tilde{z}^1,\tilde{z}^2,\kappa^1+\Delta^1,\kappa^2+\Delta^2,\tilde{\kappa}^1,\tilde{\kappa}^2,u),
	\end{aligned}
\end{equation}
then the variational inequality \eqref{equation variation inequality} is equivalent to
\begin{equation}\label{equation variation inequality 2}
	\begin{aligned}
		0&\leqslant\mathbb{E}\bigg[\int_0^T\Bigl\{\mathcal{H}\big(t,\bar{x},\bar{y},\bar{z}^1,\bar{z}^2,\bar{\tilde{z}}^1,\bar{\tilde{z}}^2,\bar{\zeta},\bar{\kappa}^1,\bar{\kappa}^2,
        \bar{\tilde{\kappa}}^1,\bar{\tilde{\kappa}}^2, m,\alpha,p,q^1,q^2,s,r,P,u\big)\\
		&\qquad\qquad-\mathcal{H}(t)\Bigr\}\mathbbm{1}_{E_\epsilon}(t)dt\bigg]+o(\epsilon),
	\end{aligned}
\end{equation}
where $\mathcal{H}(t)\equiv\mathcal{H}\big(t,\bar{x},\bar{y},\bar{z}^1,\bar{z}^2,\bar{\tilde{z}}^1,\bar{\tilde{z}}^2,\bar{\zeta},\bar{\kappa}^1,\bar{\kappa}^2,\bar{\tilde{\kappa}}^1,
\bar{\tilde{\kappa}}^2,m,\alpha,p,q^1,q^2,s,r,P,\bar{u}\big)$.

Recalling the definition $E_\epsilon=[\bar{t},\bar{t}+\epsilon]\cap\Gamma_M$, it allows us to conclude from \eqref{equation variation inequality 2} that
\begin{equation*}
	\begin{aligned}
		&\mathbb{E}\Big[\Big(\mathcal{H}(t,\bar{x},\bar{y},\bar{z}^1,\bar{z}^2,\bar{\tilde{z}}^1,\bar{\tilde{z}}^2,\bar{\zeta},\bar{\kappa}^1,\bar{\kappa}^2,
        \bar{\tilde{\kappa}}^1,\bar{\tilde{\kappa}}^2,m,\alpha,p,q^1,q^2,s,r,P,u)\\
		&\qquad-\mathcal{H}(t)\Big)\mathbbm{1}_{\Gamma_M}(t)\Big\vert\mathcal{F}^{Y}_t\Big]\geqslant0,\quad a.e.\,\,t\in[0,T],\quad\mathbb{P}\mbox{-}a.s.,
	\end{aligned}
\end{equation*}
for all $M\geqslant1$. Noting that $\mathbbm{1}_{\cup_{M\geqslant1}\Gamma_M}(t)=1$, $a.e.\,t\in[0,T]$, we finally deduce the following optimal condition
\begin{equation}\label{equation optimal condition}
	\begin{aligned}
		&\mathbb{E}\Big[\mathcal{H}(t,\bar{x},\bar{y},\bar{z}^1,\bar{z}^2,\bar{\tilde{z}}^1,\bar{\tilde{z}}^2,\bar{\zeta},\bar{\kappa}^1,\bar{\kappa}^2,\bar{\tilde{\kappa}}^1,\bar{\tilde{\kappa}}^2,m,\alpha,p,q^1,q^2,s,r,P,u)\\
		&\quad-\mathcal{H}(t)\Big\vert\mathcal{F}^{Y}_t\Big]\geqslant0,\quad a.e.\,\,t\in[0,T],\quad\mathbb{P}\mbox{-}a.s..
	\end{aligned}
\end{equation}

The main result in this paper is the following theorem.

\begin{theorem}\label{theorem SMP}
	Suppose Assumptions \ref{assumption A1 in Zheng-Shi}, \ref{assumption A2 in Zheng-Shi}, \ref{assumption A3 in Zheng-Shi} and \ref{assumption of l(t)} hold. Let $\bar{u}$ be an optimal control, and $(\bar{x},\bar{y},\bar{z}^1,\bar{z}^2,\bar{\tilde{z}}^1,\bar{\tilde{z}}^2,\bar{\kappa}^1,\bar{\kappa}^2,\bar{\tilde{\kappa}}^1,\bar{\tilde{\kappa}}^2)$ be the corresponding solution to FBSDEP \eqref{equation state x y zeta in section 3}. Then the optimal condition \eqref{equation optimal condition} holds, where $(m,n^1,n^2,\tilde{n}^1,\tilde{n}^2)$ satisfies \eqref{equation first order decouple m,n}, $(\alpha,\beta^1,\beta^2,\tilde{\beta}^1,\tilde{\beta}^2)$ satisfies \eqref{equation first order decouple alpha beta}, $r$ satisfies \eqref{equation adjoint r}, $(s,p,q^1,q^2,\tilde{q}^1,\tilde{q}^2)$ satisfies \eqref{equation adjoint s p q} and $(P,Q^1,Q^2,\tilde{Q}^1,\tilde{Q}^2)$ satisfies \eqref{equation adjoint P Q}.
\end{theorem}

\begin{remark}
	In the controlled system \eqref{equation state x y zeta in section 3} and \eqref{cost functional zeta0}, if we set the jump-related terms to $0$ as well as $y$, $z^i$, $\tilde{z}^i$ and $g$, further, denote the Hamiltonian function in (3.61) of \cite{HJX22} by $H(\cdot)$, then Theorem \ref{theorem SMP} degenerates into Theorem 3.16 of \cite{HJX22}. Actually, it follows from It\^o's formula that $\mathcal{H}(\cdot)=rH(\cdot)$, where $\mathcal{H}(\cdot)$ is defined in \eqref{Hamiltonian function} and $r$ in \eqref{equation adjoint r}.
\end{remark}

\section{Applications to an optimal investment model}

In this section, we consider a risk-sensitive optimal investment model in which the goal is to maximize the exponential utility of wealth. In this model, the mean return of the stock is explicitly affected by the underlying economic factor (see Fleming and Sheu \cite{FS00}, Nagai \cite{Nagai01}, Davis and Lleo \cite{DL13}).

Let $(\Omega,\mathcal{F},\mathbb{F},\bar{\mathbb{P}})$ be a fixed complete filtered probability space, on which is defined two independent processes: a $\mathbb{R}^{2}$-valued standard Brownian motion $W_t\equiv(W^1_t,W^2_t)$ and $1$-dimensional Poisson random measure $N(de,dt)$ with a stationary compensator $\nu(de)dt$, where $\tilde{N}(de,dt)\coloneqq N(de,dt)-\nu(de)dt$ is the compensated martingale measure.
In the financial market, suppose that there is a $1$-dimensional stock whose price $S$ is affected by a $1$-dim factor process $X$. Their dynamics are given by
\begin{equation}\label{equation securities}
	\frac{dS_t}{S_{t-}}=(a_1+A_1X_t)dt+\sigma dW_t+\int_E\tilde{\sigma}\tilde{N}(de,dt),\quad S_0=s\in\mathbb{R},
\end{equation}
and
\begin{equation}\label{equation factors}
	dX_t=(a_2+A_2X_t)dt+\Lambda dW_t+\int_{E}\tilde{\Lambda}\tilde{N}(de,dt),\quad X_0=x\in\mathbb{R},
\end{equation}
where $a_1,a_2,A_1,A_2,\tilde{\sigma},\tilde{\Lambda}\in\mathbb{R}$ are constants, $\sigma\equiv[\sigma_1,\sigma_2],\Lambda\equiv[\Lambda_1,\Lambda_2]$ are $1\times2$ constant vectors with $\sigma\sigma^\top>0$.

Let $u$ be the investment strategy on the stock $S$, for some investor. The observable information until time $t\geq0$ for him/her is $\mathcal{F}^S_t:=\sigma(S_r;r\leqslant t)$. Then the log stock price $Y_t=\ln S_t$ satisfies:
\begin{equation}
		dY_t=(a_3+A_1X_t)dt+\sigma dW_t+\int_E\ln(1+\tilde{\sigma})\tilde{N}(de,dt),\quad Y_0=0,
\end{equation}
where $a_3\coloneqq a_1-\frac{1}{2}\sigma\sigma^\top+\int_E[\ln(1+\tilde{\sigma})-\tilde{\sigma}]\nu(de)$. The admissible strategy set is defined as
\begin{equation*}
	\mathcal{U}_{ad}=\left\{u:[0,T]\times\Omega\to\mathbb{R}_+ \big| u\mbox{ is $\mathcal{F}^S$-predictable} \right\}.
\end{equation*}
For each strategy $u\in\mathcal{U}_{ad}$, one can define the investor's wealth process $V$ by
\begin{equation}
		\frac{dV_t}{V_{t-}}=u_t\left[(a_1+A_1X_t)dt+\sigma dW_t+\int_E\tilde{\sigma}\tilde{N}(de,dt)\right],\quad V(0)=v.
\end{equation}
We consider the risk-averse investor, that is $\theta>0$, whose target is to minimize the risk-sensitive utility
\begin{equation}
	J'(u)=\bar{\mathbb{E}}\big[e^{-\theta\ln V_T}\big].
\end{equation}

As in Section 3, by introducing
\begin{equation*}
	\begin{aligned}
		\Gamma_T:=\exp\left\{-\int_0^T(\sigma\sigma^\top)^{-1}\sigma(a_3+A_1X_t)dW_t-\frac{1}{2}\int_0^T(\sigma\sigma^\top)^{-1}(a_3+A_1X_t)^2dt\right\},
	\end{aligned}
\end{equation*}
we can define a new probability $\mathbb{P}$ by
\begin{equation*}
	\frac{d\mathbb{P}}{d\bar{\mathbb{P}}}=\Gamma_T,
\end{equation*}
under which $d\hat{W}_t\coloneqq dW_t+\sigma^\top(\sigma\sigma^\top)^{-1}(a_3+A_1X_t)dt$ is a new Brownian motion and $\tilde{N}$ is a Poisson martingale measure.

Setting $a_4:=a_2-\Lambda\sigma^\top(\sigma\sigma^\top)^{-1}a_3$ and $A_4:=A_2-\Lambda\sigma^\top(\sigma\sigma^\top)^{-1}A_1$, aforementioned problem is equal to minimize $J(u)=\zeta_0$, with respect to
\begin{equation}\label{equation state factor invest}
	\left\{
	\begin{aligned}
		dX_t&=(a_4+A_4X_t)dt+\Lambda d\hat{W}_t+\int_E\tilde{\Lambda}\tilde{N}(de,dt),\\
		\frac{dV_t}{V_{t-}}&=u_t\left[(a_1-a_3)dt+\sigma d\hat{W}_t+\int_E\tilde{\sigma}\tilde{N}(de,dt)\right],\\
		-d\zeta_t&=\left\{\frac{\theta}{2}(\kappa_t)^2+\frac{1}{\theta}\int_E\big(e^{\theta\tilde{\kappa}_{(t,e)}}-\theta\tilde{\kappa}_{(t,e)}-1\big)\nu(de)+\sigma^\top(\sigma\sigma^\top)^{-1}(a_3+A_1X_t)\kappa_t\right\}dt\\
		&\quad-\kappa_td\hat{W}_t-\int_E\tilde{\kappa}_{(t,e)}\tilde{N}(de,dt),\\
		X_0&=x,\quad V_0=v,\quad\zeta_T=-\ln V_T.
	\end{aligned}
	\right.
\end{equation}

As an application of Theorem \ref{theorem SMP}, the optimal condition is
\begin{equation}\label{optimal condition in Application}
	\mathbb{E}\Big[p_t(a_1-a_3)\bar{V}_t+\left\{q_t-r_t\alpha_t\left[\theta\bar{\kappa}_t+\sigma^\top(\sigma\sigma^\top)^{-1}(a_3+A_1\bar{X}_t)\right]\right\}^\top\sigma \bar{V}_t\Big|\mathcal{F}^S_t\Big]=0,
\end{equation}
$a.e.\,\,t\in[0,T],\mathbb{P}\mbox{-}a.s.$, where $r$, $p$ and $\alpha$ are first-order adjoint processes given by
\begin{equation}
	\left\{
	\begin{aligned}
		dr_t&=r_t\left[\theta\bar{\kappa}_t+\sigma^\top(\sigma\sigma^\top)^{-1}(a_3+A_1\bar{X}_t)\right]d\hat{W}_t+\int_Er_{t-}\big(e^{\theta\bar{\tilde{\kappa}}_{(t,e)}}-1\big)\tilde{N}(de,dt),\\
		r_0&=1,
	\end{aligned}
	\right.
\end{equation}
\begin{equation}
	\left\{
	\begin{aligned}
		-dp_t&=\left[(a_1-a_3)\bar{u}_tp_t+\sigma\bar{u}_tq_t+\int_E\tilde{\sigma}\bar{u}_t\tilde{q}_{(t,e)}\nu(de)\right]dt-q_td\hat{W}_t-\int_E\tilde{q}_{(t,e)}\tilde{N}(de,dt),\\
		p_T&=-\frac{r_T}{\bar{V}_T},
	\end{aligned}
	\right.
\end{equation}
and
\begin{equation}
	\left\{
	\begin{aligned}
		-d\alpha_t&=\left\{(a_1-a_3)\bar{u}_t\alpha_t+\sigma\bar{u}_t\beta_t+\int_E\tilde{\sigma}\bar{u}_t\tilde{\beta}_{(t,e)}\nu(de)\right.\\
		&\qquad +(\beta_t+\alpha_t\sigma\bar{u}_t)\left[\theta\bar{\kappa}_t+\sigma^\top(\sigma\sigma^\top)^{-1}(a_3+A_1\bar{X}_t)\right]\\
		&\qquad \left.+\int_E\big(e^{\theta\bar{\tilde{\kappa}}_{(t,e)}}-1\big)\big[\tilde{\sigma}\bar{u}_t(\alpha_t+\tilde{\beta}_{(t,e)})+\tilde{\beta}_{(t,e)}\big]\nu(de)\right\}dt\\
        &\quad -\beta_t d\hat{W}_t-\int_E\tilde{\beta}_{(t,e)}\tilde{N}(de,dt),\\
		\alpha_T&=-\frac{1}{\bar{V}_T},
	\end{aligned}
	\right.
\end{equation}
respectively, with optimal $\bar{X},\bar{V},\bar{\kappa},\bar{\tilde{\kappa}}$ and $\bar{u}$.

By It\^{o}'s formula, the following relations
\begin{equation}\label{relation of adjoint in Application}
	p_t=\alpha_t r_t,\quad q_t=\left[\theta\bar{\kappa}_t+\sigma^\top(\sigma\sigma^\top)^{-1}(a_3+A_1\bar{X}_t)\right]\alpha_t r_t+\beta_t r_t
\end{equation}
can be verified. Combining \eqref{relation of adjoint in Application}, the optimal condition \eqref{optimal condition in Application} can be rewritten as
\begin{equation}\label{equation optimal condition transed}
	\begin{aligned}
		\mathbb{E}\big[\big\{\alpha_t(a_1-a_3)+\beta_t\sigma\big\}r_t \bar{V}_t\big|\mathcal{F}^S_t\big]=0,\quad a.e.\,\,t\in[0,T],\quad\mathbb{P}\mbox{-}a.s..
	\end{aligned}
\end{equation}

To solve $\alpha$, we introduce process $\eta$ by
\begin{equation}
	\left\{
	\begin{aligned}
		d\eta_t&=\mathcal{B}_1(t)\eta_tdt+\mathcal{B}_2(t)\eta_td\hat{W}_t+\int_E\mathcal{B}_3(t-,e)\eta_{t-}\tilde{N}(de,dt),\\
		\eta_0&=1,
	\end{aligned}
	\right.
\end{equation}
where
\begin{equation*}
	\left\{
	\begin{aligned}
		\mathcal{B}_1(t)&:=(a_1-a_3)\bar{u}_t+\int_E\big(e^{\theta\bar{\tilde{\kappa}}_{(t,e)}}-1\big)\tilde{\sigma}\bar{u}_t\nu(de),\\
		\mathcal{B}_2(t)&:=\sigma\bar{u}_t+\theta\bar{\kappa}_t+\sigma^\top(\sigma\sigma^\top)^{-1}(a_3+A_1\bar{X}_t),\\
		\mathcal{B}_3(t,e)&:=\tilde{\sigma}\bar{u}_t+\big(e^{\theta\bar{\tilde{\kappa}}_{(t,e)}}-1\big)(\tilde{\sigma}\bar{u}_t+1).
	\end{aligned}
	\right.
\end{equation*}
In fact, $\eta$ has the following explicit form:
\begin{equation}
	\begin{aligned}
		\eta_t=&\exp\left\{\int_0^t\left\{\mathcal{B}_1(s)-\frac{1}{2}\mathcal{B}_2(s)+\int_E\big[\ln(1+\mathcal{B}_3(s-,e))-\mathcal{B}_3(s-,e)\big]\right\}\right.dt\\
		&\qquad \left.+\int_0^t\mathcal{B}_2(s)d\hat{W}_s+\int_0^t\int_E\ln(1+\mathcal{B}_3(s-,e))\tilde{N}(de,ds)\right\}.
	\end{aligned}
\end{equation}
Then, it follows from It\^o's formula of $\eta_t\alpha_t$ that $\alpha$ is given by
\begin{equation}\label{equation alpha explicit}
		\alpha_t=\eta_t^{-1}\mathbb{E}\Big[-\frac{\eta_T}{\bar{V}_T}\,\Big|\mathcal{F}_t\Big].
\end{equation}

\begin{remark}
	$\alpha$ has an implicit dependence on the optimal strategy $\bar{u}$ according to \eqref{equation alpha explicit}. Optimal condition \eqref{equation optimal condition transed} can also be regarded as a constraint of $\bar{u}$. However, caused by the non-linear feature of \eqref{equation state factor invest}, it is rather difficult to get the explicit representation of $\bar{u}$ from \eqref{equation optimal condition transed}.
\end{remark}

\section{Risk-sensitive filtering and modified Zakai equation}

In this section, we study the risk-sensitive stochastic filtering problem, which involves both Brownian and Poissonian correlated noises. In fact, a typical approach in literature to tackle stochastic optimal control problems with partial observation is to separate the control and estimation tasks. However, it is well known that this separation generally fails to deliver an actual optimal solution for risk-sensitive stochastic optimal control problems. We deal with the risk-sensitive stochastic filtering problem of a general controlled jump-diffusion process $x$, when the observation process $Y$ is a correlated jump-diffusion process that has common jump times with $x$. The central goal is to characterize the conditional distribution of $x$ with respect to given observations $\mathcal{F}^Y_t\coloneqq\sigma\{Y_s,0\leqslant s\leqslant t\}$. The main result in this section can be regarded as a partial generalization of Nagai \cite{Nagai01} under the Poisson jump formulation and an extension of Germ and Gy{\"o}ngy \cite{GG22} to the risk-sensitive filtering.

Let $x$ be the unobservable controlled jump-diffusion process and $Y$ be the observable component. Under the same formulation of Section \ref{section Problem formulation and preliminaries}, we consider the following partially observed system driven by
\begin{equation}
	\left\{
	\begin{aligned}
		dx_t&=b_1(t,x_t,u_t)dt+\sigma_1(t,x_t,u_t)dW^1_t+\sigma_2(t,x_t,u_t)d\tilde{W}^2_t\\
		&\quad+\int_{E_1}f_1(t,x_{t-},u_t,e)\tilde{N}_1(de,dt)+\int_{E_2}f_2(t,x_{t-},u_t,e)\tilde{N}'_2(de,dt),\\
		x_0&=x,
	\end{aligned}
	\right.
\end{equation}
and
\begin{equation}
	\left\{
	\begin{aligned}
		dY_t&=b_2(t,x_t,u_t)dt+\sigma_3(t)d\tilde{W}^2_t+\int_{E_2}f_3(t,e)\tilde{N}'_2(de,dt),\\
		Y_0&=0,
	\end{aligned}
	\right.
\end{equation}
where, under probability measure $\bar{\mathbb{P}}$,  $W^1$ and $\tilde{W}^2$ are two one-dimensional independent standard Brownian motions, and $\tilde{N}_1(de,dt)=N_1(de,dt)-\nu_1(de)dt$ and $\tilde{N}^{\prime}_2(de,dt)=N_2(de,dt)-\lambda(t,x_{t-},e)\nu_2(de)dt$ are compensated Poisson random measures. In the above, the admissible control
\begin{equation*}
	\begin{aligned}
		u\in\mathcal{U}_{ad}[0,T]\coloneqq&\bigg\{u\Big|u_t\mbox{ is }\mathcal{F}_t^Y\mbox{-progressive }U\mbox{-valued process, such that }\sup_{0\leqslant t\leqslant T}\bar{\mathbb{E}}|u_t|^p<\infty,\\
		&\mbox{ for any $p>1$ }\mbox{and }\bar{\mathbb{E}}\int_0^T|u_t|^2N_i(E_i,dt)<\infty,\mbox{\ for\ } i=1,2\bigg\}.
	\end{aligned}
\end{equation*}
We consider the risk-sensitive cost functional
\begin{equation}\label{cost in jump filter origin}
	J(u)=\bar{\mathbb{E}}\left[\exp\left\{\theta\int_0^Tl(t,x_t,u_t)dt+\theta\varphi(x_T)\right\}\right].
\end{equation}

Note that the above problem is a special case of (\ref{state equation x y}), (\ref{observation}) and (\ref{cf1}), without the backward components. As in Section \ref{section Problem formulation and preliminaries}, by introducing
\begin{equation}
	\left\{
	\begin{aligned}
		d\tilde{\Gamma}_t&=\tilde{\Gamma}_t\sigma_3^{-1}(t)b_2(t,x_t,u_t)dW^2_t+\int_{E_2}\tilde{\Gamma}_{t-}(\lambda(t,x_{t-},e)-1)\tilde{N}_2(de,dt),\\
		\Gamma_0&=1,
	\end{aligned}
	\right.
\end{equation}
where
\begin{equation}
	\left\{
	\begin{aligned}
		&dW^2_t=d\tilde{W}^2_t+\sigma_3^{-1}(t)b_2(t,x_t,u_t)dt,\\
		&\tilde{N}_2(de,dt)=\tilde{N}'_2(de,dt)+(\lambda(t,x_{t-},e)-1)\nu_2(de)dt,
	\end{aligned}
	\right.
\end{equation}
we can define a new probability measure $\mathbb{P}$ through
\begin{equation}
	\frac{d\bar{\mathbb{P}}}{d\mathbb{P}}\Big|_{\mathcal{F}_t}=\tilde{\Gamma}_t.
\end{equation}

Under $\mathbb{P}$, $W^1$, $W^2$ are independent standard Brownian motions, and $\tilde{N}_1$, $\tilde{N}_2$ are compensated Poisson random measures. Notice that cost functional \eqref{cost in jump filter origin} can be rewritten as
\begin{equation}
	\begin{aligned}
		J(u)&=\mathbb{E}\left[\tilde{\Gamma}_Te^{\theta\int_0^Tl(t,x_t,u_t)dt}e^{\theta\varphi(x_T)}\right]\\
		&=\mathbb{E}\left[\mathbb{E}\left[\tilde{\Gamma}_Te^{\theta\int_0^Tl(t,x_t,u_t)dt}e^{\theta\varphi(x_T)}\,\big|\mathcal{F}^Y_T\right]\right].
	\end{aligned}
\end{equation}

For any $F(\cdot)\in C^2_b(\mathbb{R}^d)$, and $u\in\mathcal{U}_{ad}[0,T]$, consider the unnormalized conditional probability
\begin{equation}
	\mu^u_t(F):=\mathbb{E}\left[\tilde{\Gamma}_te^{\theta\int_0^tl(s,x_s,u_s)ds}F(x_t)\,\big|\mathcal{F}^Y_t\right],
\end{equation}
then we have
\begin{equation}
	J(u)=\mathbb{E}\big[\mu^u_T(e^{\theta\varphi})\big].
\end{equation}

Then we have the following result.
\begin{theorem}(Modified Zakai equation)\label{theorem zakai}
	For any $F(\cdot)\in C^2_b(\mathbb{R}^d)$ and $u\in\mathcal{U}_{ad}[0,T]$, $\mu^u_t(F)$ defined above satisfies the following SPDE:
\begin{equation}\label{equation zakai}
	\begin{aligned}
		\mu^u_t(F)&=\mu^u_0(F)+\int_0^t\mu^u_s(\mathcal{L}_sF)ds+\int_0^t\mu^u_s(\mathcal{M}_sF)dW^2_s+\int_0^t\int_{E_2}\mu^u_{s-}\big(\mathcal{I}^{\lambda,f_2}_{s-}F\big)\tilde{N}_2(de,ds)\\
		&\quad +\int_0^t\int_{E_1}\mu^u_s\big(\mathcal{A}^{1,f_1}_sF\big)\nu_1(de)ds+\int_0^t\int_{E_2}\mu^u_s\big(\mathcal{A}^{\lambda,f_2}_sF\big)\nu_2(de)ds,
	\end{aligned}
\end{equation}
where
\begin{equation*}
	\begin{aligned}
		\mathcal{L}_tF&\coloneqq Fl+\frac{\partial F}{\partial x}b_1+\frac{1}{2}\frac{\partial^2 F}{\partial x^2}(\sigma^2_1+\sigma^2_2),\quad\mathcal{M}_tF\coloneqq \sigma_3^{-1}b_2F+\sigma_2\frac{\partial F}{\partial x},\\
		\mathcal{A}^{\xi,f}_tF&\coloneqq \xi\left[F(\cdot+f)-F(\cdot)-\frac{\partial F}{\partial x}f\right],\quad \mathcal{I}^{\xi,f}_tF\coloneqq \xi\left[F(\cdot+f)-F(\cdot)\right]+(\xi-1)F.
	\end{aligned}
\end{equation*}
\end{theorem}

\begin{proof}
For any $F(\cdot)\in C^2_b(\mathbb{R}^d)$, by using It\^o's formula and taking the $\mathcal{F}^Y_t$-conditional expectation on both sides, we can deduce that
\begin{equation*}
	\begin{aligned}	
        &\mathbb{E}\left[\tilde{\Gamma}_te^{\theta\int_0^tl(s,x_s,u_s)ds}F(x_t)\,\big|\mathcal{F}^Y_t\right]=\mathbb{E}\left[\tilde{\Gamma}_0F(x_0)\,\big|\mathcal{F}^Y_t\right]\\
        &+\mathbb{E}\left[\int_0^t\tilde{\Gamma}_se^{\theta\int_0^sl(r,x_r,u_r)dr}\left\{\frac{\partial F(x_s)}{\partial x}b_1(s)
         +\frac{1}{2}\frac{\partial^2 F(x_s)}{\partial x^2}(\sigma_1^2(s)+\sigma_2^2(s))\right.+F(x_s)l(s,x_s,u_s)\}ds\,\bigg|\mathcal{F}^Y_t\right]\\
		&+\mathbb{E}\left[\int_0^t\tilde{\Gamma}_se^{\theta\int_0^sl(r,x_r,u_r)dr}\left\{\sigma_2(s)\frac{\partial F(x_s)}{\partial x}+\sigma_3^{-1}(s)b_2(s)F(x_s)\right\}dW^2_s\,\bigg|\mathcal{F}^Y_t\right]\\
    \end{aligned}
\end{equation*}	
\begin{equation*}
	\begin{aligned}		
        &+\mathbb{E}\left[\int_0^t\int_{E_2}\tilde{\Gamma}_{s-}e^{\theta\int_0^{s-}l(r,x_r,u_r)dr}\Big\{\lambda(s-)\big[F(x_{s-}+f_2(s,e))-F(x_{s-})\big]\right.\\
        &\qquad\qquad +(\lambda(s-)-1)F(x_{s-})\Big\}\tilde{N}_2(de,ds)\,\bigg|\mathcal{F}^Y_t\bigg]\\
		&+\mathbb{E}\left[\int_0^t\int_{E_1}\tilde{\Gamma}_se^{\theta\int_0^sl(r,x_r,u_r)dr}\left\{F(x_s+f_1(s,e))-F(x_s)-\frac{\partial F(x_s)}{\partial x}f_1(s,e)\right\}\nu_1(de)ds\,\bigg|\mathcal{F}^Y_t\right]\\
		&+\mathbb{E}\left[\int_0^t\int_{E_2}\tilde{\Gamma}_se^{\theta\int_0^sl(r,x_r,u_r)dr}\lambda(s)\left\{F(x_s+f_2(s))-F(x_s)-\frac{\partial F(x_s)}{\partial x}f_2(s,e)\right\}\nu_2(de)ds\,\bigg|\mathcal{F}^Y_t\right].
	\end{aligned}
\end{equation*}
It follows from the independence of $W^1$, $W^2$, $\tilde{N}_1$ and $\tilde{N}_2$ that
\begin{equation*}
	\begin{aligned}
		\mu^u_t(F)&=\mu^u_0(F)+\int_0^t\mathbb{E}\left[\tilde{\Gamma}_se^{\theta\int_{0}^{s}l(r,x_r,u_r)dr}\mathcal{L}_sF(x_s)\,\big|\mathcal{F}^Y_s\right]dt\\
        &\quad +\int_0^t\mathbb{E}\left[\tilde{\Gamma}_se^{\theta\int_0^sl(r,x_r,u_r)dr}\mathcal{M}_sF(x_s)\,\big|\mathcal{F}^Y_s\right]dW^2_s\\
		&\quad +\int_0^t\int_{E_2}\mathbb{E}\left[\tilde{\Gamma}_{s-}e^{\theta\int_0^{s-}l(r,x_r,u_r)dr}\mathcal{I}^{\lambda,f_2}_{s-}F(x_s)\,\big|\mathcal{F}^Y_s\right]\tilde{N}_2(de,dt)\\
		&\quad +\int_0^t\int_{E_1}\mathbb{E}\left[\tilde{\Gamma}_se^{\theta\int_0^sl(r,x_r,u_r)dr}\mathcal{A}^{1,f_1}_sF(x_s)\,\big|\mathcal{F}^Y_s\right]\nu_1(de)dt\\
        &\quad +\int_0^t\int_{E_2}\mathbb{E}\left[\tilde{\Gamma}_se^{\theta\int_0^sl(r,x_r,u_r)dr}\mathcal{A}^{\lambda,f_2}_sF(x_s)\,\big|\mathcal{F}^Y_s\right]\nu_2(de)dt\\
		&=\mu^u_0(F)+\int_0^t\mu^u_s(\mathcal{L}_sF)ds+\int_0^t\mu^u_s(\mathcal{M}_sF)dW^2_s+\int_0^t\int_{E_2}\mu^u_{s-}\big(\mathcal{I}^{\lambda,f_2}_{s-}F\big)\tilde{N}_2(de,ds)\\
		&\quad +\int_0^t\int_{E_1}\mu^u_s\big(\mathcal{A}^{1,f_1}_sF\big)\nu_1(de)ds+\int_0^t\int_{E_2}\mu^u_s\big(\mathcal{A}^{\lambda,f_2}_sF\big)\nu_2(de)ds.
	\end{aligned}
\end{equation*}
Thus we arrive at the modified Zakai equation \eqref{equation zakai}.
\end{proof}

\begin{remark}
	When $f_1,f_2,f_3\equiv0$, \eqref{equation zakai} is reduced to the Zakai equation (3.2) of \cite{Nagai01}, within the dynamic asset management for the factor model. When $l\equiv0$, \eqref{equation zakai} degenerates into the Zakai equation (2.6) in \cite{GG22}. Noting that in \cite{Nagai01}, Nagai gives an explicit representation to the solution of SPDE (3.2), which heavily depends on the fact that the random noise of SPDE (3.2) only comes from Brownian motion. Since \eqref{equation zakai} also depends on Poisson random measure, it is hard to obtain the solution to our modified Zakai equation (\ref{equation zakai}) explicitly.
\end{remark}

\section{Concluding remarks}

In this paper, we have derived a global maximum principle for partially observed progressive optimal control of forward-backward jump stochastic systems with risk-sensitive criteria. By introducing a special BSDEP with quadratic-exponential growth, the original problem can be transformed to a stochastic recursive optimal control problem, where the system is a controlled FBSDEP coupled with a Q$_{exp}$BSDEP. Our work is an extension of \cite{ZS23}, since our system is coupled FBSDEP, which, more importantly, couples with a Q$_{exp}$BSDEP. Our work is also an extension of \cite{HJX22,BLLW24} to the Poisson jump setting. Compared with \cite{HJX22,BLLW24}, where only quadratic generator of state equation is considered, we study the controlled system with quadratic-exponential generator. To estimate orders of variations, inspired by \cite{BLLW24}, some new tools are introduced to deal with difficulties caused by the quadratic-exponential feature. As an application, a risk-sensitive optimal investment model affected by the underlying economic factors, is studied. The risk-sensitive stochastic filtering problem is also studied. Both Brownian and Poissonian correlated noises are involved in our setting, and the modified Zakai equation is obtained.

\appendix

\setcounter{section}{0}

\renewcommand\theequation{A\arabic{equation}}

\section*{Appendix A: Proof of Lemma \ref{lemma of hat zeta 1}}\label{Appendix A}

\begin{proof}
	For each $p\geqslant2$, set $p_0\coloneqq p\bar{q}^2$. By Lemma \ref{lemma stochastic Lip BSDEP}, we have
\begin{equation*}
\begin{aligned}
	&\mathbb{E}\bigg[\sup_{0\leqslant t\leqslant T}|\hat{\zeta}^{i,1,\epsilon}_t|^p+\Big(\int^T_0|\hat{\kappa}^{i,1,\epsilon}_t|^2dt\Big)^{\frac{p}{2}}
+\Big(\int^T_0\int_{E_i}|\hat{\tilde{\kappa}}^{i,1,\epsilon}_{(t,e)}|^2N_i(de,dt)\Big)^{\frac{p}{2}}\bigg]\\
	&\leqslant C\biggl\{\mathbb{E}\bigg[|\tilde{\varphi}_x(T,0)\hat{x}^{1,\epsilon}_T+\tilde{\varphi}_y(T,0)\hat{y}^{1,\epsilon}_0|^{p_0}
+\Big[\int^T_0\Big(\tilde{l}_x(t)\hat{x}^{1,\epsilon}_t+\tilde{l}_y(t)\hat{y}^{1,\epsilon}_t+\tilde{l}_{z^i}(t)\hat{z}^{i,1,\epsilon}_t\\
	&\qquad\qquad+\tilde{l}_{\tilde{z}^i}(t)\int_{E_i}\hat{\tilde{z}}^{i,1,\epsilon}_{(t,e)}\nu_i(de)+\delta l(t)\mathbbm{1}_{E_\epsilon}(t)\Big)dt\Big]^{\beta_0}\bigg]\biggr\}^{\frac{p}{p_0}}\\
	&\leqslant C\biggl\{\bigg(\mathbb{E}\bigg[\sup_{0\leqslant t\leqslant T}|\hat{x}^{1,\epsilon}_t|^{p_0}+\sup_{0\leqslant t\leqslant T}|\hat{y}^{1,\epsilon}_t|^{p_0}\bigg]\bigg)^{\frac{p}{p_0}}+\bigg(\mathbb{E}\bigg(\int^T_0\tilde{l}_{z^i}(t)\hat{z}^{i,1,\epsilon}_tdt\bigg)^{p_0}\bigg)^{\frac{p}{p_0}}\\
	&\qquad+\bigg(\mathbb{E}\bigg(\int^T_0\int_{E_i}\tilde{l}_{\tilde{z}^i}(t)\hat{\tilde{z}}^{i,1,\epsilon}_{(t,e)}\nu_i(de)dt\bigg)^{p_0}\bigg)^{\frac{p}{p_0}}
+\bigg(\mathbb{E}\bigg(\int^T_0\delta l(t)\mathbbm{1}_{[\bar{t},\bar{t}+\epsilon]}(t)dt\bigg)^{p_0}\bigg)^{\frac{p}{p_0}}\biggr\}\\
	&\leqslant C\biggl\{\bigg(\mathbb{E}\bigg[\sup_{0\leqslant t\leqslant T}|\hat{x}^{1,\epsilon}_t|^{p_0}+\sup_{0\leqslant t\leqslant T}|\hat{y}^{1,\epsilon}_t|^{p_0}\bigg]\bigg)^{\frac{p}{p_0}}
+\bigg(\mathbb{E}\bigg(\int^T_0|\hat{z}^{i,1,\epsilon}_t|^2dt\bigg)^{\frac{p_0}{2}}\bigg)^{\frac{p}{p_0}}\\
	&\qquad+\bigg(\mathbb{E}\bigg(\int^T_0\int_{E_i}|\hat{\tilde{z}}^{i,1,\epsilon}_{(t,e)}|^2\nu_i(de)dt\bigg)^{\frac{p_0}{2}}\bigg)^{\frac{p}{p_0}}
+\bigg(\mathbb{E}\bigg(\int^T_0\delta l(t)\mathbbm{1}_{[\bar{t},\bar{t}+\epsilon]}(t)dt\bigg)^{p_0}\bigg)^{\frac{p}{p_0}}\biggr\}\\
	&\leqslant C\epsilon^{\frac{p}{2}}+C\epsilon^{\frac{p}{2}}\bigg(\mathbb{E}\bigg(\int_{\bar{t}}^{\bar{t}+\epsilon}|\delta l(t)|^2dt\bigg)^{\frac{p_0}{2}}\bigg)^{\frac{p}{p_0}}.
\end{aligned}
\end{equation*}
Noticing the $|\delta l(t)|\leqslant C\big(1+|\bar{y}|+|\bar{z}^i|+\int_{E_i}|\bar{\tilde{z}}^i|\nu_i(de)+|\bar{\kappa}^i|+\int_{E_i}|\bar{\tilde{\kappa}}^i|\nu_i(de)\big)$, and using the energy inequality, we can further get
\begin{equation}
	\begin{aligned}
		\mathbb{E}\bigg[\sup_{0\leqslant t\leqslant T}|\hat{\zeta}^{1,\epsilon}_t|^p+\Big(\int^T_0|\hat{\kappa}^{i,1,\epsilon}_t|^2dt\Big)^{\frac{p}{2}}
+\Big(\int^T_0\int_{E_i}|\hat{\tilde{\kappa}}^{i,1,\epsilon}_{(t,e)}|^2N_i(de,dt)\Big)^{\frac{p}{2}}\bigg]\leqslant C\epsilon^{\frac{p}{2}}.
	\end{aligned}
\end{equation}
For $p\in(1,2)$, it follows immediately from H\"{o}lder's inequality. The proof is complete.
\end{proof}

\section*{Appendix B: Proof of Lemma \ref{lemma of x2 y2 z2 etc}}

\begin{proof}
	The first estimate follows from classical techniques. However, we should note that we fail to make the order reach $O(\epsilon^2)$ due to the appearance of $\mathbb{E}\big(\int_0^Tq^i_t\delta\sigma_i(t)\mathbbm{1}_{[\bar{t},\bar{t}+\epsilon]}dt\big)^2$ term, and so does the second estimate. Then, we give the proof of the third estimate directly.
    Let
	\begin{equation}\label{notation4}
		\begin{aligned}
			&\hat{x}_t^{3,\epsilon}\coloneqq x^\epsilon_t-\bar{x}_t-x_t^1-x_t^2,\quad \hat{y}_t^{3,\epsilon}\coloneqq y^\epsilon_t-\bar{y}_t-y_t^1-y_t^2,\\
			&\hat{z}_t^{i,3,\epsilon}\coloneqq z^{i,\epsilon}_t-\bar{z}^i_t-z_t^{i,1}-z_t^{i,2},\quad \hat{\tilde{z}}_{(t,e)}^{i,3,\epsilon}\coloneqq\tilde{z}^{i,\epsilon}_{(t,e)}-\bar{\tilde{z}}^i_{(t,e)}-\tilde{z}_{(t,e)}^{i,1}-\tilde{z}_{(t,e)}^{i,2},\quad i=1,2,
		\end{aligned}
	\end{equation}
	and recall $\hat{x}_t^{1,\epsilon},\hat{y}_t^{1,\epsilon},\hat{z}_t^{i,1,\epsilon},\hat{\tilde{z}}_{(t,e)}^{i,1,\epsilon},\hat{x}_t^{2,\epsilon},\hat{y}_t^{2,\epsilon},\hat{z}_t^{i,2,\epsilon}$ and $\hat{\tilde{z}}_{(t,e)}^{i,2,\epsilon}$, $i=1,2$. Further, we define
	\begin{equation*}
		\begin{aligned}
			\tilde{\Upsilon}_{xx}^\epsilon(t)&\coloneqq 2\int_0^1\int_0^1\theta \Upsilon_{xx}(t,\bar{x}+\lambda\theta(x^{\epsilon}-\bar{x}),u^\epsilon)d\lambda d\theta,\mbox{ for }\Upsilon=b,\sigma_1,\sigma_2,f_1,f_2,\\
		    \widetilde{D^2g}(t)&\coloneqq 2\int_0^1\int_0^1\theta D^2g\big(t,\Theta(t,\Delta^1\mathbbm{1}_{[\bar{t},\bar{t}+\epsilon]}(t),\Delta^2\mathbbm{1}_{[\bar{t},\bar{t}+\epsilon]}(t))+\lambda\theta(\Theta^\epsilon(t)\\
			&\qquad\qquad\qquad\quad -\Theta(t,\Delta^1\mathbbm{1}_{[\bar{t},\bar{t}+\epsilon]}(t),\Delta^2\mathbbm{1}_{[\bar{t},\bar{t}+\epsilon]}(t))),u^\epsilon\big)d\lambda d\theta,
		\end{aligned}
	\end{equation*}
where $\Theta(\cdot,\Delta^1,\Delta^2)\equiv\big[\bar{x},\bar{y},\bar{z}^1+\Delta^1,\bar{z}^2+\Delta^2,\bar{\tilde{z}}^1,\bar{\tilde{z}}^2\big]$ and
$$
\widetilde{\phi}^{\epsilon}_{xx}(T)\coloneqq 2\int_0^1\int_0^1\theta \phi_{xx}(\bar{x}_T+\hat{\tilde{z}}\theta(x_T^\epsilon-\bar{x}_T))d\hat{\tilde{z}} d\theta,
$$
then we have the following equation:
\begin{equation*}\left\{
\begin{aligned}
    d\hat{x}_t^{3,\epsilon}&=\Big[b_{x}(t)\hat{x}_t^{3,\epsilon}+b_{y}(t)\hat{y}_t^{3,\epsilon}+b_{z^i}(t)\hat{z}_t^{i,3,\epsilon}
    +b_{\tilde{z}^i}(t)\int_{E_i}\hat{\tilde{z}}_{(t,e)}^{i,3,\epsilon}\nu_i(de)+B_4^\epsilon(t)\Big]dt\\
	&\quad+\Big[\sigma_{ix}(t)\hat{x}_t^{3,\epsilon}+\sigma_{iy}(t)\hat{y}_t^{3,\epsilon}+C_{i4}^\epsilon(t)\Big]dW^i_t\\%
    &\quad+\int_{E_i}\Big[f_{ix}(t,e)\hat{x}_{t-}^{3,\epsilon}+f_{iy}(t,e)\hat{y}_{t-}^{3,\epsilon}+D_{i4}^\epsilon(t,e)\Big]\tilde{N}_i(de,dt),\\
	d\hat{y}_t^{3,\epsilon}&=-\Big[g_x(t)\hat{x}_t^{3,\epsilon}+g_y(t)\hat{y}_t^{3,\epsilon}+g_{z^i}(t)\hat{z}_t^{i,3,\epsilon}
	+g_{\tilde{z}^i}(t)\int_{E_i}\hat{\tilde{z}}_{(t,e)}^{i,3,\epsilon}\nu_i(de)+A_4^\epsilon(t)\Big]dt\\
	&\quad +\hat{z}_t^{i,3,\epsilon}dW^i_t+\int_{E_i}\hat{\tilde{z}}_{(t,e)}^{i,3,\epsilon}\tilde{N}_i(de,dt),\quad t\in[0,T],\\
	\hat{x}_0^{3,\epsilon}&=0,\quad \hat{y}_T^{3,\epsilon}=\phi_x(\bar{x}_T)\hat{x}_T^{3,\epsilon}+\frac{1}{2}\widetilde{\phi}^{\epsilon}_{xx}(T)(\hat{x}_T^{1,\epsilon})^2-\frac{1}{2}\phi_{xx}(\bar{x}_T)(x_T^1)^2,
\end{aligned}
\right.\end{equation*}
where
	\begin{equation*}
	\begin{aligned}
		B_4^\epsilon(t)&\coloneqq\Bigl\{\delta b_x(t,\Delta^1,\Delta^2)\hat{x}_t^{1,\epsilon}+\delta b_y(t,\Delta^1,\Delta^2)\hat{y}_t^{1,\epsilon}
		+\delta b_{z^i}(t,\Delta^1,\Delta^2)(\hat{z}_t^{i,1,\epsilon}-\Delta^i(t)\mathbbm{1}_{[\bar{t},\bar{t}+\epsilon]})\\
		&\qquad +\delta b_{\tilde{z}^i}(t,\Delta^1,\Delta^2)\int_{E_i}\hat{\tilde{z}}_{(t,e)}^{i,1,\epsilon}\nu_i(de)\Bigr\}\mathbbm{1}_{[\bar{t},\bar{t}+\epsilon]}\\
        &\quad+\frac{1}{2}\check{\Xi}_2(t)\widetilde{D^2b}(t)\check{\Xi}_2(t)^\top-\frac{1}{2}\tilde{\Xi}_2(t)D^2g(t)\tilde{\Xi}_2(t)^\top,\\
        C_{i4}^\epsilon(t)&\coloneqq\big[\delta \sigma_{ix}(t)\hat{x}_t^{2,\epsilon}+\delta \sigma_{iy}(t)\hat{y}_t^{2,\epsilon}\big]\mathbbm{1}_{[\bar{t},\bar{t}+\epsilon]}
        +\frac{1}{2}\big[\hat{x}^{1,\epsilon}_t,\hat{y}^{1,\epsilon}_t\big]\widetilde{D^2\sigma_i}(t)\big[\hat{x}^{1,\epsilon}_t,\hat{y}^{1,\epsilon}_t\big]^\top\\
        &\quad -\frac{1}{2}\big[x^1_t,y^1_t\big]D^2\sigma_i(t)\big[x^1_t,y^1_t\big]^\top,\\
		\end{aligned}
	\end{equation*}
    \begin{equation*}
	    \begin{aligned}
		D_{i4}^\epsilon(t,e)&\coloneqq\big[\delta f_{ix}(t,e)\hat{x}_t^{1,\epsilon}+\delta f_{iy}(t,e)\hat{y}_t^{1,\epsilon}\big]\mathbbm{1}_{\mathcal{O}}
        +\frac{1}{2}\big[\hat{x}^{1,\epsilon}_t,\hat{y}^{1,\epsilon}_t\big]\widetilde{D^2f_i}(t)\big[\hat{x}^{1,\epsilon}_t,\hat{y}^{1,\epsilon}_t\big]^\top\\
		&\quad-\frac{1}{2}\big[x^1_t,y^1_t\big]D^2f_i(t)\big[x^1_t,y^1_t\big]^\top+\delta f_i(t,e)\mathbbm{1}_{\mathcal{O}},\\
        A_4^\epsilon(t)&\coloneqq\Bigl\{\delta g_x(t,\Delta^1,\Delta^2)\hat{x}_t^{1,\epsilon}+\delta g_y(t,\Delta^1,\Delta^2)\hat{y}_t^{1,\epsilon}
		+\delta g_{z^i}(t,\Delta^1,\Delta^2)(\hat{z}_t^{i,1,\epsilon}-\Delta^i(t)\mathbbm{1}_{[\bar{t},\bar{t}+\epsilon]})\\
		&\qquad +\delta g_{\tilde{z}^i}(t,\Delta^1,\Delta^2)\int_{E_i}\hat{\tilde{z}}_{(t,e)}^{i,1,\epsilon}\nu_i(de)\Bigr\}\mathbbm{1}_{[\bar{t},\bar{t}+\epsilon]}\\
		&\quad +\frac{1}{2}\check{\Xi}_2(t)\widetilde{D^2g}(t)\check{\Xi}_2(t)^\top-\frac{1}{2}\tilde{\Xi}_2(t)D^2g(t)\tilde{\Xi}_2(t)^\top,\\
		\check{\Xi}_2&\coloneqq\Big[\hat{x}^{1,\epsilon},\hat{y}^{1,\epsilon},\hat{z}^{1,1,\epsilon}
		-\Delta^1\mathbbm{1}_{[\bar{t},\bar{t}+\epsilon]},\hat{z}^{2,1,\epsilon}
		-\Delta^2\mathbbm{1}_{[\bar{t},\bar{t}+\epsilon]},\int_{E_1}\hat{\tilde{z}}^{1,1,\epsilon}\nu_1(de),\int_{E_2}\hat{\tilde{z}}^{2,1,\epsilon}\nu_2(de)\Big],\\
		\tilde{\Xi}_2&\coloneqq\Big[x^1,y^1,z^{1,1}-\Delta^1\mathbbm{1}_{[\bar{t},\bar{t}+\epsilon]},z^{2,1}-\Delta^2\mathbbm{1}_{[\bar{t},\bar{t}+\epsilon]},
        \int_{E_1}\tilde{z}^{1,1}\nu_1(de),\int_{E_2}\tilde{z}^{2,1}\nu_2(de)\Big].
		\end{aligned}
	\end{equation*}
	By Proposition \ref{theorem solvability xyz}, the $L^p$-estimate of FBSDEP, we have
	\begin{equation*}
		\begin{aligned}
			&\mathbb{E}\bigg[\sup_{0\leqslant t\leqslant T}\Big(|\hat{x}^{3,\epsilon}_t|^2+|\hat{y}^{3,\epsilon}_t|^2\Big)+\int_0^T|\hat{z}^{i,3,\epsilon}_t|^2dt
			+\int_0^T\int_{E_i}|\hat{\tilde{z}}^{i,3,\epsilon}_{(t,e)}|^2\nu_i(de)dt\bigg]\\
            &\leqslant C \mathbb{E}\bigg[\Big|\frac{1}{2}\widetilde{\phi}^\epsilon_{xx}(T)(\hat{x}_T^{1,\epsilon})^2-\frac{1}{2}\phi_{xx}(\bar{x}_T)(x_T^1)^2\Big|^2
			+\Big(\int_0^TA_4^\epsilon(t)dt\Big)^2+\Big(\int_0^TB^\epsilon_1(t)dt\Big)^2\\
			&\qquad\quad +\int_0^T|C^\epsilon_{i1}(t)|^2dt+\int_0^T\int_{E_i}|D^\epsilon_{i1}(t,e)|^2N_i(de,dt)\bigg].
		\end{aligned}
	\end{equation*}
Notice that
\begin{equation*}
	\begin{aligned}
			&\mathbb{E}\bigg[\Big|\frac{1}{2}\widetilde{\phi}^\epsilon_{xx}(T)\big[(\hat{x}_T^{1,\epsilon})^2-(x_T^1)^2\big]
			+\frac{1}{2}\big[\widetilde{\phi}^\epsilon_{xx}(T)-\phi_{xx}(\bar{x}_T)\big](x_T^1)^2\Big|^2\bigg]\\
			&\leqslant C\biggl\{\mathbb{E}\bigg[\Big|\frac{1}{2}\widetilde{\phi}^\epsilon_{xx}(T)\hat{x}_T^{2,\epsilon}(\hat{x}_T^{1,\epsilon}+x_T^1)
			+\frac{1}{2}\big[\widetilde{\phi}^\epsilon_{xx}(T)-\phi_{xx}(\bar{x}_T)\big](x_T^1)^2\Big|^2\bigg]\biggr\}\\
			&\leqslant C\biggl\{\mathbb{E}\Big[|\hat{x}_T^{2,\epsilon}|^2|\hat{x}_T^{1,\epsilon}+x_T^1|^2\Big]+\mathbb{E}\Big[[\widetilde{\phi}^\epsilon_{xx}(T)-\phi_{xx}(\bar{x}_T)]^2(x_T^1)^4\Big]\biggr\}\\
			&\leqslant C\biggl\{\biggl\{\mathbb{E}\bigg[\sup_{0\leqslant t\leqslant T}|\hat{x}_t^{2,\epsilon}|^4\bigg]\biggr\}^\frac{1}{2}\biggl\{\mathbb{E}
            \bigg[\sup_{0\leqslant t\leqslant T}|\hat{x}_t^{1,\epsilon}+x_t^1|^4\bigg]\biggr\}^\frac{1}{2}\\
			&\qquad +\biggl\{\mathbb{E}\Big[[\widetilde{\phi}^{\epsilon}_{xx}(T)-\phi_{xx}(\bar{x}_T)]^4\Big]\biggr\}^\frac{1}{2}
			\biggl\{\mathbb{E}\bigg[\sup_{0\leqslant t\leqslant T}|x_t^1|^8\bigg]\biggr\}^\frac{1}{2}\biggr\}=o(\epsilon^2).
		\end{aligned}
	\end{equation*}
	Next, we need to give the estimate of $\mathbb{E}\big(\int_0^TA_4^\epsilon(t)dt\big)^2$. In fact,
	\begin{equation}\label{A2}
		\begin{aligned}
			&\mathbb{E}\bigg(\int_0^TA_4^\epsilon(t)dt\bigg)^2=\mathbb{E}\bigg[\bigg(\int_0^T\bigg[\delta g_x(t,\Delta^1,\Delta^2)\hat{x}_t^{1,\epsilon}+\delta g_y(t,\Delta^1,\Delta^2)\hat{y}_t^{1,\epsilon}\\
			&\quad+\sum^2_{i=1}\Big[\delta g_{z^i}(t,\Delta^1,\Delta^2)(\hat{z}_t^{i,1,\epsilon}-\Delta^i(t)\mathbbm{1}_{[\bar{t},\bar{t}+\epsilon]})
			+\delta g_{\tilde{z}^i}(t,\Delta^1,\Delta^2)\int_{E_i}\hat{\tilde{z}}_{(t,e)}^{i,1,\epsilon}\nu_i(de)\Big]\bigg]\mathbbm{1}_{[\bar{t},\bar{t}+\epsilon]}\\
            &\quad +\frac{1}{2}\check{\Xi}_2(t)\widetilde{D^2g}(t)\check{\Xi}_2(t)^\top-\frac{1}{2}\tilde{\Xi}_2(t)D^2g(t)\tilde{\Xi}_2(t)^\top dt\bigg)^2\bigg].
		\end{aligned}
	\end{equation}
	We divide the right hand of \eqref{A2} into several parts to get the following estimates:
	\begin{equation*}
		\begin{aligned}
			&\mathbb{E}\bigg(\int_0^T\delta g_x(t,\Delta^1,\Delta^2)\hat{x}_t^{1,\epsilon}\mathbbm{1}_{[\bar{t},\bar{t}+\epsilon]}dt\bigg)^2
			\leqslant C \mathbb{E}\bigg(\int_{\bar{t}}^{\bar{t}+\epsilon}\hat{x}_t^{1,\epsilon}dt\bigg)^2
			\leqslant C\epsilon^2\mathbb{E}\bigg[\sup_{0\leqslant t\leqslant T}|\hat{x}_t^{1,\epsilon}|^2\bigg]\leqslant o(\epsilon^2),\\
		\end{aligned}
	\end{equation*}
    \begin{equation*}
		\begin{aligned}	
            &\mathbb{E}\bigg(\int_0^T\delta g_{z^i}(t,\Delta^1,\Delta^2)(\hat{z}_t^{i,1,\epsilon}-\Delta^i(t)\mathbbm{1}_{[\bar{t},\bar{t}+\epsilon]})\mathbbm{1}_{[\bar{t},\bar{t}+\epsilon]}dt\bigg)^2\\
			&\leqslant C \mathbb{E}\bigg(\int_{\bar{t}}^{\bar{t}+\epsilon}|\delta g_{z^i}(t,\Delta^1,\Delta^2)||(\hat{z}_t^{i,2,\epsilon}+K^i_1(t)x_t^1)|dt\bigg)^2
			\leqslant C \mathbb{E}\bigg(\int_{\bar{t}}^{\bar{t}+\epsilon}|\hat{z}_t^{i,2,\epsilon}+K^i_1(t)x_t^1|dt\bigg)^2\\
			&\leqslant C\epsilon \mathbb{E}\int_{\bar{t}}^{\bar{t}+\epsilon}|\hat{z}_t^{i,2,\epsilon}|^2dt
			+C\epsilon\mathbb{E}\bigg[\sup_{0\leqslant t\leqslant T}|x_t^1|^2\Big(\int_{\bar{t}}^{\bar{t}+\epsilon}|K^i_1(t)|^2dt\Big)\bigg]\\
			&\leqslant C\epsilon\mathbb{E}\int_0^T|\hat{z}_t^{i,2,\epsilon}|^2dt
            +C\epsilon\biggl\{\mathbb{E}\bigg[\sup_{0\leqslant t\leqslant T}|x_t^1|^4\bigg]\biggr\}^{\frac{1}{2}}\biggl\{\mathbb{E}\bigg(\int_{\bar{t}}^{\bar{t}+\epsilon}|K^i_1(t)|^2dt\bigg)^2\biggr\}^{\frac{1}{2}}
            \leqslant o(\epsilon^2),
		\end{aligned}
	\end{equation*}
	and similarly,
	\begin{equation*}
		\begin{aligned}
			&\mathbb{E}\bigg(\int_0^T\delta g_y(t,\Delta^1,\Delta^2)\hat{y}_t^{1,\epsilon}\mathbbm{1}_{[\bar{t},\bar{t}+\epsilon]}dt\bigg)^2\leqslant o(\epsilon^2)\\
			&\mathbb{E}\bigg(\int_0^T\delta g_{\tilde{z}^i}(t,\Delta^1,\Delta^2)\int_{E_i}\hat{\tilde{z}}_{(t,e)}^{i,1,\epsilon}\nu_i(de)\mathbbm{1}_{[\bar{t},\bar{t}+\epsilon]}dt\bigg)^2\leqslant o(\epsilon^2),\quad i=1,2.
		\end{aligned}
	\end{equation*}
	Then we consider the last part of \eqref{A2} as follows. In which, we have
	\begin{equation*}
		\begin{aligned}
			&\mathbb{E}\bigg(\int_0^T\widetilde{g}^\epsilon_{z^1z^1}(t)\big[\hat{z}_t^{1,1,\epsilon}-\Delta^1(t)\mathbbm{1}_{[\bar{t},\bar{t}+\epsilon]}\big]^2
			-g_{z^1z^1}(t)\big[z_t^{1,1}-\Delta^1(t)\mathbbm{1}_{[\bar{t},\bar{t}+\epsilon]}\big]^2dt\bigg)^2\\
			&\leqslant \mathbb{E}\bigg(\int_0^T\widetilde{g}^\epsilon_{z^1z^1}(t)\big[(\hat{z}_t^{1,1,\epsilon}-\Delta^1(t)\mathbbm{1}_{[\bar{t},\bar{t}+\epsilon]})^2
			-(z_t^{1,1}-\Delta^1(t)\mathbbm{1}_{[\bar{t},\bar{t}+\epsilon]})^2\big]\\
			&\qquad +(\widetilde{g}^\epsilon_{z^1z^1}(t)-g_{z^1z^1}(t))\big[z_t^{1,1}-\Delta^1(t)\mathbbm{1}_{[\bar{t},\bar{t}+\epsilon]}\big]^2dt\bigg)^2\\\
			&\leqslant \mathbb{E}\bigg(\int_0^T\Big(\widetilde{g}^\epsilon_{z^1z^1}(t)\hat{z}_t^{1,2,\epsilon}\big[\hat{z}_t^{1,2,\epsilon}+2K^1_1(t)x_t^1\big]
			+(\widetilde{g}^{\epsilon}_{z^1z^1}(t)-g_{z^1z^1}(t))\big[K^1_1(t)x_t^1\big]^2\Big)dt\bigg)^2\\
			&\leqslant C \mathbb{E}\bigg[\Big(\int_0^T\widetilde{g}^\epsilon_{z^1z^1}(t)\hat{z}_t^{1,2,\epsilon}\big[\hat{z}_t^{1,2,\epsilon}+2K^1_1(t)x_t^1\big]dt\Big)^2
			+\Big(\int_0^T(\widetilde{g}^\epsilon_{z^1z^1}(t)-g_{z^1z^1}(t))\big[K^1_1(t)x_t^1\big]^2dt\Big)^2\bigg]\\
			&\leqslant C \mathbb{E}\bigg(\int_0^T\hat{z}_t^{1,2,\epsilon}\big[\hat{z}_t^{1,2,\epsilon}+2K^1_1(t)x_t^1\big]dt\bigg)^2\\
            &\quad +\mathbb{E}\bigg[\sup_{0\leqslant t\leqslant T}|x_t^1|^4\Big(\int_0^T(\widetilde{g}^\epsilon_{z^1z^1}(t)-g_{z^1z^1}(t))|K^1_1(t)|^2dt\Big)^2\bigg]\\
	        &\leqslant C\mathbb{E}\bigg(\int_0^T|\hat{z}_t^{1,2,\epsilon}|^2dt\bigg)^2+C \mathbb{E}\bigg(\int_0^T\hat{z}_t^{1,2,\epsilon}K^1_1(t)x_t^1dt\bigg)^2\\\
			&\quad +\mathbb{E}\bigg[\sup_{0\leqslant t\leqslant T}|x_t^1|^4\Big(\int_0^T(\widetilde{g}^\epsilon_{z^1z^1}(t)-g_{z^1z^1}(t))|K^1_1(t)|^2dt\Big)^2\bigg]\\
            &\leqslant o(\epsilon^2)+C\mathbb{E}\bigg[\sup_{0\leqslant t\leqslant T}|x_t^1|^2\Big(\int_0^T\hat{z}_t^{1,2,\epsilon}K^1_1(t)dt\Big)^2\bigg]\\
			\end{aligned}
	\end{equation*}
	\begin{equation*}
		    \begin{aligned}
            &\quad +\mathbb{E}\bigg[\sup_{0\leqslant t\leqslant T}|x_t^1|^4\Big(\int_0^T(\widetilde{g}^\epsilon_{z^1z^1}(t)-g_{z^1z^1}(t))|K^1_1(t)|^2dt\Big)^2\bigg]\\
            &\leqslant o(\epsilon^2)+C\mathbb{E}\bigg[\sup_{0\leqslant t\leqslant T}|x_t^1|^2\Big(\int_0^T|\hat{z}_t^{1,2,\epsilon}|^2dt\Big)\Big(\int_0^T|K^1_1(t)|^2dt\Big)\bigg]\\
			&\quad +\biggl\{\mathbb{E}\bigg[\sup_{0\leqslant t\leqslant T}|x_t^1|^8\bigg]\biggr\}^{\frac{1}{2}}\bigg\{\mathbb{E}\bigg(\int_0^T(\widetilde{g}^{\epsilon}_{z^1z^1}(t)-g_{z^1z^1}(t))|K^1_1(t)|^2dt\bigg)^4\bigg\}^{\frac{1}{2}}\\
			&\leqslant o(\epsilon^2)+\biggl\{\mathbb{E}\bigg[\sup_{0\leqslant t\leqslant T}|x_t^1|^6\bigg]\biggr\}^{\frac{1}{3}}
\biggl\{\mathbb{E}\bigg(\int_0^T|\hat{z}_t^{1,2,\epsilon}|^2dt\bigg)^2\biggr\}^{\frac{1}{2}}
			\biggl\{\mathbb{E}\bigg(\int_0^T|K^1_1(t)|^2dt\bigg)^6\biggr\}^{\frac{1}{6}}\\
			&\quad +\epsilon^2\biggl\{\mathbb{E}\bigg(\int_0^T(\widetilde{g}^{\epsilon}_{z^1z^1}(t)-g_{z^1z^1}(t))|K^1_1(t)|^2dt\bigg)^4\biggr\}^{\frac{1}{2}}\leqslant o(\epsilon^2),
		\end{aligned}
	\end{equation*}
	and similarly,
	\begin{equation*}
		\begin{aligned}
			&\mathbb{E}\bigg(\int_0^T\Big(\widetilde{g}^\epsilon_{xx}(t)|\hat{x}_t^{1,\epsilon}|^2-g_{xx}(t)|x_t^1|^2\Big)dt\bigg)^2\leqslant o(\epsilon^2),\\
			&\mathbb{E}\bigg(\int_0^T\Big(\widetilde{g}^\epsilon_{yy}(t)|\hat{y}_t^{1,\epsilon}|^2-g_{yy}(t)|y_t^1|^2\Big)dt\bigg)^2\leqslant o(\epsilon^2),\\
			&\mathbb{E}\bigg(\int_0^T\Big(\widetilde{g}^\epsilon_{\tilde{z}^i\tilde{z}^i}(t)\Big|\int_{E_i}\hat{\tilde{z}}_{(t,e)}^{i,1,\epsilon}\nu_i(de)\Big|^2
			-g_{\tilde{z}^i\tilde{z}^i}(t)\Big|\int_{E_i}\tilde{z}_{(t,e)}^{i,1}\nu_i(de)\Big|^2\Big)dt\bigg)^2\leqslant o(\epsilon^2),\quad i=1,2,
		\end{aligned}
	\end{equation*}
	and the other cross terms have the same estimate. Using the same method, we can also deduce $\mathbb{E}\big(\int_0^TB_4^\epsilon(t)dt\big)^2=o(\epsilon^2)$.

Then, we give the estimates of $\mathbb{E}\int_0^T|C_{i4}^\epsilon(t)|^2dt$ and $\mathbb{E}\int_0^T\int_{E_i}|D_{i4}^\epsilon(t,e)|^2N_i(de,dt)$, respectively:
	\begin{equation*}
		\begin{aligned}
			&\mathbb{E}\int_0^T|C_{i4}^\epsilon(t)|^2dt=\mathbb{E}\int_0^T\Big|\big(\delta \sigma_{ix}(t)\hat{x}_t^{2,\epsilon}+\delta \sigma_{iy}(t)\hat{y}_t^{2,\epsilon}\big)\mathbbm{1}_{[\bar{t},\bar{t}+\epsilon]}
			+\frac{1}{2}\tilde{\sigma}^\epsilon_{ixx}(t)|\hat{x}_t^{1,\epsilon}|^2-\frac{1}{2}\sigma_{ixx}(t)(x^1_t)^2\\
			&\quad+\frac{1}{2}\tilde{\sigma}^{\epsilon}_{ixy}(t)|\hat{x}_t^{1,\epsilon}||\hat{y}^{1,\epsilon}_t|-\frac{1}{2}\sigma_{ixy}(t)|x^1_t||y^1_t|
            +\frac{1}{2}\tilde{\sigma}^{\epsilon}_{iyy}(t)|\hat{y}_t^{1,\epsilon}|^2-\frac{1}{2}\sigma_{iyy}(t)(y^1_t)^2\Big|^2dt\\
			&\leqslant C\mathbb{E}\int_0^T\Big|\big(|\delta \sigma_{ix}(t)|^2|\hat{x}_t^{2,\epsilon}|^2+|\delta \sigma_{iy}(t)|^2|\hat{y}_t^{2,\epsilon}|^2\big)\mathbbm{1}_{[\bar{t},\bar{t}+\epsilon]}
            +\big|\frac{1}{2}\tilde{\sigma}^\epsilon_{ixx}(t)(|\hat{x}_t^{1,\epsilon}|^2-|x^1_t|^2)\\
			&\qquad\qquad\quad +\frac{1}{2}(\tilde{\sigma}^{\epsilon}_{ixx}(t)-\sigma_{ixx}(t))(x^1_t)^2\big|^2
            +\big|\tilde{\sigma}^\epsilon_{ixy}(t)\hat{x}^{1,\epsilon}_t\hat{y}^{1,\epsilon}_t-\tilde{\sigma}^\epsilon_{ixy}(t)\hat{x}^{1,\epsilon}_ty^1_t
            +\tilde{\sigma}^\epsilon_{ixy}(t)\hat{x}^{1,\epsilon}_ty^1_t\\
            &\qquad\qquad\quad -\sigma_{ixy}(t)\hat{x}^{1,\epsilon}_ty^1_t+\sigma_{ixy}(t)\hat{x}^{1,\epsilon}_ty^1_t-\sigma_{ixy}(t)x^1_ty^1_t\big|^2\\
	        &\qquad\qquad\quad +\big|\frac{1}{2}\tilde{\sigma}^{\epsilon}_{iyy}(t)(|\hat{y}_t^{1,\epsilon}|^2-|y^1_t|^2)
			+\frac{1}{2}(\tilde{\sigma}^{\epsilon}_{iyy}(t)-\sigma_{iyy}(t))(y^1_t)^2\big|\Big|^2dt\\
            &\leqslant C\mathbb{E}\int_{\bar{t}}^{\bar{t}+\epsilon}\big(|\delta \sigma_{ix}(t)|^2|\hat{x}_t^{2,\epsilon}|^2+|\delta \sigma_{iy}(t)|^2|\hat{y}_t^{2,\epsilon}|^2\big)dt\\
			&\quad+\mathbb{E}\int_{\bar{t}}^{\bar{t}+\epsilon}\Big(\big|\frac{1}{2}\tilde{\sigma}^{\epsilon}_{ixx}(t)(|\hat{x}_t^{1,\epsilon}|^2-|x^1_t|^2)
			+\frac{1}{2}(\tilde{\sigma}^{\epsilon}_{ixx}(t)-\sigma_{ixx}(t))(x^1_t)^2\big|^2\\
            &\qquad\qquad+\big|\tilde{\sigma}^{\epsilon}_{ixy}(t)\hat{x}^{1,\epsilon}_t\hat{y}^{2,\epsilon}_t+(\tilde{\sigma}^{\epsilon}_{ixy}(t)
             -\sigma_{ixy}(t))\hat{x}^{1,\epsilon}_ty^1_t+\sigma_{ixy}(t)\hat{x}^{2,\epsilon}_ty^1_t\big|^2\\
			&\qquad\qquad+\big|\frac{1}{2}\tilde{\sigma}^{\epsilon}_{iyy}(t)(|\hat{y}_t^{1,\epsilon}|^2-|y^1_t|^2)
			+\frac{1}{2}(\tilde{\sigma}^{\epsilon}_{iyy}(t)-\sigma_{iyy}(t))(y^1_t)^2\big|^2\Big)dt\leqslant o(\epsilon^2),
		\end{aligned}
	\end{equation*}
	and
	\begin{equation*}
		\begin{aligned}
			&\mathbb{E}\int_0^T\int_{E_i}|D_{i4}^\epsilon(t,e)|^2N_i(de,dt)\\
			&=\mathbb{E}\int_0^T\int_{E_i}\Big|\big(\delta f_{ix}(t,e)\hat{x}_{t-}^{1,\epsilon}
             +\delta f_{iy}(t,e)\hat{y}_{t-}^{1,\epsilon}\big)\mathbbm{1}_{\mathcal{O}}+\frac{1}{2}\tilde{f}^\epsilon_{ixx}(t,e)|\hat{x}_{t-}^{1,\epsilon}|^2\\
			&\qquad\qquad
			-\frac{1}{2}f_{ixx}(t,e)(x^1_{t-})^2+\tilde{f}^{\epsilon}_{ixy}(t,e)|\hat{x}_{t-}^{1,\epsilon}||\hat{y}_{t-}^{1,\epsilon}|-f_{ixy}(t,e)x^1_{t-}y^1_{t-}\\
			&\qquad\qquad+\frac{1}{2}\tilde{f}^{\epsilon}_{iyy}(t,e)|\hat{y}_{t-}^{1,\epsilon}|^2-\frac{1}{2}f_{iyy}(t,e)(y^1_{t-})^2+\delta f_i(t,e)\mathbbm{1}_{\mathcal{O}}\Big|^2N_i(de,dt)\\
			&\leqslant\mathbb{E}\int_0^T\int_{E_i}\Big|\big(\delta f_{ix}(t,e)\hat{x}_{t-}^{1,\epsilon}+\delta f_{iy}(t,e)\hat{y}_{t-}^{1,\epsilon}\big)\mathbbm{1}_{\mathcal{O}}+\delta f_i(t,e)\mathbbm{1}_{\mathcal{O}}\\
			&\qquad\qquad+\frac{1}{2}\tilde{f}^\epsilon_{ixx}(t,e)\big(|\hat{x}_{t-}^{1,\epsilon}|^2-|x^1_{t-}|^2\big) +\frac{1}{2}\big(\tilde{f}^\epsilon_{ixx}(t,e)-f_{ixx}(t,e)\big)(x^1_{t-})^2\\
			&\qquad\qquad+\tilde{f}^{\epsilon}_{ixy}(t,e)\hat{x}^{1,\epsilon}_{t-}\hat{y}^{1,\epsilon}_{t-}-\tilde{f}^\epsilon_{ixy}(t,e)\hat{x}^{1,\epsilon}_{t-}y^{1}_{t-}
            +\tilde{f}^\epsilon_{ixy}(t,e)\hat{x}^{1,\epsilon}_{t-}y^{1}_{t-}\\
			&\qquad\qquad-f_{ixy}(t,e)\hat{x}^{1,\epsilon}_{t-}y^{1}_{t-}+f_{ixy}(t,e)\hat{x}^{1,\epsilon}_{t-}y^{1}_{t-}-f_{ixy}(t,e)x^1_{t-}y^1_{t-}\\
			&\qquad\qquad+\frac{1}{2}\tilde{f}^\epsilon_{iyy}(t,e)(|\hat{y}_{t-}^{1,\epsilon}|^2-|y^1_{t-}|^2) +\frac{1}{2}(\tilde{f}^\epsilon_{iyy}(t,e)-f_{iyy}(t,e))(y^1_{t-})^2\Big|^2N_i(de,dt)\\
			&\leqslant\mathbb{E}\int_0^T\int_{E_i}\Big[\big(|\delta f_{ix}(t,e)|^2|\hat{x}_{t-}^{1,\epsilon}|^2+|\delta f_{iy}(t,e)|^2|\hat{y}_{t-}^{1,\epsilon}|^2+|\delta f_i(t,e)|^2\big)\mathbbm{1}_{\mathcal{O}}\\
			&\qquad\qquad+\big|\frac{1}{2}\tilde{f}^\epsilon_{ixx}(t,e)(|\hat{x}_{t-}^{1,\epsilon}|^2-|x^1_{t-}|^2)\big|^2+\big|\frac{1}{2}(\tilde{f}^{\epsilon}_{ixx}(t,e)-f_{ixx}(t,e))(x^1_{t-})^2\big|^2\\
			&\qquad\qquad+\big|\tilde{f}^{\epsilon}_{ixy}(t,e)\hat{x}^{1,\epsilon}_{t-}\hat{y}^{2,\epsilon}_{t-}\big|^2+\big|(\tilde{f}^\epsilon_{ixy}(t,e)-f_{ixy}(t,e))\hat{x}^{1,\epsilon}_{t-}y^{1}_{t-}\big|^2
             +\big|f_{ixy}(t,e)\hat{x}^{2,\epsilon}_{t-}y^{1}_{t-}\big|^2\\
			&\qquad\qquad+\big|\frac{1}{2}\tilde{f}^\epsilon_{iyy}(t,e)(|\hat{y}_{t-}^{1,\epsilon}|^2-|y^1_{t-}|^2)\big|^2+\big|\frac{1}{2}(\tilde{f}^\epsilon_{iyy}(t,e)-f_{iyy}(t,e))(y^1_{t-})^2\big|^2\Big]N_i(de,dt)\\
            &\leqslant o(\epsilon^2).
		\end{aligned}
	\end{equation*}
	The proof is complete.
\end{proof}

\section*{Appendix C: Proof of Lemma \ref{lemma estimate of zeta1 hat zeta2}}

\begin{proof}
	Noting that $\zeta^1_t=\alpha_tx^1_t$, $\kappa^{i,1}_t=K'_i(t)x^1_t+\pi^i_t\mathbbm{1}_{[\bar{t},\bar{t}+\epsilon]}$, $\tilde{\kappa}^{i,1}_{(t,e)}=\tilde{K}'_i(t,e)x^1_{t-}$, by estimate of $x^1$, space of $\alpha$, we get, for $\forall p_0>p$,
	\begin{equation*}
		\begin{aligned}
			&\mathbb{E}\bigg[\Big(\int_0^T|\beta^i_tx^1_t|^2dt\Big)^{\frac{p}{2}}+\Big(\int_0^T\int_{E_i}|\tilde{\beta}^i_{(t,e)}x^1_t|^2N_i(de,dt)\Big)^{\frac{p}{2}}\bigg]\\
			&\leqslant C \mathbb{E}\bigg[\sup_{0\leqslant t\leqslant T}|x^1_t|^p\Big(\int_0^T|\beta^i_t|^2dt\Big)^{\frac{p}{2}}\bigg]
             +\mathbb{E}\bigg[\sup_{0\leqslant t\leqslant T}|x^1_t|^p\Big(\int_0^T\int_{E_i}|\tilde{\beta}^i_{(t,e)}|^2N_i(de,dt)\Big)^{\frac{p}{2}}\bigg]\\
			&\leqslant C \biggl\{\mathbb{E}\bigg[\sup_{0\leqslant t\leqslant T}|x^1_t|^{p_0}\bigg]\biggr\}^{\frac{p}{p_0}}
             \biggl\{\bigg(\mathbb{E}\bigg(\int_0^T|\beta^i_t|^2dt\bigg)^{\frac{pp_0}{2(p_0-p)}}\bigg)^{\frac{p_0-p}{p_0}}\\
            &\qquad+\bigg(\mathbb{E}\bigg(\int_0^T\int_{E_i}|\tilde{\beta}^i_{(t,e)}|^2N_i(de,dt)\bigg)^{\frac{pp_0}{2(p_0-p)}}\bigg)^{\frac{p_0-p}{p_0}}\biggr\}\leqslant C\epsilon^{\frac{p}{2}}.
		\end{aligned}
		\end{equation*}
Other terms can be estimated similar. Hence we deduce \eqref{estimate of zeta1}.

Next we consider the second estimate. Set $\hat{\zeta}^{2,\epsilon}\coloneqq\zeta^{\epsilon}-\bar{\zeta}-\zeta^1$, $\hat{\kappa}^{i,2,\epsilon}\coloneqq\kappa^{i,\epsilon}-\bar{\kappa}^i-\kappa^{i,1}$, $\hat{\tilde{\kappa}}^{2,\epsilon}\coloneqq\tilde{\kappa}^{i,\epsilon}-\bar{\tilde{\kappa}}^i-\tilde{\kappa}^{i,1}$, which satisfy
\begin{equation*}\label{equation of hat zeta 2}
	\left\{
	\begin{aligned}
		-d\hat{\zeta}^{2,\epsilon}_t&=\biggl\{l_x(t)\hat{x}^{2,\epsilon}_t+l_y(t)\hat{y}^{2,\epsilon}_t+l_{z^i}(t)\hat{z}^{i,2,\epsilon}_t
+l_{\tilde{z}^i}(t)\int_{E_i}\hat{\tilde{z}}^{i,2,\epsilon}_{(t,e)}\nu_i(de)+l_{\kappa^i}(t)\hat{\kappa}^{i,2,\epsilon}_t\\
		&\qquad+l_{\tilde{\kappa}^i}(t)\int_{E_i}\hat{\tilde{\kappa}}^{i,2,\epsilon}_{(t,e)}\nu_i(de)+A^{\epsilon}_1(t)\biggr\}dt-\hat{\kappa}^{i,2,\epsilon}_tdW^i_t
+\int_{E_i}\hat{\tilde{\kappa}}^{i,2,\epsilon}_{(t,e)}\tilde{N}_i(de,dt),\quad t\in[0,T],\\
		\hat{\zeta}^{2,\epsilon}_T&=\tilde{\varphi}_x(T,0)\hat{x}^{2,\epsilon}_T+\tilde{\varphi}_y(T,0)\hat{y}^{2,\epsilon}_0+[\tilde{\varphi}_x(T,0)-\varphi_x(\bar{x}_T,\bar{y}_0)]x^1_T
+[\tilde{\varphi}_y(T,0)-\varphi_y(\bar{x}_T,\bar{y}_0)]y^1_0,
	\end{aligned}
	\right.
\end{equation*}
where
\begin{equation*}
	\begin{aligned}
		A^{\epsilon}_1(t)&\coloneqq(\tilde{l}_x(t)-l_x(t))\hat{x}^{1,\epsilon}_t+(\tilde{l}_y(t)-l_y(t))\hat{y}^{1,\epsilon}_t+(\tilde{l}_{z^i}(t)-l_{z^i}(t))\hat{z}^{i,1,\epsilon}_t\\
		&\quad+(\tilde{l}_{\tilde{z}^i}(t)-l_{\tilde{z}^i}(t))\int_{E_i}\hat{\tilde{z}}^{i,1,\epsilon}_{(t,e)}\nu_i(de)+(\tilde{l}_{\kappa^i}(t)-l_{\kappa^i}(t))\hat{\kappa}^{i,1,\epsilon}_t
+(\tilde{l}_{\tilde{\kappa}^i}(t)-l_{\tilde{\kappa}^i}(t))\int_{E_i}\hat{\tilde{\kappa}}^{i,1,\epsilon}_{(t,e)}\nu_i(de)\\
		&\quad+\big(\delta l(t)+\beta_t\delta\sigma_i(t)+l_{z^i}(t)\Delta^i_t+l_{\kappa^i}(t)\pi^i_t\big)\mathbbm{1}_{[\bar{t},\bar{t}+\epsilon]}(t),
	\end{aligned}
\end{equation*}
and $\tilde{l}_x(t)$ is defined in \eqref{equation of hat x y zeta 1}.
	
	By Lemma \ref{lemma stochastic Lip BSDEP}, we have
	\begin{equation*}
		\begin{aligned}
			&\mathbb{E}\bigg[\sup_{0\leqslant t\leqslant T}|\hat{\zeta}^{2,\epsilon}_t|^p+\Big(\int_0^T|\hat{\kappa}^{i,2,\epsilon}_t|^2dt\Big)^{\frac{p}{2}}+\Big(\int_0^T\int_{E_i}|\hat{\tilde{\kappa}}^{i,2,\epsilon}_{(t,e)}|^2\nu_i(de,dt)\Big)^{\frac{p}{2}}\bigg]\\
			&\leqslant C\bigg(\mathbb{E}\bigg[\Big|\tilde{\varphi}_x(T,0)\hat{x}^{2,\epsilon}_T+\tilde{\varphi}_y(T,0)\hat{y}^{2,\epsilon}_0+[\tilde{\varphi}_x(T,0)-\varphi_x(\bar{x}_T,\bar{y}_0)]x^1_T\Big|^{p\bar{q}^2}\\
			&\qquad+\Big(\int_0^T\Big|l_x(t)\hat{x}^{2,\epsilon}_t+l_y(t)\hat{y}^{2,\epsilon}_t+l_{z^i}(t)\hat{z}^{i,2,\epsilon}_t
+l_{\tilde{z}^i}(t)\int_{E_i}\hat{\tilde{z}}^{i,2,\epsilon}_{(t,e)}\nu_i(de)+A^{\epsilon}_1(t)\Big|dt\Big)^{p\bar{q}^2}\bigg]\bigg)^{\frac{1}{\bar{q}^2}}\\
&\leqslant C\bigg(\mathbb{E}\bigg[\sup_{0\leqslant t\leqslant T}\Big(|\hat{x}^{2,\epsilon}_t|^{p\bar{q}^2}+|\hat{y}^{2,\epsilon}_t|^{p\bar{q}^2}\Big)
+\Big(\int_0^T|\hat{z}^{i,2,\epsilon}_t|^2dt\Big)^{\frac{p\bar{q}^2}{2}}+\Big(\int_0^T\int_{E_i}|\hat{\tilde{z}}^{i,2,\epsilon}_{(t,e)}|^2\nu_i(de)dt\Big)^{\frac{p\bar{q}^2}{2}}\\
			&\qquad+\big([\tilde{\varphi}_x(T,0)-\varphi_x(\bar{x}_T,\bar{y}_0)]x^1_T\big)^{p\bar{q}^2}+\Big(\int_0^TA^{\epsilon}_1(t)dt\Big)^{p\bar{q}^2}\bigg]\bigg)^{\frac{1}{\bar{q}^2}}\\
			&\leqslant C\bigg(\epsilon^{p\bar{q}^2}+\mathbb{E}\Big(\big(|\hat{x}^{1,\epsilon}_T|+|\hat{y}^{1,\epsilon}_0|\big)x^1_T\Big)^{p\bar{q}^2}
+\mathbb{E}\bigg(\int_0^TA^\epsilon_1(t)dt\bigg)^{p\bar{q}^2}\bigg)^{\frac{1}{\bar{q}^2}}\\
			&\leqslant C\Biggl\{\epsilon^{p\bar{q}^2}+\bigg(\mathbb{E}\bigg[\sup_{0\leqslant t\leqslant T}\big(|\hat{x}^{1,\epsilon}_t|^{2p\bar{q}^2}+|\hat{y}^{1,\epsilon}_t|^{2p\bar{q}^2}\big)\bigg]\bigg)^{\frac{1}{2}}
\bigg(\mathbb{E}\bigg[\sup_{0\leqslant t\leqslant T}|x^1_t|^{2p\bar{q}^2}\bigg]\bigg)^{\frac{1}{2}}\\
            &\qquad +\mathbb{E}\bigg(\int_0^TA^{\epsilon}_1(t)dt\bigg)^{p\bar{q}^2}\Biggr\}^{\frac{1}{\bar{q}^2}}
            \leqslant C\bigg(\epsilon^{p\bar{q}^2}+\mathbb{E}\bigg(\int_0^TA^{\epsilon}_1(t)dt\bigg)^{p\bar{q}^2}\bigg)^{\frac{1}{\bar{q}^2}}.
		\end{aligned}
	\end{equation*}
Next we estimate $\mathbb{E}\big(\int_0^TA^{\epsilon}_1(t)dt\big)^{p\bar{q}^2}$. It follows from the definition of $\tilde{l}_{\kappa^i}(t)$, $\tilde{l}_{\tilde{\kappa}^i}(t)$ and Assumption \ref{assumption of l(t)} that
\begin{equation*}
	\begin{aligned}
		&|\tilde{l}_{\kappa^i}(t)-l_{\kappa^i}(t)|+|\tilde{l}_{\tilde{\kappa}^i}(t)-l_{\tilde{\kappa}^i}(t)|\\
		&\leqslant|\tilde{l}_{\kappa^i}(t)-l_{\kappa^i}(t,\bar{\Xi}_t,u^{\epsilon}_t)|+|\delta l_{\kappa^i}(t)|\mathbbm{1}_{[\bar{t},\bar{t}+\epsilon]}+|\tilde{l}_{\tilde{\kappa}^i}(t)-l_{\tilde{\kappa}^i}(t,\bar{\Xi}_t,u^{\epsilon}_t)|+|\delta l_{\tilde{\kappa}^i}(t)|\mathbbm{1}_{[\bar{t},\bar{t}+\epsilon]}(t)\\
		&\leqslant C\bigg[|\hat{x}^{1,\epsilon}_t|+|\hat{y}^{1,\epsilon}_t|+|\hat{z}^{i,1,\epsilon}_t|+\int_{E_i}|\hat{\tilde{z}}^{i,1,\epsilon}_{(t,e)}|\nu_i(de)+|\hat{\kappa}^{i,1,\epsilon}_t|
+\int_{E_i}|\hat{\tilde{\kappa}}^{i,1,\epsilon}_{(t,e)}|\nu_i(de)\\
		&\qquad+\Big(1+|\bar{\kappa}^i_t|+\int_{E_i}|\bar{\tilde{\kappa}}^i_{(t,e)}|\nu_i(de)+|u_t|+|\bar{u}_t|\Big)\mathbbm{1}_{[\bar{t},\bar{t}+\epsilon]}(t)\bigg].
	\end{aligned}
\end{equation*}
Hence,
\begin{equation*}
	\begin{aligned}
		&\mathbb{E}\bigg(\int_0^T(\tilde{l}_{\kappa}(t)-l_{\kappa}(t))\hat{\kappa}^{i,1,\epsilon}_tdt\bigg)^{p\bar{q}^2}\\
		&\leqslant C\mathbb{E}\bigg[\Big(\int_0^T\Big[|\hat{x}^{1,\epsilon}_t|+|\hat{y}^{1,\epsilon}_t|+|\hat{z}^{i,1,\epsilon}_t|+\int_{E_i}|\hat{\tilde{z}}^{i,1,\epsilon}_{(t,e)}|\nu_i(de)
+|\hat{\kappa}^{i,1,\epsilon}_t|+\int_{E_i}|\hat{\tilde{\kappa}}^{i,1,\epsilon}_{(t,e)}|\nu_i(de)\\
		&\qquad\quad+\Big(1+|\bar{\kappa}^i_t|+\int_{E_i}|\bar{\tilde{\kappa}}^i_{(t,e)}|\nu_i(de)+|u_t|+|\bar{u}_t|\Big)\mathbbm{1}_{[\bar{t},\bar{t}+\epsilon]}(t)\Big]\hat{\kappa}^{i,1,\epsilon}_tdt\Big)^{p\bar{q}^2}\bigg]\\
		&\leqslant C\mathbb{E}\bigg[\Big(\int_0^T\Big[|\hat{x}^{1,\epsilon}_t|^2+|\hat{y}^{1,\epsilon}_t|^2+|\hat{z}^{i,1,\epsilon}_t|^2+\int_{E_i}|\hat{\tilde{z}}^{i,1,\epsilon}_{(t,e)}|^2\nu_i(de)
+|\hat{\kappa}^{i,1,\epsilon}_t|^2+\int_{E_i}|\hat{\tilde{\kappa}}^{i,1,\epsilon}_{(t,e)}|^2\nu_i(de)\\
		&\qquad\quad+\Big(1+|\bar{\kappa}^i_t|+\int_{E_i}|\bar{\tilde{\kappa}}^i_{(t,e)}|\nu_i(de)+|u_t|+|\bar{u}_t|\Big)\hat{\kappa}^{i,1,\epsilon}_t\mathbbm{1}_{[\bar{t},\bar{t}+\epsilon]}(t)\Big]dt\Big)^{p\bar{q}^2}\bigg]\\
&\leqslant C\mathbb{E}\bigg[\bigg(\sup_{0\leqslant t\leqslant T}(|\hat{x}^{1,\epsilon}_t|^2+|\hat{y}^{1,\epsilon}_t|^2)+\int_0^T|\hat{z}^{i,1,\epsilon}_t|^2dt+\int_0^T\int_{E_i}|\hat{\tilde{z}}^{i,1,\epsilon}_{(t,e)}|^2\nu_i(de)dt\\
		&\qquad\quad+\int_0^T|\hat{\kappa}^{i,1,\epsilon}_t|^2dt+\int_0^T\int_{E_i}|\hat{\tilde{\kappa}}^{i,1,\epsilon}_{(t,e)}|^2\nu_i(de)dt\\
		&\qquad\quad+\int_0^T\Big(1+|\bar{\kappa}^i_t|+\int_{E_i}|\bar{\tilde{\kappa}}^i_{(t,e)}|\nu_i(de)+|u_t|
+|\bar{u}_t|\Big)\hat{\kappa}^{i,1,\epsilon}_t\mathbbm{1}_{[\bar{t},\bar{t}+\epsilon]}(t)dt\bigg)^{p\bar{q}^2}\bigg]\\
&\leqslant C\biggl\{\epsilon^{p\bar{q}^2}+\mathbb{E}\bigg(\int_0^T\Big(1+|\bar{\kappa}^i_t|+\int_{E_i}|\bar{\tilde{\kappa}}^i_{(t,e)}|\nu_i(de)
+|u_t|+|\bar{u}_t|\Big)\hat{\kappa}^{i,1,\epsilon}_t\mathbbm{1}_{[\bar{t},\bar{t}+\epsilon]}(t)dt\bigg)^{p\bar{q}^2}\biggr\}\\
&\leqslant C\biggl\{\epsilon^{p\bar{q}^2}+\mathbb{E}\bigg(\sup_{0\leqslant t\leqslant T}(|u_t|+|\bar{u}_t|)\int_0^T|\hat{\kappa}^{i,1,\epsilon}_t|\mathbbm{1}_{[\bar{t},\bar{t}+\epsilon]}(t)dt\\
        &\qquad+\int_0^T\Big(|\bar{\kappa}^i_t|
+\int_{E_i}|\bar{\tilde{\kappa}}^i_{(t,e)}|\nu_i(de)\Big)\hat{\kappa}^{i,1,\epsilon}_t\mathbbm{1}_{[\bar{t},\bar{t}+\epsilon]}(t)dt\bigg)^{p\bar{q}^2}\biggr\}\\
		&\leqslant C\biggl\{\epsilon^{p\bar{q}^2}+\mathbb{E}\bigg[\Big(\sup_{0\leqslant t\leqslant T}(|u_t|+|\bar{u}_t|)\int_0^T|\hat{\kappa}^{i,1,\epsilon}_t|\mathbbm{1}_{[\bar{t},\bar{t}+\epsilon]}(t)dt\Big)^{p\bar{q}^2}
+\Big(\int_0^T|\bar{\kappa}^i_t||\hat{\kappa}^{i,1,\epsilon}_t|\mathbbm{1}_{[\bar{t},\bar{t}+\epsilon]}(t)dt\Big)^{p\bar{q}^2}\\
		&\qquad+\Big(\int_0^T\int_{E_i}|\bar{\tilde{\kappa}}^i_{(t,e)}||\hat{\kappa}^{i,1,\epsilon}_t|\mathbbm{1}_{[\bar{t},\bar{t}+\epsilon]}(t)\nu_i(de)dt\Big)^{p\bar{q}^2}\bigg]\biggr\}.
	\end{aligned}
\end{equation*}
Since $\mathbb{E}\big(\int_0^T|\bar{\kappa}^i_t|^2dt\big)^{p\bar{q}^2}<\infty$, it follows from the dominated convergence theorem that
$$
\varpi^{i}_1(\epsilon)\coloneqq\mathbb{E}\Big(\int_{\bar{t}}^{\bar{t}+\epsilon}|\bar{\kappa}^i_t|^2dt\Big)^{p\bar{q}^2}\downarrow0,\quad as\,\,\epsilon\downarrow0.
$$
Similarly,
$$
\varpi^{i}_2(\epsilon)\coloneqq\mathbb{E}\Big(\int_{\bar{t}}^{\bar{t}+\epsilon}\int_{E_i}|\bar{\tilde{\kappa}}^i_{(t,e)}|^2\nu_i(de)dt\Big)^{p\bar{q}^2}\downarrow0,\quad as\,\,\epsilon\downarrow0.
$$

A straight-forward argument gives only
\begin{equation*}\label{equation estimate not enough}
	\begin{aligned}
		&\mathbb{E}\bigg[\Big(\int_0^T|\bar{\kappa}^i_t||\hat{\kappa}^{i,1,\epsilon}_t|\mathbbm{1}_{[\bar{t},\bar{t}+\epsilon]}(t)dt\Big)^{p\bar{q}^2}
+\Big(\int_0^T\int_{E_i}|\bar{\tilde{\kappa}}^i_{(t,e)}||\hat{\kappa}^{i,1,\epsilon}_t|\mathbbm{1}_{[\bar{t},\bar{t}+\epsilon]}(t)\nu_i(de)dt\Big)^{p\bar{q}^2}\bigg]\\
		&\leqslant C\Biggl\{\mathbb{E}\bigg[\Big(\int_0^T|\bar{\kappa}^i_t|^2\mathbbm{1}_{[\bar{t},\bar{t}+\epsilon]}(t)dt\Big)^{\frac{p\bar{q}^2}{2}}\Big(\int_0^T|\hat{\kappa}^{i,1,\epsilon}_t|^2dt\Big)^{\frac{p\bar{q}^2}{2}}\bigg]\\
		&\qquad+\mathbb{E}\bigg[\Big(\int_0^T\int_{E_i}|\bar{\tilde{\kappa}}^i_{(t,e)}|^2\mathbbm{1}_{[\bar{t},\bar{t}+\epsilon]}(t)\nu_i(de)dt\Big)^{\frac{p\bar{q}^2}{2}}
\Big(\int_0^T\int_{E_i}|\hat{\tilde{\kappa}}^{i,1,\epsilon}_{(t,e)}|^2\nu_i(de)dt\Big)^{\frac{p\bar{q}^2}{2}}\bigg]\Biggr\}\\
		&\leqslant C \Biggl\{\bigg(\mathbb{E}\bigg(\int_0^T|\bar{\kappa}^i_t|^2\mathbbm{1}_{[\bar{t},\bar{t}+\epsilon]}(t)dt\bigg)^{p\bar{q}^2}\bigg)^{\frac{1}{2}}
\bigg(\mathbb{E}\bigg(\int_0^T|\hat{\kappa}^{i,1,\epsilon}_t|^2dt\bigg)^{p\bar{q}^2}\bigg)^{\frac{1}{2}}\\
		&\qquad+\bigg(\mathbb{E}\bigg(\int_0^T\int_{E_i}|\bar{\tilde{\kappa}}^i_{(t,e)}|^2\mathbbm{1}_{[\bar{t},\bar{t}+\epsilon]}(t)\nu_i(de)dt\bigg)^{p\bar{q}^2}\bigg)^{\frac{1}{2}}
\bigg(\mathbb{E}\bigg(\int_0^T\int_{E_i}|\hat{\tilde{\kappa}}^{i,1,\epsilon}_{(t,e)}|^2\nu_i(de)dt\bigg)^{p\bar{q}^2}\bigg)^{\frac{1}{2}}\Biggr\}\\
		&\leqslant C\Big(\sqrt{\varpi^i_1(\epsilon)}+\sqrt{\varpi^i_2(\epsilon)}\Big)\epsilon^{\frac{p\bar{q}^2}{2}},
	\end{aligned}
\end{equation*}
which is not the desired result since we need an estimate which leads to a convergence speed quiker than $\epsilon^{p\bar{q}^2}$ as $\epsilon\downarrow0$.

Inspired by Buckdahn et. al \cite{BLLW24}, for $M\geqslant1$, we introduce the deterministic set $$\Gamma_M\coloneqq\Big\{t\in[0,T]:\mathbb{E}|\bar{\kappa}^1_t|^2\vee\mathbb{E}|\bar{\kappa}^2_t|^2\vee\mathbb{E}\int_{E_1}|\bar{\tilde{\kappa}}^1_{(t,e)}|^2\nu_1(de)
\vee\mathbb{E}\int_{E_2}|\bar{\tilde{\kappa}}^2_{(t,e)}|^2\nu_2(de)\leqslant M\Big\}.$$
Moreover, as
\begin{equation*}
	\int_0^T\mathbb{E}|\bar{\kappa}^i_t|^2dt+\int_0^T\mathbb{E}\int_{E_i}|\bar{\tilde{\kappa}}^i_{(t,e)}|^2\nu_i(de)dt
=\mathbb{E}\Big[\int_0^T|\bar{\kappa}^i_t|^2dt+\int_0^T\int_{E_i}|\bar{\tilde{\kappa}}^i_{(t,e)}|^2\nu_i(de)dt\Big]<\infty,
\end{equation*}
we have $\mathbbm{1}_{\cup_{M\geqslant1}\Gamma_M}(t)=1$, $dt\mbox{-}a.e..$

For arbitrarily fixed $M\geqslant1$, let $E_\epsilon\coloneqq[\bar{t},\bar{t}+\epsilon]\cap\Gamma_M\subset\Gamma_M$. we have $|E_\epsilon|\leqslant\epsilon$.

For all $1<p\bar{q}^2<2$, from H\"{o}lder's inequality, we get
\begin{equation*}
	\begin{aligned}
		&\mathbb{E}\bigg[\Big(\int_0^T|\bar{\kappa}^i_t||\hat{\kappa}^{i,1,\epsilon}_t|\mathbbm{1}_{E_\epsilon}(t)dt\Big)^{p\bar{q}^2}
+\Big(\int_0^T\int_{E_i}|\bar{\tilde{\kappa}}^i_{(t,e)}||\hat{\kappa}^{i,1,\epsilon}_t|\mathbbm{1}_{E_\epsilon}(t)\nu_i(de)dt\Big)^{p\bar{q}^2}\bigg]\\
		&\leqslant\mathbb{E}\bigg[\Big(\int_{E_\epsilon}|\bar{\kappa}^i_t|^2dt\Big)^{\frac{p\bar{q}^2}{2}}\Big(\int_0^T|\hat{\kappa}^{i,1,\epsilon}_t|^2dt\Big)^{\frac{p\bar{q}^2}{2}}\bigg]\\
		&\quad +\mathbb{E}\bigg[\Big(\int_{E_\epsilon}\int_{E_i}|\bar{\tilde{\kappa}}^i_{(t,e)}|^2\nu_i(de)dt\Big)^{\frac{p\bar{q}^2}{2}}
\Big(\int_0^T\int_{E_i}|\hat{\tilde{\kappa}}^{i,1,\epsilon}_{(t,e)}|^2\nu_i(de)dt\Big)^{\frac{p\bar{q}^2}{2}}\bigg]\\
		&\leqslant\bigg(\mathbb{E}\int_{E_\epsilon}|\bar{\kappa}^i_t|^2dt\bigg)^{\frac{p\bar{q}^2}{2}}
\bigg(\mathbb{E}\bigg(\int_0^T|\hat{\kappa}^{i,1,\epsilon}_t|^2dt\bigg)^{\frac{p\bar{q}^2}{2-p\bar{q}^2}}\bigg)^{\frac{2-p\bar{q}^2}{2}}\\
		\end{aligned}
	\end{equation*}
	\begin{equation*}
		    \begin{aligned}
        &\quad +\bigg(\mathbb{E}\int_{E_\epsilon}\int_{E_i}|\bar{\tilde{\kappa}}^i_{(t,e)}|^2\nu_i(de)dt\bigg)^{\frac{p\bar{q}^2}{2}}
\bigg(\mathbb{E}\bigg(\int_0^T\int_{E_i}|\hat{\tilde{\kappa}}^{i,1,\epsilon}_{(t,e)}|^2\nu_i(de)dt\bigg)^{\frac{p\bar{q}^2}{2-p\bar{q}^2}}\bigg)^{\frac{2-p\bar{q}^2}{2}}\\
        &=\bigg(\int_{E_\epsilon}\mathbb{E}|\bar{\kappa}^i_t|^2dt\bigg)^{\frac{p\bar{q}^2}{2}}
\bigg(\mathbb{E}\bigg(\int_0^T|\hat{\kappa}^{i,1,\epsilon}_t|^2dt\bigg)^{\frac{p\bar{q}^2}{2-p\bar{q}^2}}\bigg)^{\frac{2-p\bar{q}^2}{2}}\\
		&\quad
+\bigg(\int_{E_\epsilon}\mathbb{E}\int_{E_i}|\bar{\tilde{\kappa}}^i_{(t,e)}|^2\nu_i(de)dt\bigg)^{\frac{p\bar{q}^2}{2}}
\bigg(\mathbb{E}\bigg(\int_0^T\int_{E_i}|\hat{\tilde{\kappa}}^{i,1,\epsilon}_{(t,e)}|^2\nu_i(de)dt\bigg)^{\frac{p\bar{q}^2}{2-p\bar{q}^2}}\bigg)^{\frac{2-p\bar{q}^2}{2}}.
	\end{aligned}
\end{equation*}
Noticing
$$
\bigg(\int_{E_\epsilon}\mathbb{E}|\bar{\kappa}^i_t|^2dt\bigg)^{\frac{p\bar{q}^2}{2}}
\vee\bigg(\int_{E_\epsilon}\mathbb{E}\int_{E_i}|\bar{\tilde{\kappa}}^i_{(t,e)}|^2\nu_i(de)dt\bigg)^{\frac{p\bar{q}^2}{2}}
\leqslant M^{\frac{p\bar{q}^2}{2}}\epsilon^{\frac{p\bar{q}^2}{2}},
$$
we can deduce that for all $1<p\bar{q}^2<2$,
\begin{equation*}
	\begin{aligned}
		&\mathbb{E}\bigg[\Big(\int_0^T|\bar{\kappa}^i_t||\hat{\kappa}^{i,1,\epsilon}_t|\mathbbm{1}_{E_\epsilon}(t)dt\Big)^{p\bar{q}^2}
+\Big(\int_0^T\int_{E_i}|\bar{\tilde{\kappa}}^i_{(t,e)}||\hat{\kappa}^{i,1,\epsilon}_t|\mathbbm{1}_{E_\epsilon}(t)\nu_i(de)dt\Big)^{p\bar{q}^2}\bigg]\leqslant CM^{\frac{p\bar{q}^2}{2}}\epsilon^{p\bar{q}^2}.
	\end{aligned}
\end{equation*}
Then, replacing $\mathbbm{1}_{[\bar{t},\bar{t}+\epsilon]}$ by $\mathbbm{1}_{E_\epsilon}$, we have
\begin{equation*}
	\begin{aligned}
		&\mathbb{E}\bigg(\int_0^T(\tilde{l}_\kappa(t)-l_\kappa(t))\hat{\kappa}^{i,1,\epsilon}_tdt\bigg)^{p\bar{q}^2}\\
		&\leqslant C\biggl\{\epsilon^{p\bar{q}^2}
        +\mathbb{E}\bigg[\Big(\sup_{0\leqslant t\leqslant T}\big(|u_t|+|\bar{u}_t|\big)\int_0^T|\hat{\kappa}^{i,1,\epsilon}_t|\mathbbm{1}_{[\bar{t},\bar{t}+\epsilon]}(t)dt\Big)^{p\bar{q}^2}
+\Big(\int_0^T|\bar{\kappa}^i_t||\hat{\kappa}^{i,1,\epsilon}_t|\mathbbm{1}_{[\bar{t},\bar{t}+\epsilon]}(t)dt\Big)^{p\bar{q}^2}\\
		&\qquad+\Big(\int_0^T\int_{E_i}|\bar{\tilde{\kappa}}^i_{(t,e)}||\hat{\kappa}^{i,1,\epsilon}_t|\mathbbm{1}_{[\bar{t},\bar{t}+\epsilon]}(t)\nu_i(de)dt\Big)^{p\bar{q}^2}\bigg]\biggr\}\\
		&\leqslant C\biggl\{\epsilon^{p\bar{q}^2}+M^{\frac{p\bar{q}^2}{2}}\epsilon^{p\bar{q}^2}+\Big(\mathbb{E}\Big[\sup_{0\leqslant t\leqslant T}(|u_t|^2+|\bar{u}_t|^2)\Big]\Big)^{\frac{p\bar{q}^2}{2}}\\
		&\qquad\times\bigg(\mathbb{E}\bigg(\int_0^T\mathbbm{1}^2_{E_\epsilon}(t)dt\bigg)^{2p\bar{q}^2}\bigg)^{\frac{1}{4}}
\bigg(\mathbb{E}\bigg(\int_0^T|\hat{\kappa}^{i,1,\epsilon}_t|^2dt\bigg)^{2p\bar{q}^2}\bigg)^{\frac{1}{4}}\biggr\}\leqslant C_M\epsilon^{p\bar{q}^2}.
	\end{aligned}
\end{equation*}
Similarly,
\begin{equation*}
		\mathbb{E}\bigg(\int_0^T\Big[(\tilde{l}_y(t)-l_y(t))\hat{y}^{1,\epsilon}_t+(\tilde{l}_{z^i}(t)-l_{z^i}(t))\hat{z}^{i,1,\epsilon}_t
+(\tilde{l}_{\tilde{z}^i}(t)-l_{\tilde{z}^i}(t))\int_{E_i}\hat{\tilde{z}}^{i,1,\epsilon}_{(t,e)}\nu_i(de)\Big]dt\bigg)^{p\bar{q}^2}\leqslant C_M\epsilon^{p\bar{q}^2}.
\end{equation*}
Other terms can be estimated in the classical way. The proof is complete.
\end{proof}

\section*{Appendix D: Proof of Lemma \ref{lemma estimate of zeta2,hat zeta 3}}

\begin{proof}
	The first estimate can be obtained by using the same technique in Lemma \ref{lemma estimate of zeta1 hat zeta2}. However, we should note that we fail to make the order reach $O(\epsilon^2)$ due to the appearance of the $\mathbb{E}\big(\int_0^T\beta^i_t\delta\sigma_i(t)\mathbbm{1}_{E_\epsilon}dt\big)^2$ term, and so does the second estimate. Next, we consider the third estimate. Use the notations
	\begin{equation*}	
\hat{\zeta}^{3,\epsilon}\coloneqq\zeta^{\epsilon}-\bar{\zeta}-\zeta^1-\zeta^2,\ \hat{\kappa}^{i,3,\epsilon}\coloneqq\kappa^{i,\epsilon}-\bar{\kappa}^i-\kappa^{i,1}-\kappa^{i,2},\ \hat{\tilde{\kappa}}^{i,3,\epsilon}\coloneqq\tilde{\kappa}^{i,\epsilon}-\bar{\tilde{\kappa}}^i-\tilde{\kappa}^{i,1}-\tilde{\kappa}^{i,2},\quad i=1,2,
	\end{equation*}
we have the following equation:
\begin{equation*}
	\left\{
	\begin{aligned}
		-d\hat{\zeta}^{3,\epsilon}_t&=\Bigl\{l_x(t)\hat{x}^{3,\epsilon}_t+l_y(t)\hat{y}^{3,\epsilon}_t+l_{z^i}(t)\hat{z}^{i,3,\epsilon}_t+l_{\tilde{z}^i}(t)\int_{E_i}\hat{\tilde{z}}^{i,3,\epsilon}_{(t,e)}\nu_i(de)
+l_{\kappa^i}(t)\hat{\kappa}^{i,3,\epsilon}_t\\
		&\quad\ +l_{\tilde{\kappa}^i}(t)\int_{E_i}\hat{\tilde{\kappa}}^{i,3,\epsilon}_{(t,e)}\nu_i(de)+A^{\epsilon}_2(t)\Bigr\}dt
-\hat{\kappa}^{i,3,\epsilon}_tdW^i_t-\int_{E_i}\hat{\tilde{\kappa}}^{i,3,\epsilon}_{(t,e)}\tilde{N}_i(de,dt),\quad t\in[0,T],\\
		\hat{\zeta}^{3,\epsilon}_T&=\varphi_x(\bar{x}_T,\bar{y}_0)\hat{x}^{3,\epsilon}_T+\varphi_y(\bar{x}_T,\bar{y}_0)\hat{y}^{3,\epsilon}_0+\Phi^{\epsilon}_2(T),
	\end{aligned}
	\right.
\end{equation*}
where
\begin{equation*}
	\begin{aligned}
		&A^{\epsilon}_2(t)\coloneqq\Bigl\{\delta l_x(t,\Delta^1,\Delta^2,\pi^1,\pi^2)\hat{x}^{1,\epsilon}_t+\delta l_y(t,\Delta^1,\Delta^2,\pi^1,\pi^2)\hat{y}^{1,\epsilon}_t\\
        &\quad +\delta l_{z^i}(t,\Delta^1,\Delta^2,\pi^1,\pi^2)(\hat{z}^{i,1,\epsilon}_t-\Delta^i_t\mathbbm{1}_{E_\epsilon}(t))
        +\delta l_{\tilde{z}^i}(t,\Delta^1,\Delta^2,\pi^1,\pi^2)\int_{E_i}\hat{\tilde{z}}^{i,1,\epsilon}_{(t,e)}\nu_i(de)\\
		&\quad+\delta l_{\kappa^i}(t,\Delta^1,\Delta^2,\pi^1,\pi^2)(\hat{\kappa}^{i,1,\epsilon}_t-\pi^i_t\mathbbm{1}_{E_\epsilon}(t))
        +\delta l_{\tilde{\kappa}^i}(t,\Delta^1,\Delta^2,\pi^1,\pi^2)\int_{E_i}\hat{\tilde{\kappa}}^{i,1,\epsilon}_{(t,e)}\nu_i(de)\Bigr\}\mathbbm{1}_{E_\epsilon}(t)\\
&\quad+\frac{1}{2}\check{\Xi}(t)\widetilde{D^2l}(t)\check{\Xi}(t)^\top-\frac{1}{2}\tilde{\Xi}(t)D^2l(t)\tilde{\Xi}(t)^\top,\\
		&\widetilde{D^2l}(t)\coloneqq2\int_0^1\int_0^1\theta D^2l\big(t,\Xi(t,\Delta^1_t\mathbbm{1}_{E_\epsilon}(t),\Delta^2_t\mathbbm{1}_{E_\epsilon}(t),\pi^1_t\mathbbm{1}_{E_\epsilon}(t),\pi^2_t\mathbbm{1}_{E_\epsilon}(t))\\
		&\qquad\qquad\qquad+\lambda\theta(\Xi^\epsilon(t)-\Xi(t,\Delta^1_t\mathbbm{1}_{E_\epsilon}(t),\Delta^2_t\mathbbm{1}_{E_\epsilon}(t),
\pi^1_t\mathbbm{1}_{E_\epsilon}(t),\pi_t^2\mathbbm{1}_{E_\epsilon}(t))),u^{\epsilon}\big)d\lambda d\theta,\\
		&\widetilde{D^2\varphi}(T,0)\coloneqq2\int_{0}^{1}\int_{0}^{1}\theta D^2\varphi(\bar{x}_T+\rho\theta\hat{x}^{1,\epsilon}_T,\bar{y}_0+\rho\theta\hat{y}^{1,\epsilon}_0)d\theta d\rho,
	\end{aligned}
\end{equation*}
with
\begin{equation*}
	\begin{aligned}
		&\Xi(\cdot,\Delta^1,\Delta^2,\pi^1,\pi^2)\equiv\big[\bar{x},\bar{y},\bar{z}^1+\Delta^1,\bar{z}^2+\Delta^2,\bar{\tilde{z}}^1,\bar{\tilde{z}}^2,\bar{\kappa}^1+\pi^1,\bar{\kappa}^2
+\pi^2,\bar{\tilde{\kappa}}^1,\bar{\tilde{\kappa}}^2\big],\\
        \check{\Xi}&\equiv\Big[\check{\Xi}_2,\hat{\kappa}^{1,1,\epsilon}_t-\pi^1_t\mathbbm{1}_{[\bar{t},\bar{t}+\epsilon]},\hat{\kappa}^{2,1,\epsilon}_t-\pi^2_t\mathbbm{1}_{[\bar{t},\bar{t}+\epsilon]},
\int_{E_1}\hat{\tilde{\kappa}}^{1,1,\epsilon}_{(t,e)}\nu_1(de),\int_{E_2}\hat{\tilde{\kappa}}^{2,1,\epsilon}_{(t,e)}\nu_2(de)\Big],\\
		\tilde{\Xi}&\equiv\Big[\tilde{\Xi}_2,\kappa^{1,1}_t-\pi^1_t\mathbbm{1}_{[\bar{t},\bar{t}+\epsilon]},\kappa^{2,1}_t-\pi^2_t\mathbbm{1}_{[\bar{t},\bar{t}+\epsilon]},
\int_{E_1}\tilde{\kappa}^{1,1}_{(t,e)}\nu_1(de),\int_{E_2}\tilde{\kappa}^{2,1}_{(t,e)}\nu_2(de)\Big],\\
		\Phi^{\epsilon}_2(T)&\equiv\frac{1}{2}\text{Tr}\biggl\{\widetilde{D^2\varphi}(T,0)\begin{bmatrix}
			\hat{x}^{1,\epsilon}_T\\
			\hat{y}^{1,\epsilon}_0
		\end{bmatrix}\begin{bmatrix}
			\hat{x}^{1,\epsilon}_T& \hat{y}^{1,\epsilon}_0
		\end{bmatrix}\biggr\}-\frac{1}{2}\text{Tr}\biggl\{D^2\varphi(\bar{x}_T,\bar{y}_0)\begin{bmatrix}
			x^1_T\\y^1_0
		\end{bmatrix}\begin{bmatrix}
			x^1_T&y^1_0
		\end{bmatrix}\biggr\}.
	\end{aligned}
\end{equation*}
It follows from Lemma \ref{lemma stochastic Lip BSDEP} that, for all $p>1$,
\begin{equation*}\label{equation estiamte hat zeta3 1}
	\begin{aligned}
		&\mathbb{E}\bigg[\sup_{0\leqslant t\leqslant T}|\hat{\zeta}^{3,\epsilon}_t|^p+\Big(\int_0^T|\hat{\kappa}^{i,3,\epsilon}_t|^2dt\Big)^{\frac{p}{2}}
+\Big(\int_0^T\int_{E_i}|\hat{\tilde{\kappa}}^{i,3,\epsilon}_{(t,e)}|^2\nu_i(de)dt\Big)^{\frac{p}{2}}\bigg]\\
		\leqslant&\ C\bigg(\mathbb{E}\bigg[\sup_{0\leqslant t\leqslant T}\big|\varphi_x(\bar{x}_T,\bar{y}_0)\hat{x}^{3,\epsilon}_T+\varphi_y(\bar{x}_T,\bar{y}_0)\hat{y}^{3,\epsilon}_0+\Phi^{\epsilon}_2(T)\big|^{p\bar{q}^2}\\
		&\qquad+\Big(\int_0^T\Big\vert l_x(t)\hat{x}^{3,\epsilon}_t+l_y(t)\hat{y}^{3,\epsilon}_t+l_{z^i}(t)\hat{z}^{i,3,\epsilon}_t+l_{\tilde{z}^i}(t)\int_{E_i}\hat{\tilde{z}}^{i,3,\epsilon}_{(t,e)}\nu_i(de)
+A^{\epsilon}_2(t)\Big\vert dt\Big)^{p\bar{q}^2}\bigg]\bigg)^{\frac{1}{\bar{q}^2}}.
	\end{aligned}
\end{equation*}

Noting that
\begin{equation*}
	\begin{aligned}
		&\mathbb{E}\Big\vert\tilde{\varphi}_{xx}(T,0)|\hat{x}^{1,\epsilon}_T|^2-\varphi_{xx}(\bar{x}_T,\bar{y}_0)|x^1_T|^2\Big\vert^{p\bar{q}^2}\\
		\leqslant&\ C\mathbb{E}\Big\vert\tilde{\varphi}_{xx}(T,0)(|\hat{x}^{1,\epsilon}_T|^2-|x^1_T|^2)+[\tilde{\varphi}_{xx}(T,0)-\varphi_{xx}(\bar{x}_T,\bar{y}_0)]|x^1_T|^2\Big\vert^{p\bar{q}^2}\\
		\leqslant&\ C\mathbb{E}\Big\vert\tilde{\varphi}_{xx}(T,0)|\hat{x}^{2,\epsilon}_T|(|\hat{x}^{1,\epsilon}_T|+|x^1_T|)+(|\hat{x}^{1,\epsilon}_T|+|\hat{y}^{1,\epsilon}_0|)|x^1_T|^2\Big\vert^{p\bar{q}^2}
=o(\epsilon^{p\bar{q}^2}),
	\end{aligned}
\end{equation*}
and applying classical technical, we can deduce that
\begin{equation*}
	\begin{aligned}
		&\mathbb{E}\bigg[\sup_{0\leqslant t\leqslant T}\big|\varphi_x(\bar{x}_T,\bar{y}_0)\hat{x}^{3,\epsilon}_T+\varphi_y(\bar{x}_T,\bar{y}_0)\hat{y}^{3,\epsilon}_0+\Phi^{\epsilon}_2(T)\big|^{p\bar{q}^2}\\
		&\quad+\Big(\int_0^T\Big\vert l_x(t)\hat{x}^{3,\epsilon}_t+l_y(t)\hat{y}^{3,\epsilon}_t+l_{z^i}(t)\hat{z}^{i,3,\epsilon}_t+l_{\tilde{z}^i}(t)\int_{E_i}\hat{\tilde{z}}^{i,3,\epsilon}_{(t,e)}\nu_i(de)\Big\vert dt\Big)^{p\bar{q}^2}\bigg]=o(\epsilon^{p\bar{q}^2}).
	\end{aligned}
\end{equation*}

Next, we want to show that $\mathbb{E}\big(\int_0^TA^{\epsilon}_2(t)dt\big)^{p\bar{q}^2}=o(\epsilon^{p\bar{q}^2})$. Different from \cite{HJX22} and \cite{ZS23}, caused by our quadratic-exponential growth feature, some new tools are introduced to get the desired result. We only estimate the most important and difficult terms as follows.

Recall the relationship that $\hat{\kappa}^{i,1,\epsilon}_t-\pi^i_t\mathbbm{1}_{E_\epsilon}(t)=\hat{\kappa}^{i,2,\epsilon}_t+K'_i(t)x^1_t$ and $\hat{\tilde{\kappa}}^{i,1,\epsilon}_{(t,e)}=\hat{\tilde{\kappa}}^{i,2,\epsilon}_{(t,e)}+\tilde{K}'_i(t,e)x^1_{t-}$, where $K'_i(t)=\alpha_t(\sigma_{ix}(t)+\sigma_{iy}(t)m_t)+\beta^i_t$ and $\tilde{K}'_i(t,e)=\alpha_{t-}\big[f_{ix}(t,e)+f_{iy}(t,e)m_{t-}\big]+\tilde{\beta}^i_{(t,e)}+\tilde{\beta}^i_{(t,e)}\big[f_{ix}(t,e)+f_{iy}(t,e)m_{t-}\big]$. It follows from the definition and Assumption \ref{assumption of l(t)} that
\begin{equation*}
	\left\{
	\begin{aligned}
		&|\delta l_{\kappa^i}(t,\Delta^1,\Delta^2,\pi^1,\pi^2)|\leqslant C(1+|\bar{\kappa}^i_t|+|\bar{x}_t|+|\bar{y}_t|+|u_t|+|\bar{u}_t|),\\
		&|\delta l_{\tilde{\kappa}^i}(t,\Delta^1,\Delta^2,\pi^1,\pi^2)|\leqslant C\Big(1+\int_{E_i}|e^{\theta\bar{\tilde{\kappa}}^i_{(t,e)}}|\nu_i(de)\Big),\\
		&\Big\vert\tilde{l}_{\kappa^i\kappa^i}(t)|\hat{\kappa}^{i,1,\epsilon}_t-\pi^i_t\mathbbm{1}_{[\bar{t},\bar{t}+\epsilon]}(t)|^2
-l_{\kappa^i\kappa^i}(t)|\kappa^{i,1}_t-\pi^i_t\mathbbm{1}_{[\bar{t},\bar{t}+\epsilon]}(t)|^2\Big\vert\\
		&\quad\leqslant\Big\vert\tilde{l}_{\kappa^i\kappa^i}(t)\hat{\kappa}^{i,2,\epsilon}_t\big[\hat{\kappa}^{i,1,\epsilon}_t-\pi^i_t\mathbbm{1}_{[\bar{t},\bar{t}+\epsilon]}(t)+K'_i(t)x^1_t\big] \Big\vert+|\tilde{l}_{\kappa^i\kappa^i}(t)-l_{\kappa^i\kappa^i}(t)|\big\vert K'_i(t)x^1_t\big\vert^2,\\
		&\Big\vert\tilde{l}_{\tilde{\kappa}^i\tilde{\kappa}^i}(t)\big|\int_{E_i}\hat{\tilde{\kappa}}^{i,1,\epsilon}_{(t,e)}\nu_i(de)|^2
-l_{\tilde{\kappa}^i\tilde{\kappa}^i}(t)\big|\int_{E_i}\tilde{\kappa}^{i,1}_{(t,e)}\nu_i(de)|^2\Big\vert\\
		&\quad\leqslant\Big\vert\tilde{l}_{\tilde{\kappa}^i\tilde{\kappa}^i}(t)\int_{E_i}\hat{\tilde{\kappa}}^{i,2,\epsilon}_{(t,e)}\nu_i(de)\int_{E_i}\big[\hat{\tilde{\kappa}}^{i,1,\epsilon}_{(t,e)}
+\tilde{K}'_i(t,e)x^1_{t-}\big]\nu_i(de)\Big\vert\\
        &\qquad+|\tilde{l}_{\tilde{\kappa}^i\tilde{\kappa}^i}(t)-l_{\tilde{\kappa}^i\tilde{\kappa}^i}(t)|\Big\vert\int_{E_i}\tilde{K}'_i(t,e)x^1_{t-}\nu_i(de)\Big\vert^2.
	\end{aligned}
	\right.
\end{equation*}

Recalling that in Lemma \ref{lemma estimate of zeta1 hat zeta2}, different from \cite{HJX22} and \cite{ZS23}, on $E_\epsilon=[\bar{t},\bar{t}+\epsilon]\cap\Gamma_M$, we have proved that $\hat{\kappa}^{i,2,\epsilon}\sim O(\epsilon)$. A straight forward argument gives
\begin{equation*}
	\begin{aligned}
		&\mathbb{E}\bigg(\int_0^T|\bar{\kappa}^i_t||\hat{\kappa}^{i,2,\epsilon}_t|\mathbbm{1}_{E_\epsilon}(t)dt\bigg)^{p\bar{q}^2}
		\leqslant C\mathbb{E}\bigg[\Big(\int_0^T|\bar{\kappa}^i_t|^2\mathbbm{1}_{E_{\epsilon}}(t)dt\Big)^{\frac{p\bar{q}^2}{2}}\Big(\int_0^T|\hat{\kappa}^{i,2,\epsilon}_t|^2dt\Big)^{\frac{p\bar{q}^2}{2}}\bigg]\\
		\leqslant&\ C \bigg(\mathbb{E}\bigg(\int_0^T|\bar{\kappa}^i_t|^2\mathbbm{1}_{E_\epsilon}(t)dt\bigg)^{p\bar{q}^2}\bigg)^{\frac{1}{2}}
\bigg(\mathbb{E}\bigg(\int_0^T|\hat{\kappa}^{i,2,\epsilon}_t|^2dt\bigg)^{p\bar{q}^2}\bigg)^{\frac{1}{2}}=o(\epsilon^{p\bar{q}^2}).
	\end{aligned}
\end{equation*}

Noting that, for $\theta>0$, $\mathbb{E}\big[\int_0^T\int_{E_i}|e^{\theta\bar{\tilde{\kappa}}^i_{(t,e)}}|\nu_i(de)dt\big]<\infty$, we have
\begin{equation*}
	\begin{aligned}
		&\mathbb{E}\bigg(\int_0^T\Big[\int_{E_i}|e^{\theta\bar{\tilde{\kappa}}^i_{(t,e)}}|\nu_i(de)\Big]\Big[\int_{E_i}|\hat{\kappa}^{i,2,\epsilon}_t|\nu_i(de)\Big]\mathbbm{1}_{E_\epsilon}(t)dt\bigg)^{p\bar{q}^2}\\
		&\leqslant C\mathbb{E}\bigg[\Big(\int_0^T\int_{E_i}|e^{\theta\bar{\tilde{\kappa}}^i_{(t,e)}}|^2\mathbbm{1}_{E_\epsilon}(t)\nu_i(de)dt\Big)^{\frac{p\bar{q}^2}{2}}
\Big(\int_0^T\int_{E_i}|\hat{\tilde{\kappa}}^{i,2,\epsilon}_{(t,e)}|^2\nu_i(de)dt\Big)^{\frac{p\bar{q}^2}{2}}\bigg]\\
		&\leqslant C\bigg(\mathbb{E}\bigg(\int_0^T\int_{E_i}|e^{\theta\bar{\tilde{\kappa}}^i_{(t,e)}}|^2\mathbbm{1}_{E_\epsilon}(t)\nu_i(de)dt\bigg)^{p\bar{q}^2}\bigg)^{\frac{1}{2}}
\bigg(\mathbb{E}\bigg(\int_0^T\int_{E_i}|\hat{\tilde{\kappa}}^{i,2,\epsilon}_{(t,e)}|^2\nu_i(de)dt\bigg)^{p\bar{q}^2}\bigg)^{\frac{1}{2}}
		=o(\epsilon^{p\bar{q}^2}).
	\end{aligned}
\end{equation*}

Next, for $l_{\kappa\kappa}$ term, we have
\begin{equation*}
	\begin{aligned}
		&\mathbb{E}\bigg(\int_0^T|\hat{\kappa}^{i,2,\epsilon}_t\beta^i_tx^1_t|dt\bigg)^{p\bar{q}^2}\\
		&\leqslant C\bigg(\mathbb{E}\bigg[\sup_{0\leqslant t\leqslant T}|x^1_t|^{3p\bar{q}^2}\bigg]\bigg)^{\frac{1}{3}}
        \bigg(\mathbb{E}\bigg(\int_0^T|\hat{\kappa}^{i,2,\epsilon}_t|^2dt\bigg)^{\frac{3p\bar{q}^2}{2}}\bigg)^{\frac{1}{3}}
        \bigg(\mathbb{E}\bigg(\int_0^T|\beta^i_t|^2dt\bigg)^{\frac{3p\bar{q}^2}{2}}\bigg)^{\frac{1}{3}}
		\leqslant C\epsilon^{\frac{3p\bar{q}^2}{2}},
	\end{aligned}
\end{equation*}
and similarly, $\mathbb{E}\Big(\int_0^T\big(\int_{E_i}|\hat{\tilde{\kappa}}^{i,2,\epsilon}_{(t,e)}|\nu_i(de)\big)\big(\int_{E_i}|\tilde{\beta}^i_{(t,e)}|\nu_i(de)\big)x^1_{t-}dt\Big)^{p\bar{q}^2}\leqslant C\epsilon^{\frac{3p\bar{q}^2}{2}}$.

Other terms are similar. Thus, we get $\mathbb{E}\big(\int_0^T|A^{\epsilon}_2(t)|dt\big)^{p\bar{q}^2}=o(\epsilon^{p\bar{q}^2})$. From \eqref{equation estiamte hat zeta3 1}, we finally deduce that, for $p>1$,
\begin{equation*}
		\mathbb{E}\bigg[\sup_{0\leqslant t\leqslant T}|\hat{\zeta}^{3,\epsilon}_t|^p+\Big(\int_0^T|\hat{\kappa}^{i,3,\epsilon}_t|^2dt\Big)^{\frac{p}{2}}+\Big(\int_0^T\int_{E_i}|\hat{\tilde{\kappa}}^{i,3,\epsilon}_{(t,e)}|^2\nu_i(de)dt\Big)^{\frac{p}{2}}\bigg]=o(\epsilon^p).
\end{equation*}
The proof is complete.
\end{proof}

\end{document}